\documentclass[final]{siamltex}

\def\d{\delta} 
 
\def\e{{\epsilon}}

\def\norm#1{\|#1\|} 
\def\wb{{\bar w}}

\usepackage{algorithm}
\usepackage{algorithmic}
\usepackage{amsmath}
\usepackage{amssymb}
\usepackage{caption}
\usepackage{comment}
\usepackage{graphicx}
\usepackage{float}
\usepackage[latin1]{inputenc}
\usepackage[pdfborder={0 0 0}, colorlinks=true, linkcolor=blue, citecolor=red]{hyperref}\hypersetup{pdfborder=0 0 0}
\usepackage{mathrsfs}
\usepackage{multicol}
\usepackage{multirow}
\usepackage{showkeys,cite}
\usepackage{subfigure}
\usepackage{wrapfig}
\usepackage{xcolor}
\usepackage[textsize=footnotesize]{todonotes}
\usepackage{ulem}
\graphicspath{{./plots/}}

\begin{document}

\title{
       iHDG: An Iterative HDG Framework for Partial
       Differential Equations.\thanks{This research was
    partially supported by DOE grants DE-SC0010518 and
    DE-SC0011118, NSF Grants NSF-DMS1620352 and RNMS (Ki-Net) 1107444, ERC Advanced Grant FP7-246775 NUMERIWAVES and ERC Advanced Grant DYCON : Dynamic Control. We are grateful for the supports.}}

    \author{Sriramkrishnan Muralikrishnan\thanks{Department of Aerospace Engineering and Engineering Mechanics, The University of Texas at Austin, Austin, TX 78712, USA.} \and Minh-Binh Tran\thanks{Department of Mathematics, University of Wisconsin, Madison, WI 53706, USA.} \and Tan Bui-Thanh\thanks{Department of Aerospace Engineering and Engineering Mechanics, and the Institute for Computational Engineering and Sciences, The University of Texas at Austin, Austin, TX 78712, USA.} }

\bibliographystyle{siam}
\newcommand{\TODO}[1]{ \fbox{\parbox{3in}{\bf TODO: #1}}}

\newcommand{\grbf}[1] {\mbox{\boldmath${#1}$\unboldmath}}
\newcommand{\gbf}[1] {\mathbf{#1}}

\newcommand{\beq} {\begin{equation}}
\newcommand{\eeq} {\end{equation}}
\newcommand{\bdm} {\begin{displaymath}}
\newcommand{\edm} {\end{displaymath}}
\newcommand{\bit}{\begin{itemize}}
\newcommand{\eit}{\end{itemize}}
\newcommand{\bde}{\begin{description}}
\newcommand{\ede}{\end{description}}
\newcommand{\bce}{\begin{center}}
\newcommand{\ece}{\end{center}}
\newcommand{\ben} {\begin{enumerate}}
\newcommand{\een} {\end{enumerate}}
\newcommand{\bea} {\begin{eqnarray}}
\newcommand{\eea} {\end{eqnarray}}
\newcommand{\barr} {\begin{array}}
\newcommand{\earr} {\end{array}}
\newcommand{\bean} {\begin{eqnarray*}}
\newcommand{\eean} {\end{eqnarray*}}
\newcommand{\edoc} {\end{document}}
\newcommand{\On}{\hspace{0.2cm} \mbox{on} \hspace{0.2cm}}
\newcommand{\In}{\hspace{0.2cm} \mbox{in} \hspace{0.2cm}}
\newcommand{\Minimize}{\hspace{0.2cm} \mbox{minimize} \hspace{0.2cm}}
\newcommand{\Maximize}{\hspace{0.2cm} \mbox{maximize} \hspace{0.2cm}}
\newcommand{\SubjectTo}{\hspace{0.2cm} \mbox{subject} \hspace{0.15cm}
            \mbox{to} \hspace{0.2cm}}

\def\etal{{\it et al.~}}

\newcommand{\point}[1] {\ensuremath{\boldsymbol{#1}}}
\newcommand{\vvec}[1] {\ensuremath{\boldsymbol{#1}}}
\renewcommand{\vec}[1] {\ensuremath{\boldsymbol{#1}}}
\newcommand{\cvec}[1] {\ensuremath{\boldsymbol{{#1}}}}
\newcommand{\ten}[1] {\ensuremath{\boldsymbol{#1}}}
\newcommand{\tenfour}[1] {\ensuremath{\boldsymbol{\mathsf{#1}}}}
\newcommand{\comp}[1]{\begin{pmatrix}{ #1 }\end{pmatrix}}
\newcommand{\vv}{\ensuremath{\vec{v}}}
\newcommand{\vr}{\left(\ensuremath{\vec{r}}\right)}
\newcommand{\att}{\left(t\right)}
\newcommand{\tvv}{\ensuremath{\tilde{\vec{v}}}}
\newcommand{\ww}{\ensuremath{\vec{w}}}
\newcommand{\GG}{\ensuremath{\ten{G}}}
\newcommand{\FF}{\ensuremath{\boldsymbol{\mathfrak{F}}}}
\newcommand{\tEE}{\ensuremath{\tilde{\vec{E}}}}
\newcommand{\nn}{\ensuremath{\vec{n}}}
\renewcommand{\SS}{\ensuremath{\ten{S}}}
\newcommand{\tSS}{\ensuremath{\tilde{\ten{S}}}}
\newcommand{\cp}{\ensuremath{{c_p}}}
\newcommand{\cs}{\ensuremath{{c_s}}}
\newcommand{\tcs}{\ensuremath{{\tilde{c}_s}}}
\newcommand{\Vp}{\ensuremath{{V^e_p}}}
\newcommand{\Vs}{\ensuremath{{V^e_s}}}
\newcommand{\Ss}{\ensuremath{{S^e_s}}}
\newcommand{\nxnx}[1] {\ensuremath{\nn\times\left(\nn\times{#1}\right)}}

\newcommand{\dd}[2] {\ensuremath{\frac{\partial {#1}}{\partial {#2}}}}
\newcommand{\DD}[2] {\ensuremath{\frac{d {#1}}{d {#2}}}}
\newcommand{\Grad} {\ensuremath{\nabla}}
\newcommand{\Gradv}[1]{\ensuremath{\nabla_{#1}}}
\newcommand{\Div} {\ensuremath{\nabla\cdot}}
\newcommand{\Divv}[1] {\ensuremath{\nabla_{#1}\cdot}}
\newcommand{\Curl} {\ensuremath{\nabla\times}}
\newcommand{\Cten}{\ensuremath{\tenfour{C}}}

\newcommand{\eval}[2][\right]{\relax
  \ifx#1\right\relax \left.\fi#2#1\rvert}

\providecommand{\abs}[1]{\left\lvert#1\right\rvert}
\providecommand{\norm}[1]{\left\lVert#1\right\rVert}

\newcommand{\Nel} {\ensuremath{{N_\text{el}}}}
\newcommand{\Ned} {\ensuremath{{N_\text{ed}}}}
\newcommand{\De} {\ensuremath{{\mathsf{D}^e}}}
\newcommand{\Dep} {\ensuremath{{\mathsf{D}^{e'}}}}
\newcommand{\Dhat} {\ensuremath{\hat{\mathsf{D}}}}
\newcommand{\Dehat} {\ensuremath{{\hat{\mathsf{D}}^e}}}
\newcommand{\half} {\ensuremath{\frac{1}{2}}}
\newcommand{\IN} {\ensuremath{{\mathcal{I}_N}}}
\newcommand{\INe} {\ensuremath{{\mathcal{I}_{N_e}}}}

\newcommand{\mc}[1]{\mathcal{#1}}
\newcommand{\mcC}[1]{\mathcal{#1}}
\newcommand{\mb}[1]{\mathbf{#1}}
\newcommand{\mbb}[1]{\mathbb{#1}}
\newcommand{\mi}[1]{\mathit{#1}}
\newcommand{\mf}[1]{\mathfrak{#1}}
\newcommand{\Oi}[1]{\int_{\mc{D}} #1 \, d\vec{x}}
\newcommand{\TOi}[1]{\int_{0}^T \int_{\mc{D}} #1 \, d\vec{x}dt}
\newcommand{\Ti}[1]{\int_{0}^T #1 \, dt}
\newcommand{\TDei}[1]{\int_{0}^T\int_{\De} #1 \, d\vec{x}dt}
\newcommand{\TDeid}[1]{\int_{0}^T\int_{\De,N} #1 \, d\vec{x}dt}
\newcommand{\Dei}[1]{\int_{\De} #1 \, d\vec{x}}
\newcommand{\DeI}[1]{\int_{D} #1 \, d\vec{x}}
\newcommand{\DeiM}[1]{\int_{\Dhat} #1 \, d\vec{r}}
\newcommand{\DeMd}[1]{\int_{\De,N_e} #1 \, d\vec{x}}
\newcommand{\DeiMd}[1]{\int_{\Dhat,N_e} #1 \, d\vec{r}}
\newcommand{\DeMdN}[1]{\int_{D,N} #1 \, d\vec{x}}
\newcommand{\DeiMdN}[1]{\int_{\Dhat,N} #1 \, d\vec{r}}
\newcommand{\DeiMdNC}[2]{\int_{\Dhat,N_{#2}} #1 \, d\vec{r}}
\newcommand{\pDei}[1]{\int_{\partial \De} #1 \, d\vec{x}}
\newcommand{\pDeiM}[1]{\int_{\partial \Dhat} #1 \, d\vec{r}}
\newcommand{\pDeiMd}[1]{\int_{\partial \Dhat,N_e} #1 \, d\vec{r}}
\newcommand{\pDeiMdN}[1]{\int_{\partial \Dhat,N} #1 \, d\vec{r}}
\newcommand{\pDeiMdNC}[2]{\int_{\partial \Dhat,N_{#2}} #1 \, d\vec{r}}
\newcommand{\pMortar}[2]{\int_{\mc{M}_{#2}} #1 \, d\vec{r}}
\newcommand{\pMortard}[2]{\int_{\mc{M}^{#2},N_{#2}} #1 \, d\vec{r}}
\newcommand{\pMortardd}[3]{\int_{\mc{M}^{#2},N_{#3}} #1 \, d\vec{r}}
\newcommand{\pMortardh}[3]{\int_{\mc{M}_{#2},N_{#3}} #1 \, d\vec{r}}
\newcommand{\pMortardhn}[4]{\int_{{#2}_{#3},N_{#4}} #1 \, d\vec{r}}
\newcommand{\TpDei}[1]{\int_{0}^T\int_{\partial \De} #1 \, d\vec{x}}
\newcommand{\TpDeid}[1]{\int_{0}^T\int_{\partial \De,N} #1 \, d\vec{x}}
\newcommand{\Deid}[1]{\int_{\De,N_e} #1 \, d\vec{x}}
\newcommand{\DeidN}[1]{\int_{\De,N} #1 \, d\vec{x}}
\newcommand{\pDeidN}[1]{\int_{\partial \De,N} #1 \, d\vec{x}}
\newcommand{\pDeid}[1]{\int_{\partial \De,N_e} #1 \, d\vec{x}}
\newcommand{\nor}[1]{\left\| #1 \right\|}
\newcommand{\snor}[1]{\left| #1 \right|}
\newcommand{\LRp}[1]{\left( #1 \right)}
\newcommand{\LRs}[1]{\left[ #1 \right]}
\newcommand{\LRa}[1]{\left< #1 \right>}
\newcommand{\LRc}[1]{\left\{ #1 \right\}}
\newcommand{\pp}[2]{\frac{\partial #1}{\partial #2}}
\newcommand{\ppcp}[1]{\frac{\partial #1}{\partial \vec{c_p}}}
\newcommand{\ppcpj}[1]{\frac{\partial #1}{\partial c_p}}
\newcommand{\ppcpmf}[1]{\frac{\partial #1}{\partial \vec{\mf{c_p}}}}
\newcommand{\ppm}[1]{\frac{\partial #1}{\partial \vec{m}}}
\newcommand{\ppho}[3]{\frac{\partial^{#3} #1}{{\partial #2}^{#3}}}
\newcommand{\ppmi}[1]{\frac{\partial #1}{\partial m_i}}
\newcommand{\ppmj}[1]{\frac{\partial #1}{\partial m_j}}
\newcommand{\ppcpm}[1]{\frac{\partial #1}{\partial \vec{c_p}^-}}
\newcommand{\ppcpp}[1]{\frac{\partial #1}{\partial \vec{c_p}^+}}
\newcommand{\ppcs}[1]{\frac{\partial #1}{\partial \vec{c_s}}}
\newcommand{\ppcsm}[1]{\frac{\partial #1}{\partial \vec{c_s}^-}}
\newcommand{\ppcsp}[1]{\frac{\partial #1}{\partial \vec{c_s}^+}}
\newcommand{\td}[2]{\frac{d #1}{d #2}}
\newcommand{\Rvert}[1]{\left. #1 \right|}
\newcommand{\bs}[1]{\boldsymbol{#1}}
\newcommand{\ipDed}[3]{\LRp{ #1, #2}_{\De,N}^{#3}}

\newcommand{\bigO}{\mathcal{O}}

\newcommand{\iOm}[1]{\int_{\Omega} #1 \, d\Omega}
\newcommand{\iGb}[1]{\int_{\Gamma_{b}} #1 \, ds}
\newcommand{\iGs}[1]{\int_{\Gamma_{s}} #1 \, ds}
\newcommand{\iGsy}[1]{\int_{\Gamma_{s}} #1 \, ds(\mb{y})}
\newcommand{\RE}{\mbox{{I}}\hspace{-0.5ex}\mbox{{R}}}
\newcommand{\iC}[1]{\int_{0}^{2\pi} #1 \, d\theta}
\newcommand{\iGr}[1]{\int_{\Gamma_{\infty}} #1 \, ds}

\newcommand{\jump}[1] {\ensuremath{[\![{#1}]\!]}}
\newcommand{\diff}[1] {\ensuremath{\LRs{#1}}}
\newcommand{\average}[1] {\ensuremath{\LRc{\!\!\LRc{#1}\!\!}}}

\newcounter{tablecell}

\newcommand*{\numbercell}{%
  \stepcounter{tablecell}%
  \thetablecell. 
}

\newtheorem{propo}{Proposition}
\newtheorem{coro}{Corollary}
\newtheorem{theo}{Theorem}
\newtheorem{lem}{Lemma}

\newtheorem{defi}{definition}

\newtheorem{rema}{remark}

\newcommand{\figlab}[1]{\label{fig:#1}}
\newcommand{\eqnlab}[1]{\label{eq:#1}}
\newcommand{\theolab}[1]{\label{theo:#1}}
\newcommand{\corolab}[1]{\label{coro:#1}}
\newcommand{\propolab}[1]{\label{propo:#1}}
\newcommand{\lemlab}[1]{\label{lem:#1}}
\newcommand{\defilab}[1]{\label{defi:#1}}
\newcommand{\remalab}[1]{\label{rema:#1}}
\newcommand{\tablab}[1]{\label{tab:#1}}

\newcommand{\figref}[1]{\ref{fig:#1}}
\newcommand{\theoref}[1]{\ref{theo:#1}}
\newcommand{\defiref}[1]{\ref{defi:#1}}
\newcommand{\remaref}[1]{\ref{rema:#1}}
\newcommand{\cororef}[1]{\ref{coro:#1}}
\newcommand{\proporef}[1]{\ref{propo:#1}}
\newcommand{\lemref}[1]{\ref{lem:#1}}
\newcommand{\eqnref}[1]{\eqref{eq:#1}}
\newcommand{\alglab}[1]{\label{alg:#1}}
\newcommand{\algref}[1]{\ref{alg:#1}}
\newcommand{\seclab}[1]{\label{sect:#1}}
\newcommand{\secref}[1]{\ref{sect:#1}}
\newcommand{\tabref}[1]{\ref{tab:#1}}

\renewcommand{\u}{u}
\newcommand{\uk}{\u^k}
\newcommand{\upk}{\u^{+k}}
\newcommand{\ukp}{\u^{k+1}}
\newcommand{\umkp}{\u^{-k+1}}
\newcommand{\umk}{\u^{-k}}
\renewcommand{\v}{v}
\newcommand{\vk}{{\v}^k}
\newcommand{\vkp}{{\v}^{k+1}}
\newcommand{\un}{\u_h}
\renewcommand{\e}{e}
\newcommand{\w}{w}
\renewcommand{\wb}{\mb{\w}}
\newcommand{\J}{J}
\newcommand{\Jb}{\mb{J}}
\newcommand{\q}{q}
\renewcommand{\r}{r}
\newcommand{\qh}{\hat{\q}}
\newcommand{\qs}{\q^*}
\newcommand{\qht}{\tilde{\qh}}
\newcommand{\qbar}{\overline{\q}}
\newcommand{\qb}{{\mb{\q}}}
\newcommand{\sig}{{\sigma}}
\newcommand{\thetah}{{\hat{\theta}}}
\newcommand{\thetas}{{\theta^*}}
\newcommand{\thetahn}{{\thetah_h}}
\newcommand{\sign}[1]{\text{ sgn}\!\LRp{#1}}
\newcommand{\sigh}{\hat{\sig}}
\newcommand{\sigs}{\sig^*}
\newcommand{\sigk}{\sig^{k}}
\newcommand{\sigpk}{\sig^{+k}}
\newcommand{\sigkp}{\sig^{k+1}}
\newcommand{\sigmkp}{\sig^{-k+1}}
\newcommand{\sigmk}{\sig^{-k}}
\newcommand{\sigb}{{\bs{\sig}}}
\newcommand{\sigbk}{\sigb^k}
\newcommand{\sigbkp}{\sigb^{k+1}}
\newcommand{\sigbn}{\sigb_h}
\newcommand{\sigbs}{\sigb^*}
\newcommand{\taub}{{\bs{\tau}}}
\newcommand{\taustar}{\tau_*}
\newcommand{\kappab}{{\bs{\kappa}}}
\newcommand{\gamb}{{\bs{\gamma}}}
\newcommand{\ut}{\tilde{\u}}
\newcommand{\sigbt}{\tilde{\sigb}}
\newcommand{\vt}{\tilde{\v}}
\newcommand{\taubt}{\tilde{\taub}}
\newcommand{\ub}{\mb{\u}}
\newcommand{\ubp}{\mb{\u}^+}
\newcommand{\ubm}{\mb{\u}^-}
\newcommand{\vb}{\mb{\v}}
\newcommand{\um}{\u^-}
\newcommand{\up}{\u^+}
\newcommand{\ubk}{\ub^k}
\newcommand{\ubkp}{\ub^{k+1}}

\newcommand{\m}{{m}}

\newcommand{\ud}{{\dot{\u}}}
\newcommand{\vd}{{\dot{\v}}}
\newcommand{\sigbd}{{\dot{\sigb}}}
\newcommand{\taubd}{\dot{\taub}}

\renewcommand{\d}{{d}}
\newcommand{\R}{{\mathbb{R}}}
\newcommand{\kap}{\kappa}
\newcommand{\F}{\mb{F}}
\newcommand{\f}{f}
\newcommand{\g}{g}
\newcommand{\fb}{\mb{\f}}
\newcommand{\gb}{\mb{\g}}
\newcommand{\gD}{g_D}
\newcommand{\pOmega}{\partial \Omega}
\newcommand{\K}{K}
\newcommand{\Kp}{K^+}
\newcommand{\Km}{K^-}
\newcommand{\pK}{{\partial \K}}
\newcommand{\pKin}{{\pK^{\text{in}}}}
\newcommand{\pKout}{{\pK^{\text{out}}}}
\newcommand{\Kj}{{\K_j}}
\newcommand{\Gh}{{\mc{E}_h}}
\newcommand{\Gho}{{\mc{E}_h^o}}
\newcommand{\Ghb}{{\mc{E}_h^{\partial}}}
\newcommand{\Poly}{{\mc{P}}}
\newcommand{\D}{{D}}
\newcommand{\p}{{p}}
\newcommand{\Ts}{{\mb{T}}}
\newcommand{\Vb}{{\mb{V}}}
\newcommand{\V}{{V}}
\newcommand{\Lam}{{\Lambda}}
\newcommand{\Lamb}{\bs{\Lambda}}
\newcommand{\mub}{\bs{\mu}}
\newcommand{\lambdab}{\bs{\lambda}}
\newcommand{\Th}{{\Ts_h}}
\newcommand{\Vh}{{\V_h}}
\newcommand{\Vbh}{{\Vb_h}}
\newcommand{\Lamh}{{\Lam_h}}
\newcommand{\Lambh}{{\Lamb_h}}
\newcommand{\Lamho}{{\overline{\Lam}_h}}
\newcommand{\Lambht}{{\Lamb_h^t}}
\newcommand{\Ltwo}{{L^2\LRp{\Omega}}}
\newcommand{\Ltwoe}{{L^2\LRp{\e}}}
\newcommand{\Ltw}{{L^2\LRp{\K}}}
\newcommand{\VhK}{{\V_h\LRp{\K}}}
\newcommand{\ThK}{{\Ts_h\LRp{\K}}}
\newcommand{\VbhK}{{\Vb_h\LRp{\K}}}
\newcommand{\Lamhe}{{\Lam_h\LRp{\e}}}
\newcommand{\Lambhe}{{\Lamb_h\LRp{\e}}}
\renewcommand{\L}{\mc{L}}
\renewcommand{\D}{\mc{D}}
\newcommand{\T}{\mc{T}}
\newcommand{\Tt}{\tilde{\T}}
\newcommand{\A}{\mb{A}}
\newcommand{\B}{\mb{B}}
\renewcommand{\D}{\mb{D}}
\newcommand{\Am}{\A^-}
\newcommand{\Ap}{\A^+}
\newcommand{\Aa}{\snor{\A}}
\newcommand{\Rb}{\mb{R}}
\newcommand{\C}{\mb{C}}
\renewcommand{\S}{\mb{S}}
\newcommand{\Sm}{\S^{-}}
\newcommand{\Sp}{\S^{+}}
\newcommand{\Spm}{\S^{\pm}}
\newcommand{\Q}{\mb{Q}}

\newcommand{\Lte}{{L^2\LRp{\Gh}}}
\newcommand{\n}{\mb{n}}
\newcommand{\nm}{\n^-}
\newcommand{\np}{\n^+}
\newcommand{\uh}{{\hat{\u}}}
\newcommand{\uhk}{{\uh}^k}
\newcommand{\uhkp}{{\uh}^{k+1}}
\newcommand{\us}{{\u^*}}
\newcommand{\uhn}{{\uh_h}}
\newcommand{\ubh}{{\hat{\ub}}}
\newcommand{\ubs}{{\ub^*}}
\newcommand{\ubht}{{\tilde{\ubh}}}
\newcommand{\uht}{{\tilde{\uh}}}
\newcommand{\uhtn}{{\uht_h}}
\newcommand{\uhd}{{\dot{\uh}}}
\newcommand{\ubhk}{\ubh^k}
\newcommand{\ubhkp}{\ubh^{k+1}}
\newcommand{\sigbh}{{\hat{\sigb}}}
\newcommand{\Fh}{{\hat{\F}}}
\newcommand{\Fs}{\F^*}
\renewcommand{\c}{{c}}

\newcommand{\zb}{\mb{z}}
\renewcommand{\sb}{\mb{s}}
\newcommand{\rb}{\mb{r}}
\newcommand{\betab}{\bs{\beta}}
\newcommand{\veps}{{\varepsilon}}
\newcommand{\vepsb}{\bs{\veps}}
\newcommand{\Hb}{\mb{H}}
\newcommand{\Hbt}{\Hb^t}
\newcommand{\Eb}{\mb{E}}
\newcommand{\Ebt}{\Eb^t}
\newcommand{\hb}{\mb{h}}
\newcommand{\eb}{\mb{e}}
\newcommand{\Hbh}{\hat{\mb{H}}}
\newcommand{\Ebh}{\hat{\mb{E}}}
\newcommand{\Hbs}{{\mb{H}^*}}
\newcommand{\Ebs}{{\mb{E}^*}}
\newcommand{\Ebht}{\Ebh^t}
\newcommand{\Hbht}{\Hbh^t}
\newcommand{\Lb}{\mb{L}}
\newcommand{\Gb}{\mb{G}}
\newcommand{\Lbh}{\hat{\Lb}}
\newcommand{\Lbs}{\Lb^*}
\newcommand{\Ib}{\mb{I}}
\newcommand{\Sigb}{\bs{\Sigma}}
\newcommand{\Piu}{\Pi_\u}
\newcommand{\Pisigb}{\Pi_\sigb}
\newcommand{\Z}{Z}
\newcommand{\Y}{Y}
\newcommand{\rhou}{\rho\u}
\newcommand{\rhov}{\rho\v}
\newcommand{\E}{E}
\renewcommand{\P}{P}
\newcommand{\h}{H}

\newcommand{\mygreen}[1]{{\color[rgb]{0,.65,0} #1}}
\newcommand{\mywhite}[1]{{\color[rgb]{1.0,1.0,1.0} #1}}
\newcommand{\myred}[1]{{\color[rgb]{0.65,0.0,0.0} #1}}
\newcommand{\myblue}[1]{{\color[rgb]{0.0,0.0,1.0} #1}}
\newcommand{\mygray}[1]{{\color[rgb]{.15,.35,.60} #1}}
\newcommand{\mythemec}[1]{{\color[RGB]{51,108,121} #1}}
\newcommand{\mydarkred}[1]{{\color[rgb]{.6,.1,.1} #1}}
\newcommand{\mydarkblue}[1]{{\color[rgb]{.1,.1,.9} #1}}
\newcommand{\mygreenback}[1]{{\color[rgb]{.75,1,.75} #1}}
\newcommand{\myorange}[1]{{\color[rgb]{.76,.39,.13} #1}}
\newcommand{\mygrass}[1]{{\color[rgb]{.19,.64,.13} #1}}
\newcommand{\mysierp}[1]{{\color[RGB]{209,28,209} #1}}

\newcommand{\tanbui}[2]{\textcolor{blue}{#1} \textcolor{red}{#2}}

\newcommand{\vel}{\bs{\vartheta}}
\newcommand{\vels}{\vel^*}
\newcommand{\vele}{\vel^e}
\newcommand{\velh}{\hat{\vel}}
\newcommand{\velk}{\vel^k}
\newcommand{\velkp}{\vel^{k+1}}
\newcommand{\vhk}{\hat\v^{k}}

\newcommand{\phis}{\phi^*}
\newcommand{\phie}{\phi^e}
\newcommand{\phih}{\hat{\phi}}
\newcommand{\phihk}{\phih^k}
\newcommand{\phik}{\phi^k}
\newcommand{\phikp}{\phi^{k+1}}

\maketitle

\begin{abstract}
We present a scalable iterative solver for high-order hybridized
discontinuous Galerkin (HDG) discretizations of linear partial
differential equations. It is an interplay between domain
decomposition methods and HDG discretizations, and hence inheriting
advances from both sides. In particular, the method can be viewed as a
Gauss-Seidel approach that requires only independent
element-by-element and face-by-face local solves in each iteration. As
such, it is well-suited for current and future computing systems with
massive concurrencies.  Unlike conventional Gauss-Seidel schemes which
are purely algebraic, the convergence of iHDG, thanks to the built-in
HDG numerical flux, does not depend on the ordering of unknowns.  We
rigorously show the convergence of the proposed method for the
transport equation, the linearized shallow water equation and the
convection-diffusion equation. For the transport equation, the method
is convergent regardless of mesh size $h$ and solution order $p$, and
furthermore the convergence rate is independent of the solution
order. For the linearized shallow water and the convection-diffusion
equations we show that the convergence is conditional on both $h$ and
$p$. Extensive steady and time-dependent numerical results for the 2D and
3D transport equations, the linearized shallow water equation, and
the convection-diffusion equation are presented to verify the theoretical
findings.

\end{abstract}



\pagestyle{myheadings} \thispagestyle{plain}
\markboth{S. MURALIKRISHNAN, M-B. TRAN, AND T. BUI-THANH}{AN ITERATIVE HDG FRAMEWORK}

\section{Introduction}

The discontinuous Galerkin (DG) method was originally developed by
Reed and Hill \cite{ReedHill73} for the neutron transport equation,
first analyzed in \cite{LeSaintRaviart74, JohnsonPitkaranta86}, and
then has been extended to other problems governed by partial
differential equations (PDEs) \cite{CockburnKarniadakisShu00}. Roughly
speaking, DG combines advantages of classical finite volume and finite
element methods. In particular, it has the ability to treat solutions
with large gradients including shocks, it provides the flexibility to
deal with complex geometries, and it is highly parallelizable due to
its compact stencil.
However, for steady state problems or
time-dependent ones that require implicit time-integrators, DG methods
typically have many more (coupled) unknowns compared to the other
existing numerical methods, and hence more expensive in
general.

In order to mitigate the computational expense associated with DG
methods, Cockburn, coauthors, and others have introduced hybridizable
(also known as hybridized) discontinuous Galerkin (HDG) methods for
various types of PDEs including Poisson-type equation
\cite{CockburnGopalakrishnanLazarov09, CockburnGopalakrishnanSayas10,
  KirbySherwinCockburn12, NguyenPeraireCockburn09a,
  CockburnDongGuzmanEtAl09, EggerSchoberl10}, Stokes equation
\cite{CockburnGopalakrishnan09, NguyenPeraireCockburn10}, Euler and
Navier-Stokes equations, wave equations \cite{NguyenPeraireCockburn11,
  MoroNguyenPeraire11, NguyenPeraireCockburn11b,
  LiLanteriPerrrussel13, NguyenPeraireCockburn11a, GriesmaierMonk11,
  CuiZhang14}, to name a few. The upwind HDG framework proposed in
\cite{Bui-Thanh15, Bui-Thanh15a, bui2016construction} provides a unified and
a systematic construction of HDG methods for a large class of PDEs. In
HDG discretizations, the coupled unknowns are single-valued traces
introduced on the mesh skeleton, i.e. the faces, and the resulting
matrix is substantially smaller and sparser compared to standard DG
approaches. Once they are solved for, the usual DG unknowns can be
recovered in an element-by-element fashion, completely independent of
each other.  Nevertheless, the trace system is still a bottleneck for
practically large-scale applications, where complex and high-fidelity
simulations involving features with a large range of spatial and
temporal scales are necessary.

Meanwhile, Schwarz-type domain decomposition methods (DDMs) have been
introduced as procedures to parallelize and solve partial differential
equations efficiently, in which each iteration involves the solutions
of the original equation on smaller subdomains
\cite{Lions:1987:OSA,Lions:1989:OSA,Lions:1990:OSA}. Schwarz waveform relaxation methods and optimized Schwarz
methods
\cite{Halpern:OSM:2009,BennequinGanderHalpern:AHB:2009,HalpernSzeftel:2009:NOS,GanderGouarinHalpern:2011:OSW,GanderHajian:2015:ASM,Binh1,Binh2}
have attracted substantial attention over the past decades since they
can be adapted to the underlying physics, and thus lead to efficient
parallel solvers for challenging problems. We view the iterative HDG (iHDG) method as
an extreme DDM approach in which each subdomain is an element. For
current DDM methods, decomposing the computational domain into smaller
subdomains encounters difficulties {when the subdomains have
  cross-points \cite{GanderKwok:2012:BRP} or irregular shapes
  \cite{GanderHalpern:2013:NSR}}; moreover, the geometry of the
decomposition also has a profound influence on the method
\cite{Gander:2011:OTI}. We design the iHDG method to have a two-tier
parallelism to adapt to current and future computing systems: a
coarse-grained parallelism on subdomains, and a fine-grained
parallelism on elements level within each subdomain. Unlike
existing approaches, our method does not rely on any specific
partition of the computational domain,
 and hence is
independent of the geometry of the decomposition.

While either HDG community or DDM community can contribute
individually towards advancing its own field, the potential for
breakthroughs may lie in bringing together the advances from both
sides and in exploiting opportunities at their interfaces. In \cite{GanderHajian:2015:ASM}, Schwarz methods for the
hybridizable interior penalty method have been introduced for
elliptic PDEs in the second order form. The methods are proposed for
two or multi sub-domains, entirely at an algebraic level and correspond to the class of block Jacobi 
methods on the trace system.   

In this paper, we blend the HDG method and Schwarz idea to produce
efficient and scalable iterative solvers for HDG discretizations. One
of the main features of the proposed approach is that it is provably
convergent. From a linear algebra point of view, the method can be
understood as a block Gauss-Seidel iterative solver for the augmented
HDG system with volume and trace unknowns.  Usually the HDG system is
not realized from this point of view and the linear system is
assembled for trace unknowns only. But for iHDG, since we never form
any global matrices it allows us to create an efficient solver which
completely depends on independent element-by-element calculations.
Traditional Gauss-Seidel schemes for convection-diffusion problems or
pure advection problems require the unknowns to be ordered in the flow
direction for convergence. Several ordering schemes for these kind of
problems have been developed and studied in
\cite{bank1981comparison,hackbusch1997downwind,kim2004uniformly,wang1999crosswind}.
In the context of discontinuous Galerkin methods robust Gauss-Seidel
smoothers are developed in \cite{kanschat2008robust} and again these
smoothers depend on the ordering of the unknowns. For
a complex velocity field (e.g. hyperbolic systems) it is, however, not trivial
to obtain a mesh and an ordering which coincide with the flow
direction. Moreover the point or the block Gauss-Seidel scheme (for the trace system alone) 
requires a lot of communication between processors for calculations within an
iteration. These aspects affect the scalabilty of these schemes to a
large extent and in general are not favorable for parallelization
\cite{atkins2005numerical,fidkowski2014algebraic}.

Unlike traditional Gauss-Seidel methods, which are purely algebraic,
the iHDG approach is built upon, and hence exploiting, the HDG
discretization. Of importance is the upwind flux, or more specifically
the upwind stabilization, that automatically determines the flow
directions. 
Consequently {\em the
  convergence of iHDG is independent of the ordering of the unknowns}.
Another crucial property inheritted from HDG is that each iteration
consists of only independent element-by-element local solves to
compute the volume unknowns.  Thanks to the compact stencil of HDG,
this is overlapped by a single communication of the trace of the
volume unknowns restricted on faces shared between the neighboring
processors. 
The communication requirement is thus similar to that of block Jacobi
methods (for the volume system alone). The iHDG approach is designed
with these properties to suit the current and future computing systems
with massive concurrencies.  We rigorously show that our proposed
methods are convergent with explicit contraction constants using an
energy approach. Furthermore the convergence rate is independent of
the solution order for hyperbolic PDEs. The theoretical findings will
be verified on various 2D and 3D numerical results for steady and
time-dependent problems.

The structure of this paper is organized as follows. Section \secref{iHDG}
introduces the iHDG algorithm for an abstract system of PDEs
discretized by the upwind HDG discretization \cite{Bui-Thanh15}. The
convergence of the iHDG algorithm for the scalar and for the system of hyperbolic
PDEs is proved in section \secref{iHDG-hyperbolic} using an energy
approach.  In section \secref{iHDG-convection-diffusion} the
convection-diffusion PDE is considered in the first order
form and the conditions for the convergence of the iHDG algorithm are
stated and proved. Section \secref{numerics} presents various steady
and time dependent examples, in both two and three spatial dimensions,
to support the theoretical findings. We finally conclude the paper in
section \secref{conclusions} and discuss future research directions.

\section{The idea of iHDG}
\seclab{iHDG} 
In this section we first briefly review the upwind HDG framework for a general system of linear PDEs, and then present the main idea behind the iHDG approach. To begin, let us consider the following system
\begin{equation}
\eqnlab{LPDE}
\sum_{k=1}^{\d}\partial_k\F_k\LRp{\ub} + \C\ub :=
\sum_{k=1}^{\d}\partial_k \LRp{\A_k\ub} + \C \ub = \fb, \quad \text{ in } \Omega,
\end{equation}
where $\d$ is the spatial dimension (which, for the clarity of the
exposition, is assumed to be $\d = 3$ whenever a particular value of
the dimension is of concern, but the result is also valid for $d =
\LRc{1,2}$), $\F_k$ the $k$th component of the flux vector (or tensor)
$\F$, $\ub$ the unknown solution with values in $\R^m$, and $\fb$ the
forcing term. The matrices $\A_k$ and $\C$ are assumed to be
continuous\footnote{This assumption is not a limitation but for the
  simplicity of the exposition.} across $\Omega$. The notation
$\partial_k$ stands for the $k$-th partial derivative and by
the subscript $k$ we denote the $k$th component of a vector/tensor. We will discretize
\eqnref{LPDE} using the HDG framework. To that end, let us introduce
notations and conventions used in the paper.

Let us partition $\Omega \in \R^\d$, an  open and bounded domain,  into
$\Nel$ non-overlapping elements $\Kj, j = 1, \hdots, \Nel$ with
Lipschitz boundaries such that $\Omega_h := \cup_{j=1}^\Nel \Kj$ and
$\overline{\Omega} = \overline{\Omega}_h$. Here, $h$ is defined as $h
:= \max_{j\in \LRc{1,\hdots,\Nel}}diam\LRp{\Kj}$. We denote the
skeleton of the mesh by $\Gh := \cup_{j=1}^\Nel \partial K_j$,
the set of all (uniquely defined) faces $\e$. We conventionally identify $\nm$ as the outward 
normal vector on the boundary $\pK$ of element $\K$ (also denoted as $\Km$) and $\np = -\nm$ as the outward normal vector of the boundary of a neighboring element (also denoted as $\Kp$). Furthermore, we use $\n$ to denote either $\nm$ or $\np$ in an expression that is valid for both cases, and this convention is also used for other quantities (restricted) on
a face $\e \in \Gh$.
   For the sake of convenience, we denote by $\Ghb$ the sets of all boundary faces on $\pOmega$, by $\Gho := \Gh \setminus \Ghb$ the set of all interior faces, and $\pOmega_h := \LRc{\pK:\K \in \Omega_h}$.

For simplicity in writing we define $\LRp{\cdot,\cdot}_\K$ as the
$L^2$-inner product on a domain $\K \in \R^\d$ and
$\LRa{\cdot,\cdot}_\K$ as the $L^2$-inner product on a domain $\K$ if
$\K \in \R^{\d-1}$. We shall use $\nor{\cdot}_{\K} :=
\nor{\cdot}_{\Ltw}$ as the induced norm for both cases and the
particular value of $\K$ in a context will indicate which inner
product the norm is coming from. We also denote the $\veps$-weighted
norm of a function $\u$ as $\nor{\u}_{\veps, \K} :=
\nor{\sqrt{\veps}\u}_{\K}$ for any positive $\veps$.  We shall use
boldface lowercase letters for vector-valued functions and in that
case the inner product is defined as $\LRp{\ub,\vb}_\K :=
\sum_{i=1}^m\LRp{\ub_i,\vb_i}_\K$, and similarly $\LRa{\ub,\vb}_\K :=
\sum_{i = 1}^m\LRa{\ub_i,\vb_i}_\K$, where $\m$ is the number of
components ($\ub_i, i=1,\hdots,\m$) of $\ub$.  Moreover, we define
$\LRp{\ub,\vb}_\Omega := \sum_{\K\in \Omega_h}\LRp{\ub,\vb}_\K$ and
$\LRa{\ub,\vb}_\Gh := \sum_{\e\in \Gh}\LRa{\ub,\vb}_\e$ whose
induced (weighted) norms are clear, and hence their definitions are
omitted. We  employ boldface uppercase letters, e.g. $\mb{L}$, to
denote matrices and tensors.
We conventionally use $\ub$ ($\vb$ and $\ubh)$ for
the numerical solution and $\ub^e$ for the exact solution.

We define $\Poly^\p\LRp{\K}$ as the space of polynomials of degree at
most $\p$ on a domain $\K$. Next, we introduce two discontinuous
piecewise polynomial spaces
\begin{align*}
\Vbh\LRp{\Omega_h} &:= \LRc{\vb \in \LRs{L^2\LRp{\Omega_h}}^m:
  \eval{\vb}_{\K} \in \LRs{\Poly^\p\LRp{\K}}^m, \forall \K \in \Omega_h}, \\
\Lambh\LRp{\Gh} &:= \LRc{\lambdab \in \LRs{\Lte}^m:
  \eval{\lambdab}_{\e} \in \LRs{\Poly^\p\LRp{\e}}^m, \forall \e \in \Gh},
\end{align*}
and similar spaces for $\VbhK$ and $\Lambhe$ by replacing $\Omega_h$ with
$\K$ and $\Gh$ with $\e$. For scalar-valued functions,
we denote the corresponding spaces as
\begin{align*}
\Vh\LRp{\Omega_h} &:= \LRc{\v \in L^2\LRp{\Omega_h}:
  \eval{\v}_{\K} \in \Poly^\p\LRp{\K}, \forall \K \in \Omega_h}, \\
\Lamh\LRp{\Gh} &:= \LRc{\lambda \in \Lte:
  \eval{\lambda}_{\e} \in \Poly^\p\LRp{\e}, \forall \e \in \Gh}.
\end{align*}

Following \cite{Bui-Thanh15}, we introduce an upwind HDG discretization for \eqnref{LPDE} as:
for each element $\K$, the
DG local unknown $\ub$ and the extra ``trace'' unknown $\ubh$ need to
satisfy
\begin{subequations}
\eqnlab{KHDG}
\begin{align}
\eqnlab{HDGlocal}
&-\LRp{\F\LRp{\ub}, \Grad\vb}_\K + \LRa{\Fh\LRp{\ub,\ubh}\cdot
\n,\vb}_\pK +\LRp{\C\ub,\vb}_\K = \LRp{\fb,\vb}_\K, \quad \forall \vb \in \VbhK, \\
\eqnlab{HDGeqn}
&\LRa{\jump{\Fh\LRp{\ub,\ubh} \cdot \n},\mub}_\e = \mb{0}, \quad \forall \e \in \Gho, \quad \forall \mub \in \Lambh\LRp{\e}, 
\end{align}
\end{subequations}
 where we have defined the ``jump'' operator for any quantity
 $\LRp{\cdot}$ as $\jump{\LRp{\cdot}} := \LRp{\cdot}^- +
 \LRp{\cdot}^+$. We also define the ``average'' operator
 $\average{\LRp{\cdot}}$ via $2\average{\LRp{\cdot}} :=
 \jump{\LRp{\cdot}}$.  The HDG flux is defined by
\begin{equation}
\eqnlab{RiemannFluxBoth}
\Fh\cdot \n = \F\LRp{\ub}\cdot \n + \Aa \LRp{\ub - \ubh},
\end{equation}
with\footnote{We assume that $\A$ admits an eigen-decomposition, and
  this is valid for a large class of PDEs of Friedrichs' type, for
  example.} the matrix $\A := \sum_{k = 1}^\d{\A_k\n_k} = \Rb \S \Rb^{-1}$, and
  $\Aa := \Rb \snor{\S}\Rb^{-1}$. Here $\n_k$ is the $k$th component of the outward normal vector 
  $\n$ and $\snor{\S}$ represents a matrix obtained by taking the absolute 
  value of the main diagonal of the matrix $\S$. 
  
  The key idea behind the iHDG approach is the following.
The approximation of the HDG solution 
 at the $\LRp{k+1}$th iteration is governed by the local equation \eqnref{HDGlocal} as
\begin{subequations}
\eqnlab{iHDG}
\begin{align}
\eqnlab{iHDGlocal}
&-\LRp{\F\LRp{\ubkp}, \Grad\vb}_\K + \LRa{\F\LRp{\ubkp}\cdot \n + \Aa\ubkp - \Aa\ubhk,\vb}_\pK + \LRp{\C\ub,\vb}_\K = \LRp{\fb,\vb}_\K,
\intertext{where the weighted trace $\Aa\ubhk$ (not the trace itself) is computed using information at the $k$-iteration via the conservation condition \eqnref{HDGeqn}, i.e.,}
\eqnlab{iHDGtrace}
&\LRa{\Aa\ubhk,\mub}_\pK = \LRa{\average{\Aa\ubk} + \average{\F\LRp{\ubk}\cdot \n}, \mub}_\pK.
\end{align}
\end{subequations}
Algorithm \ref{al:DDMhyperbolic} summarizes the iHDG approach.
\begin{algorithm}
  \begin{algorithmic}[1]
    \ENSURE Given initial guess $\ub^0$, compute the weighted
    trace $\Aa\ubh^0$ using \eqnref{iHDGtrace}.
      \WHILE{not converged} 
      \STATE Solve the local equation \eqnref{iHDGlocal} for $\ubkp$ using the weighted trace $\Aa\ubhk$.
      \STATE Compute $\Aa\ubhkp$ using \eqnref{iHDGtrace}.
      \STATE Check convergence. If yes, {\bf exit}, otherwise {\bf set} $k = k+ 1$ and {\bf continue}.
    \ENDWHILE
  \end{algorithmic}
  \caption{The iHDG approach.}
  \label{al:DDMhyperbolic}
\end{algorithm}

The appealing feature of iHDG algorithm \ref{al:DDMhyperbolic} is that
each iteration requires only {\em independent local solve
\eqnref{iHDGlocal} element-by-element, completely independent of each
other}. The method exploits the structure of HDG in which each local
solve is well-defined as long as the trace $\ubhk$ is
given. Furthermore, the global solve via the conservation condition
\eqnref{HDGeqn} is not needed. Instead, we compute the weighted trace $\Aa\ubhk$
{\em face-by-face (on the mesh skeleton) in parallel, completely
independent of each other}. The iHDG approach is therefore well-suited
for parallel computing systems. It can be viewed as a fixed-point
iterative solver by alternating the computation of the local solver
\eqnref{HDGlocal} and conservation condition \eqnref{HDGeqn}. It can
be also understood as a block Gauss-Seidel approach for the linear system with 
volume and weighted trace unknowns. However, unlike matrix-based iterative 
schemes \cite{Greenbaum97, Saad03}, the proposed iHDG method
arises from the structure of HDG methods. As such it's convergence does not depend upon the ordering
of unknowns as the stabilization (i.e. the weighting matrix $\Aa$) automatically takes care of the direction.
For that reason, we term it as {\em iterative HDG discretization} (iHDG). Unlike the original
HDG discretization, it promotes fine-grained parallelism in the
conservation constraints. What remains is to show that iHDG converges
 as the iteration $k$ increases, and this is the focus of
Sections \secref{iHDG-hyperbolic}--\secref{iHDG-convection-diffusion}.

\section{iHDG methods for hyperbolic PDEs}
\seclab{iHDG-hyperbolic}
In this section, we present iHDG methods for scalar and system of
hyperbolic PDEs. For the clarity of the exposition, we consider the
transport equation and a linearized shallow water system, and
the extension of the proposed approach to other hyperbolic PDEs is straightforward. 
To begin, let us consider the transport equation
\begin{subequations}
\eqnlab{transport}
\begin{align}
\betab \cdot \Grad \u^e &= f \quad \text{ in }
\Omega, \\
\u^e &= \g \quad \text{ on } \pOmega^-,
\end{align}
\end{subequations}
where $\pOmega^-$ is the inflow part of the boundary $\pOmega$, and
again $\u^e$ denotes the exact solution. {\it Note that $\betab$ is
  assumed to be a continuous function across the mesh skeleton}. An upwind HDG
discretization \cite{Bui-Thanh15} for \eqnref{transport} consists of
the following local equation for each element $\K$
\begin{equation}
\eqnlab{transportLocal}
-\LRp{\u,\Div\LRp{\betab \v}}_\K  + \LRa{\betab\cdot\n\u + \snor{\betab\cdot \n}\LRp{\u -\uh},\v}_\pK= \LRp{f,\v}_\K, \quad \forall \v \in \Vh\LRp{\K},
\end{equation}
and conservation conditions on all edges $\e$ in the mesh skeleton $\Gh$:
\begin{equation}
\eqnlab{transportGlobal}
\LRa{\jump{\betab\cdot\n\u + \snor{\betab\cdot \n}\LRp{\u -\uh}},\mu}_\e = 0, \quad \forall \mu \in \Lamh\LRp{\e}.
\end{equation}

Solving \eqnref{transportGlobal} for $\snor{\betab\cdot \n}\uh$ we get
\begin{equation}
\eqnlab{tracehyperbolic}
\snor{\betab\cdot \n}\uh = \average{\betab\cdot \n \u} + \snor{\betab\cdot \n}\average{\u}.
\end{equation}

Applying the iHDG algorithm \ref{al:DDMhyperbolic} to the upwind HDG
method \eqnref{transportLocal}--\eqnref{transportGlobal} we obtain the
approximate solution $\u^{k+1}$ at the $(k+1)$th iteration restricted
on each element $\K$ via the following independent local solve: $\forall \v \in \Vh\LRp{\K}$,
\begin{equation}
\eqnlab{transportLocalk1}
- \LRp{\ukp,\Div\LRp{\betab \v}}_\K  + \LRa{\betab\cdot\n\ukp + \snor{\betab\cdot \n}\LRp{\ukp -\uhk},\v}_\pK= \LRp{f,\v}_\K,
\end{equation}
where the trace weighted trace  $\snor{\betab\cdot \n}\uhk$ is computed using information from the previous iteration as
\begin{equation}
\eqnlab{transportLocale}
\snor{\betab\cdot \n}\uhk = \average{\betab\cdot \n\uk} + \snor{\betab\cdot \n}\average{\uk}.
\end{equation}

Next we study the convergence of the iHDG method
\eqnref{transportLocalk1}--\eqnref{transportLocale}. Since
\eqnref{transport} is linear, it is sufficient to show that iHDG
converges for the homogeneous equation with zero forcing $\f$ and zero
boundary condition $\g$. Let us  define $\pKout$ as the outflow part of $\pK$,
i.e. $\betab\cdot \n\geq 0$ on $\pKout$, and $\pKin$ as the inflow
part of $\pK$, i.e. $\betab\cdot \n\leq 0$ on $\pKin$. 
\begin{theorem}\theolab{DDMConvergence} Assume $-\Div \betab \ge \alpha > 0$, i.e. \eqnref{transport} is well-posed. The above iHDG iterations for the homogeneous transport equation \eqnref{transport} converge exponentially with respect to the number of iterations $k$. In particular, there exist $J \le \Nel$ such that 
\begin{align}
\label{ThmExponentialDecay}
\sum_{\K \in \Omega_h}\nor{\uk}^2_{\frac{-\Div \betab}{2}, \K} +
\nor{\uk}_{\snor{\betab\cdot\n},\pKout}^2&\le
\frac{c(k)}{2^k}\nor{\u^0}_{\snor{\betab\cdot\n},\Gh}^2,
\end{align}
 where $c(k)$ is a polynomial in $k$ of order at most $J$ and is independent of $h$ and $\p$.
\end{theorem}
\begin{remark} Note that the factor $\varrho(k)=\frac{k^J}{2^{k/2}}$, i.e. the largest possible term in $c\LRp{k}$, is a bounded function, which implies  $ \frac{c(k)}{2^{\frac{k}{2}}}$ is also bounded by a constant $C_J$ depending only on $J$. As a consequence:
$$\frac{c(k)}{2^{{k}}} \le \frac{C_J}{2^{k/2}} \stackrel{k\to \infty}{\longrightarrow} 0.$$
\end{remark}
\begin{proof}
Taking $\v = \u^{k+1}$ in \eqnref{transportLocalk1} and applying homogeneous forcing condition yield
\begin{equation}
\eqnlab{transportLocalb}
- \LRp{\u^{k+1},\nabla\cdot\LRp{\betab \u^{k+1}}}_{\K}  + \LRa{\betab\cdot \n\u^{k+1}+|\betab\cdot \n|(\u^{k+1}-\hat{\u}^k), \u^{k+1}}_{\pK} =0.
\end{equation}
 Since 
 \begin{eqnarray}\nonumber
 \LRp{\u^{k+1},\nabla\cdot\LRp{\betab \u^{k+1}}}_{K} &=&\LRp{\u^{k+1},\Div\betab \u^{k+1}}_{K}+\LRp{\u^{k+1},\betab \cdot \nabla  \u^{k+1}}_{K}, \nonumber
 \end{eqnarray}
integrating by parts the second term on right hand side
\begin{eqnarray}\nonumber
\LRp{\u^{k+1},\nabla\cdot\LRp{\betab \u^{k+1}}}_{K} &=&\LRp{\u^{k+1},\Div \betab \u^{k+1}}_{K}-\LRp{\u^{k+1},\nabla\cdot\LRp{\betab \u^{k+1}}}_{K}\\ \nonumber
 & &+\LRa{\betab\cdot \n\u^{k+1}, \u^{k+1}}_{\partial K}. \nonumber
 \end{eqnarray}
Rearranging the terms
\begin{equation}
\eqnlab{Identity}
\LRp{\u^{k+1},\nabla\cdot\LRp{\betab \u^{k+1}}}_{\K} 
=\LRp{\u^{k+1},\frac{\Div \betab}{2} \u^{k+1}}_{\K}+\frac{1}{2}\LRa{\betab\cdot \n\u^{k+1} ,\u^{k+1}}_{\pK}.
\end{equation}
Using \eqnref{Identity} we can rewrite \eqnref{transportLocalb} as
\begin{equation}
\label{transportLocalc}
\nor{\ukp}^2_{\frac{-\Div \betab}{2}, \K}+\LRa{\LRp{|\betab\cdot \n|+\frac{1}{2}\betab\cdot \n}\u^{k+1} ,\u^{k+1}}_{\pK}= \LRa{|\betab\cdot \n|\hat{\u}^k, \u^{k+1}}_{\pK}.
\end{equation}

On the other hand, \eqnref{transportLocale} is equivalent to
\begin{equation}
\label{uhk}
 |\betab\cdot\n|\hat{\u}^k=
\left\{
\begin{array}{ll}
|\betab\cdot\n|\uk & \text{ on } \pKout\\ 
|\betab\cdot\n|\uk_{\text{ext}}, & \text{ on } \pKin
\end{array}
\right.,
\end{equation}
where ${\u}^k_{\text{ext}}$ is either the physical boundary condition or the solution of the neighboring
element that shares the same inflow boundary $\pKin$.

Rewriting \eqref{transportLocalc} in terms of $\pKin$ and $\pKout$, we obtain
\begin{eqnarray*}
    &&\nor{\ukp}^2_{\frac{-\Div \betab}{2}, \K}+\frac{3}{2}\LRa{{|\betab\cdot \n|}\u^{k+1} ,\u^{k+1}}_{\pKout}+\frac{1}{2}\LRa{{|\betab\cdot \n|}\u^{k+1} ,\u^{k+1}}_{\pKin}\\\nonumber
&=& \LRa{|\betab\cdot \n|{\u}^k, \u^{k+1}}_{\pKout}+\LRa{|\betab\cdot \n|{\u}^k_{\text{ext}}, \u^{k+1}}_{\pKin}.
\end{eqnarray*}
By Cauchy-Schwarz inequality we have
\begin{eqnarray*}
    &&\nor{\ukp}^2_{\frac{-\Div \betab}{2}, \K}+\frac{3}{2}\LRa{{|\betab\cdot \n|}\u^{k+1} ,\u^{k+1}}_{\pKout}+\frac{1}{2}\LRa{{|\betab\cdot \n|}\u^{k+1} ,\u^{k+1}}_{\pKin}\\\nonumber
&\leq& \frac{1}{2}\LRa{|\betab\cdot \n|{\u}^k, \u^{k}}_{\pKout}+\frac{1}{2}\LRa{|\betab\cdot \n|{\u}^{k+1}, \u^{k+1}}_{\pKout}\\\nonumber
& &+\frac{1}{2}\LRa{|\betab\cdot \n|{\u}^k_{\text{ext}}, {\u}^k_{\text{ext}}}_{\pKin}+\frac{1}{2}\LRa{|\betab\cdot \n|\u^{k+1}, \u^{k+1}}_{\pKin},
\end{eqnarray*}
which implies
\begin{multline}
\nor{\ukp}^2_{\frac{-\Div \betab}{2}, \K}+\nor{\ukp}_{|\betab\cdot \n|,\pKout}^2
\leq \half\LRc{\nor{\uk}_{|\betab\cdot \n|,\pKout}^2+\nor{\uk_{\text{ext}}}_{|\betab\cdot \n|,\pKin}^2}.\eqnlab{transportLocalg}
\end{multline}

Consider the set $\mathcal{K}^1$ of all elements $K$ such that
$\pKin$ is a subset of the physical inflow boundary
$\pOmega^{\text{in}}$ on which
we have $\u^k_{\text{ext}}=0$ for all $k \in \mathbb{N}$.
 We obtain from \eqnref{transportLocalg} that
\begin{equation}
\eqnlab{transportLocalh}
\nor{\ukp}^2_{\frac{-\Div \betab}{2}, \K}+\nor{\ukp}_{|\betab\cdot \n|,\pKout}^2
\leq \half\nor{\uk}_{|\betab\cdot \n|,\pKout}^2,
\end{equation}
which implies
\begin{equation}
\eqnlab{transportLocali}
\nor{\ukp}_{|\betab\cdot \n|,\pKout}^2
\leq \half\nor{\uk}_{|\betab\cdot \n|,\pKout}^2 \leq \cdots\leq
\frac{1}{2^{k+1}}\nor{\u^0}_{|\betab\cdot \n|,\pKout}^2.
\end{equation}
From \eqnref{transportLocalh} and \eqnref{transportLocali} we also have
\begin{equation}
\nor{\ukp}^2_{\frac{-\Div \betab}{2}, \K} \leq
\frac{1}{2^{k+1}}\nor{\u^0}_{|\betab\cdot \n|,\pKout}^2.
\end{equation}

Next, let us define $\Omega^1_h := \Omega_h$ and
\[
\Omega^2_h :=\Omega^1_h\backslash \mc{K}^1.
\]
Consider the set $\mathcal{K}^2$ of all $K$ in $\Omega^2_h$ such that $\pKin$ is either
(possibly partially) a subset of the physical inflow boundary
$\pOmega^{\text{in}}$ or (possibly partially) a subset of the outflow
boundary of elements in $\mc{K}^1$. This implies, on $\pKin \in \mc{K}^2$, $\u^k_{\text{ext}}$ either is zero
for all $k \in \mathbb{N}\setminus\LRc{1}$ or
satisfies the bound 
\begin{equation}
\label{transportLocalrighta}
\nor{\uk_{\text{ext}}}_{|\betab\cdot \n|,\pKin}^2
\leq 
\frac{1}{2^{k}}\nor{\u^0_{\text{ext}}}_{|\betab\cdot \n|,\pKin}^2.
\end{equation}
Combining \eqnref{transportLocalg} and \eqref{transportLocalrighta}, we obtain
\begin{eqnarray}\eqnlab{pKoutBound}
\nor{\ukp}_{|\betab\cdot \n|,\pKout}^2 \le \frac{1}{2^{k+1}}\LRc{\nor{\u^0}_{|\betab\cdot \n|,\pKout}^2 + (k+1)\nor{\u^0_{\text{ext}}}_{|\betab\cdot \n|,\pKin}^2},
\end{eqnarray}
which, together with \eqnref{transportLocalg}, leads to
\begin{equation}
\eqnlab{KBound}
\nor{\ukp}^2_{\frac{-\Div \betab}{2}, \K}\leq \frac{1}{2^{k+1}}\LRc{\nor{\u^0}_{|\betab\cdot \n|,\pKout}^2 + (k+1)\nor{\u^0_{\text{ext}}}_{|\betab\cdot \n|,\pKin}^2}.
\end{equation}
Now defining $\Omega^i_h$ and $\mc{K}^{i}$ recursively, and
repeating the above arguments concludes the proof.
\end{proof}

We can see that the contraction constant in this case is $1/2$, in our numerical experiments we found the
spectral radius of the iteration matrix to be exactly $1/2$ which confirms the theoretical result. 
We are in a position to discuss the convergence of the $k$th iterative
solution to the exact solution $\u^e$. For sufficiently smooth exact
solution, e.g. $\eval{u^e}_K \in H^{s}\LRp{\K}, s > 3/2$, we assume
the following standard convergence result of DG (HDG) methods for
transport equation: let $\sigma = \min\LRc{\p+1,s}$, we have
\begin{equation}
\eqnlab{HDGconvergence}
\nor{\u - \u^e}_{\Omega_h}^2 \le C \frac{h^{2\sigma - 1}}{\p^{2s-1}}\nor{\u^e}_{H^{s}\LRp{\Omega_h}}^2,
\end{equation}
and we refer the readers to, for example, \cite{JohnsonPitkaranta86,
  Bui-Thanh15} for a proof.
\begin{corollary}
Suppose the exact solution $\u^e$ is sufficiently smooth,
i.e. $\eval{u^e}_K \in H^{s}\LRp{\K}, s > 3/2$, then there exists a constant
$C$ independent of $k,h$ and $\p$ such that
\begin{align*}
\nor{\uk - \u^e}_{\Omega_h}^2 &\le C\LRp{\frac{c(k)}{2^k}\nor{\u^0}_{\snor{\betab\cdot\n},\Gh}^2 +  \frac{h^{2\sigma - 1}}{\p^{2s-1}}\nor{\u^e}_{H^{s}\LRp{\Omega_h}}^2},
\end{align*}
 where $c(k)$ is a polynomial in $k$ of order at most $\Nel$ and is independent of $h$ and $\p$.
\end{corollary}
\begin{proof}
The result is a direct consequence of the result from Theorem \theoref{DDMConvergence},
the HDG (DG) convergence result \eqnref{HDGconvergence}, and the triangle inequality.
\end{proof}

\begin{remark}
\remalab{time-dependent-transport}
For time-dependent transport equation, we discretize the spatial
operator using HDG and time using backward Euler method (or Crank-Nicholson method or higher-order method if desired). The iHDG approach in this case is almost identical to the
one for steady state equation except that we now have an additional
$L^2$-term $\LRp{\ukp,\v}_{\K}/\Delta t$ in the local equation
\eqnref{transportLocalk1}. This improves the convergence of iHDG. Indeed, the convergence analysis is almost
identical except we now have $\nor{\ukp}^2_{-\Div \betab/2 +
  1/\Delta t, \K}$ instead of $\nor{\ukp}^2_{-\Div \betab/2,
  \K}$ in \eqnref{transportLocalg}.
\end{remark}

Now let us consider the flux given in Nguyen et. al. (NPC flux) \cite{NguyenPeraireCockburn09a} and analyze 
the convergence of iHDG scheme. The stabilization $\tau$ of NPC flux is given by
\begin{equation}
\eqnlab{Ntaub}
\tau=|\betab\cdot \n|\frac{1+sgn({\betab\cdot \n})}{2}-{\betab\cdot\n},
\end{equation}
and the trace $\uhk$ at the $k$th iteration is computed as
\[
	\uhk = \frac{\average{\tau\uk}+\average{\betab\cdot\n\uk}}{\average{\tau}}.
\]

In this case we apply the iHDG algorithm \ref{al:DDMhyperbolic} without the weighting $\Aa$ in front 
of the trace $\uhk$ as the stabilization comes from $\tau$.

\begin{theorem}
\theolab{iHDG-NPC} Assume $-\Div \betab \ge \alpha > 0$,
i.e. \eqnref{transport} is well-posed.  There exists $J \le \Nel$ such
that the iHDG algorithm with the NPC flux for the homogeneous
transport equation converges in $J$ iterations.\footnote{The proof is
similar to the proof of Theorem \theoref{DDMConvergence}, and hence
is omitted here.}
\end{theorem}

This theorem shows that for the scalar hyperbolic equation
\eqnref{transport}, iHDG with the NPC flux converges in finite number of
iterations, which is faster than the upwind HDG flux. The reason is that the NPC flux
mimics the matching of wave propagation from the inflow to the outflow. 
However, designing such a scheme for a system of hyperbolic equations, such as 
the linearized shallow water system, does not seem to be tractable due to the 
interaction of more than one waves. In this sense the upwind HDG flux is more robust, since
it is applicable for other system of hyperbolic PDEs as well, as we now show.

We next consider the following system of linear hyperbolic PDEs arisen
from the oceanic linearized shallow water system
\cite{GiraldoWarburton08}:
\begin{equation}
    \pp{}{t}\LRp{
      \begin{array}{c}
        \phi^e \\
       \Phi \u^e \\
        \Phi \v^e
      \end{array}
    } + 
    \pp{}{x}\LRp{
      \begin{array}{c}
        \Phi  \u^e \\
        \Phi\phi^e \\
        0
      \end{array}
    } + 
    \pp{}{y}\LRp{
      \begin{array}{c}
        \Phi \v^e \\
        0 \\
        \Phi\phi^e
      \end{array}
    } = 
    \LRp{
      \begin{array}{c}
        0 \\
        f\Phi\v^e - \gamma \Phi\u^e + \frac{\tau_x}{\rho} \\
       -f\Phi\u^e - \gamma \Phi\v^e + \frac{\tau_y}{\rho}
      \end{array}
    }
\eqnlab{linearizedShallow}
\end{equation}
where $\phi = g \h$ is the geopotential height with $g$ and $\h$ being
the gravitational constant and the perturbation of the free surface
height, $\Phi > 0$ is a constant mean flow geopotential height, $\vel := \LRp{\u,\v}$ is
the perturbed velocity, $\gamma \ge 0$ is the bottom friction, 
$\bs{\tau}:=\LRp{\tau_x,\tau_y}$ is the wind stress, and $\rho$ is the density of
the water. Here, $f = f_0 + \beta \LRp{y - y_m}$ is the
Coriolis parameter,  where $f_0$, $\beta$, and $y_m$ are given constants.

Again, for the simplicity of the exposition and the analysis, let us employ
the backward Euler discretization for temporal derivatives and HDG \cite{bui2016construction} 
for spatial ones. Since the unknowns of interest
are those at the $(m+1)$th time step, we can suppress the time index
for the clarity of the exposition.  Furthermore, since the system
\eqnref{linearizedShallow} is linear, a similar argument as above shows that it is sufficient to consider homogeneous system with zero initial condition, boundary condition, and forcing. Also here we consider the case of $\bs{\tau}=0$. 
Applying the iHDG algorithm \ref{al:DDMhyperbolic} to 
 the  homogeneous system gives
\begin{subequations}
\eqnlab{localSolverDD}
\begin{align}
\eqnlab{localSolverPhiDD}
&\LRp{\frac{\phikp}{\Delta t},\varphi_1}_K - \LRp{\Phi\velkp,\Grad\varphi_1}_K + \LRa{\Phi \velkp \cdot \n + \sqrt{\Phi}\LRp{\phikp -\phihk},\varphi_1}_\pK = 0, \\
\eqnlab{localSolverUDD}
&\LRp{\frac{{\Phi\ukp}}{\Delta t},\varphi_2}_K - \LRp{\Phi\phikp,\pp{\varphi_2}{x}}_K + \LRa{\Phi\phihk\n_1,\varphi_2}_\pK = \LRp{f\Phi\vkp - \gamma \Phi\ukp, \varphi_2}_K, \\
\label{localSolverVDD}
&\LRp{\frac{{\Phi\vkp}}{\Delta t},\varphi_3}_K - \LRp{\Phi\phikp,\pp{\varphi_3}{y}}_K + \LRa{\Phi\phihk\n_2,\varphi_3}_\pK = \LRp{-f\Phi\ukp - \gamma \Phi\vkp , \varphi_3}_K,
\end{align}
\end{subequations}
where $\varphi_1, \varphi_2$ and $\varphi_3$ are the test functions, and 
\begin{equation*}
\phihk= \average{\phik} + \sqrt{\Phi}\average{\velk\cdot\n}.
\end{equation*}
Our goal is to show that $\LRp{\phikp,\Phi\velkp}$ converges to
zero. To that end, let us define
\begin{eqnarray}\eqnlab{ContractionConstant}
\mathcal{C}:=\frac{\mathcal{A}}{\mathcal{B}}, \quad \mathcal{A}
:=\max\left\{\frac{\Phi+\sqrt\Phi}{2},\frac{1+\sqrt\Phi}{2}\right\},
\end{eqnarray}
and
\begin{eqnarray*}
\mathcal{B}
:=\min\left\{\left(\frac{ch}{\Delta t(\p+1)(\p+2)}+\frac{\sqrt\Phi -\Phi}{2}\right),\left(\left(\gamma+\frac{1}{\Delta t}\right)\frac{ch}{(\p+1)(\p+2)}+\frac{(-1-\sqrt\Phi)}{2}\right)\right\}.
\end{eqnarray*}
where $0<c\le1$ is a constant.
We also need the following norms:
\begin{align}
\nor{\LRp{\phik,\velk}}^2_{\Omega_h}&:=\nor{\phik}_{\Omega_h}^2 + \nor{\velk}_{\Phi, \Omega_h}^2, \quad \nor{\LRp{\phik,\velk}}^2_\Gh&:=\nor{\phik}_{\Gh}^2 + \nor{\velk}_{\Phi, \Gh}^2.\nonumber
\end{align}

\begin{theorem}\theolab{ThmConvergenceShallowWater} Assume that the mesh size $h$, the time step $\Delta t$ and the solution order $\p$ are chosen such that $\mc{B} > 0$ and $\mathcal{C}<1$,
then the approximate solution at the $k$th iteration
$\LRp{\phik,\velk}$ converges to zero, i.e.,
\begin{align}\nonumber
\nor{\LRp{\phik,\velk}}^2_\Gh \leq \mathcal{C}^k \nor{\LRp{\phi^0,\vel^0}}^2_\Gh, \quad
\nor{\LRp{\phik,\velk}}^2_{\Omega_h} \leq \Delta t \mc{A}\LRp{\mc{C}+1}\mc{C}^{k-1}\nor{\LRp{\phi^0,\vel^0}}^2_\Gh, \nonumber
\end{align}
where $\mc{C}$ is defined in \eqnref{ContractionConstant}.
\end{theorem}
\begin{proof}
Choosing the test functions $\varphi_1 =\phikp$, $\varphi_2 = \ukp$
and $\varphi_3 = \vkp$ in \eqnref{localSolverDD}, integrating the second term in \eqnref{localSolverPhiDD} by parts, and then summing equations in \eqnref{localSolverDD} altogether, we obtain
\begin{align}\label{SumlocalSolverTest}\nonumber
\frac{1}{\Delta t}\LRp{\phikp,\phikp}_K+\frac{\Phi}{\Delta t}\LRp{\velkp,\velkp}_K+\sqrt\Phi\LRa{\phikp,\phikp}_\pK&+\gamma\Phi\LRp{\velkp,\velkp}_\K \\
=\sqrt\Phi\LRa{\phihk,\phikp}_\pK-\Phi\LRa{\phihk,{\bf\n}\cdot{\velkp}}_\pK&.
\end{align}
Summing \eqref{SumlocalSolverTest} over all elements yields
\begin{align}\nonumber
&\sum_K\frac{1}{\Delta t}\LRp{\phikp,\phikp}_K+\frac{\Phi}{\Delta t}\LRp{\velkp,\velkp}_K+\sqrt\Phi\LRa{\phikp,\phikp}_\pK+\gamma\Phi\LRp{\velkp,\velkp}_\K\\\nonumber
&=\sum_{\pK}\sqrt\Phi\LRa{\phihk,\phikp}_\pK-\Phi\LRa{\phihk,{\bf\n}\cdot{\velkp}}_\pK\\
&=\sum_{\e\in\Gh}\LRa{2\sqrt\Phi\LRp{\average{\phik}+\sqrt{\Phi}\average{\velk\cdot \n}},\LRp{\average{\phikp}-\sqrt\Phi\average{\velkp\cdot \n}}}_\e,\nonumber
\intertext{by Cauchy-Schwarz inequality, we could bound the rhs as} \nonumber
&\leq\sum_{\e\in\Gh}\sqrt\Phi\LRp{\nor{\average{\phik}}_e^2+\nor{\average{\phikp}}_e^2}+\Phi\LRp{\nor{\average{\phik}}_e^2+\nor{\average{\velkp\cdot\n}}_e^2}\\\nonumber 
&+\Phi\LRp{\nor{\average{\phikp}}_e^2+\nor{\average{\velk\cdot\n}}_e^2}+\Phi\sqrt\Phi\LRp{\nor{\average{\velk\cdot\n}}_e^2+\nor{\average{\velkp\cdot\n}}_e^2},\\\nonumber
\intertext{with little algebraic manipulation we have} \nonumber
\label{SumK}
&\leq\sum_\pK\left[\frac{\Phi+\sqrt\Phi}{2}\LRa{\phik,\phik}_\pK+\frac{\Phi(1+\sqrt\Phi)}{2}\LRa{\velk,\velk}_\pK\right] \nonumber\\
&+\sum_\pK\left[\frac{\Phi+\sqrt\Phi}{2}\LRa{\phikp,\phikp}_\pK+\frac{\Phi(1+\sqrt\Phi)}{2}\LRa{\velkp,\velkp}_\pK\right].
\end{align}
An application of inverse trace inequality \cite{chan2016gpu} for tensor product elements gives
\begin{subequations}\label{FEMTraceIn}
\begin{align}
    \LRp{\phikp,\phikp}_K&\geq\frac{2ch}{d(\p+1)(\p+2)}\LRa{\phikp,\phikp}_\pK,\\
    \LRp{\velkp,\velkp}_K&\geq\frac{2ch}{d(\p+1)(\p+2)}\LRa{\velkp,\velkp}_\pK,
\end{align}
\end{subequations}
where $d$ is the spatial dimension which in this case is $2$ and $0<c\le1$ is a constant. For simplices we can use the trace inequalities in \cite{MR1986022} and it will change only the constants in the proof. Inequality \eqref{FEMTraceIn}, together with \eqref{SumK}, implies
 \begin{eqnarray}\label{CompareTraces}\nonumber
&&\sum_\pK\left[\left(\frac{ch}{\Delta t(\p+1)(\p+2)}+\frac{\sqrt\Phi -\Phi}{2}\right)\LRa{\phikp,\phikp}_\pK\right.\\\nonumber
&&\left.+\left(\left(\gamma+\frac{1}{\Delta t}\right)\frac{ ch}{(\p+1)(\p+2)}+\frac{(-1-\sqrt\Phi)}{2}\right)\LRa{\Phi\velkp,\velkp}_\pK\right]\\
&\leq&\sum_\pK\left[\frac{\Phi+\sqrt\Phi}{2}\LRa{\phik,\phik}_\pK+\frac{(1+\sqrt\Phi)}{2}\LRa{\Phi\velk,\velk}_\pK\right],
\end{eqnarray}
which implies
 \begin{eqnarray}\label{Contraction}\nonumber
\nor{\LRp{\phikp,\velkp}}^2_\Gh \le \mc{C}\nor{\LRp{\phik,\velk}}^2_\Gh,
\end{eqnarray}
where the constant $\mathcal{C}$ is computed as in \eqnref{ContractionConstant}. Therefore 
\begin{eqnarray}\label{TraceExponentialDecay}
\nor{\LRp{\phikp,\velkp}}^2_\Gh \le \mc{C}^{k+1}\nor{\LRp{\phi^0,\vel^0}}^2_\Gh.
\end{eqnarray}
On the other hand, inequalities \eqref{SumK} and \eqref{TraceExponentialDecay} imply
\begin{align}\nonumber
\nor{\LRp{\phikp,\velkp}}^2_{\Omega_h} &\le \Delta t \mc{A}\LRp{\mc{C}+1}\mc{C}^k\nor{\LRp{\phi^0,\vel^0}}^2_\Gh
\end{align}
and this ends the proof.
\end{proof}
\begin{remark}
\label{Shallow-water-remark}
The above theorem implies that, in order to have a convergent
algorithm, we need to have the following relation between $\Delta t$, $p$,
and $h$ 
	$$\Delta t=\mc{O}\LRp{\frac{h}{\Phi(p+1)(p+2)}}.$$
Unlike the convergent result in 
Theorem \theoref{DDMConvergence} for scalar hyperbolic equation, the finding in Theorem
\theoref{ThmConvergenceShallowWater} shows that iHDG is conditionally
convergent for a system of hyperbolic equations. More specifically, the
convergence rate depends on the mesh size $h$, the time step $\Delta
t$, and the solution order $\p$. This will be confirmed by numerical
results in Section \secref{numerics}.
\end{remark}

To show the convergence of the $k$th iterative solution to the exact
solution $\LRp{\phie,\Phi\vele}$, we assume that the exact solution is
smooth, i.e.  $\eval{\LRp{\phie,\Phi\vele}}_\K \in \LRs{H^{s}\LRp{\K}}^3, s >
3/2$. If we define 
\[
\mc{E}^e\LRp{t} := \sum_\K \nor{\phie\LRp{t}}_{H^{s}\LRp{\K}}^2 + \nor{\vele\LRp{t}}_{\Phi,H^{s}\LRp{\K}}^2,
\]
results from \cite{bui2016construction} show that, for $\gamma > 0$ and $\sigma = \min\LRc{\p+1,s}$, we have
\begin{equation}
\eqnlab{HDGconvergenceshallow}
\nor{\LRp{\phi-\phie, \vel - \vele}}_{\Omega_h}^2 \le C \Delta t\frac{h^{2\sigma - 1}}{\p^{2s-1}}\mc{E}^e\LRp{m\Delta t},
\end{equation}
at the $m$th time step.

\begin{corollary}
Suppose the exact solution satisfies $\eval{\LRp{\phie,\Phi\vele}}_\K \in
\LRs{H^{s}\LRp{\K}}^3, s > 3/2$, then there exists a constant $C$
independent of $k,h$ and $\p$ such that 
\[
\nor{\LRp{\phik-\phie, \velk - \vele}}_{\Omega_h}^2 \le C\Delta t\LRp{\mc{A}\LRp{\mc{C}+1}\mc{C}^{k-1} \nor{\LRp{\phi^0,\vel^0}}^2_{\Gh} +\frac{h^{2\sigma - 1}}{\p^{2s-1}}\mc{E}^e\LRp{m\Delta t}},
\]
with $\sigma = \min\LRc{\p+1,s}$.
\end{corollary}

\section{iHDG methods for convection-diffusion PDEs}
\seclab{iHDG-convection-diffusion}
\subsection{First order form}
\seclab{iHDG-first-order-CDR}
In this section we apply the iHDG algorithm \ref{al:DDMhyperbolic} to the following prototype 
convection-diffusion equation in the first order form
\begin{subequations}
\eqnlab{Convection-diffusion-eqn}
\begin{align}
\kappa^{-1}\sigb^e + \Grad \u^e &= 0 \quad \text{ in } \Omega, \\ 
\Div\sigb^e +\betab \cdot \Grad \u^e +\nu \u^e &= f \quad \text{ in }
\Omega.
\end{align}
\end{subequations}
We suppose that \eqnref{Convection-diffusion-eqn} is well-posed, i.e.,
\begin{equation}
	\label{CondCoercivity}
	\nu-\frac{\Div \betab}{2}\geq \lambda > 0.
\end{equation}
Moreover, we restrict ourselves to a constant diffusion coefficient $\kappa$. 
An upwind HDG numerical flux \cite{Bui-Thanh15} is given by
\[
\Fh\cdot \n = 
\LRs{
\begin{array}{c}
\uh \n_1
\\ \uh \n_2
\\ \uh \n_3
\\ \sigb \cdot \n + \betab \cdot \n \u + \tau \LRp{\u -\uh}
\end{array}
}.
\]
Strongly enforcing the conservation condition yields
\[
\uh=\frac{1}{\tau^{+}+\tau^{-}} \LRp{\jump{\sigb\cdot\n}+\jump{\betab\cdot\n \u}+\jump{\tau \u}},
\]
with $\tau$ being chosen as
\begin{equation}
\label{general_flux_convdiff}
\tau^{\pm}=\frac{\gamma}{2}\LRp{\alpha-\betab\cdot\n^{\pm}}.
\end{equation}
where $\gamma=1$ and $\alpha=\sqrt{|\betab\cdot\n|^2+4}$ for the upwind flux in \cite{Bui-Thanh15}. We 
see that $\tau^{\pm}$ in general is a function which depends upon $\betab$ and is always positive. Similar to the previous 
sections, it is sufficient to
consider the homogeneous problem. Applying the iHDG algorithm
\ref{al:DDMhyperbolic} with $\taub,\v$ as test functions we have the following iterative scheme

\begin{subequations}
\begin{align}
\label{ErrorDDMCD1}
\kappa^{-1}\LRp{\sigbkp,\taub}_K-\LRp{\ukp,\nabla\cdot\taub}_K+\LRa{\uhk,\taub\cdot\n}_\pK&=0,\\
\label{ErrorDDMCD2}
-\LRp{\sigbkp,\nabla\v}_K-\LRp{\ukp,\nabla\cdot\LRp{\betab\v}-\nu{\v}}_K&+\nonumber\\
\LRa{\sigbkp\cdot\n+\betab\cdot\n\ukp+\tau(\ukp-\uhk),\v}_\pK&=0,
\end{align}
\end{subequations}
where 
$$\uhk=\frac{\jump{\sigbk\cdot\n}+\jump{\betab\cdot\n \uk}+\jump{\tau \uk}}{\sqrt{|\betab\cdot\n|^2+4}}.$$
For $\varepsilon$, $h>0$ and $0<c\le1$ given, define
\begin{align}\label{DefineC1}
\mathcal{C}_1&:=\frac{3(\|\betab\cdot{\bf n}\|_{L^\infty(\pK)}^{2}+\bar\tau^{2})(\bar\tau\varepsilon+1)}{2\varepsilon}, \, \mathcal{C}_2:=\frac{3(\bar\tau\varepsilon+1)}{2\varepsilon},\\\label{DefineC2}
    \mathcal{C}_3&:=\frac{\bar\tau+\|\betab\cdot{\bf n}\|_{L^\infty(\pK)}}{2}, \, \mathcal{C}_4:=\frac{\varepsilon}{2},\\\label{DefineD}
    \mc{D}&:= \frac{\mc{A}}{\mc{B}}, \, \mc{A}={\max\{\mathcal{C}_1,\mathcal{C}_2\}}, \, \mc{E}:=\frac{\max\{\mathcal{C}_3,\mathcal{C}_4\}}{\min\{\kappa^{-1},\lambda\}}, \, \mc{F}:=\frac{\mc{A}}{\min\{\kappa^{-1},\lambda\}},\\\label{DefineB} 
    \mc{B}&:={\min\{\frac{2ch\kappa^{-1}}{d(p+1)(p+2)}-\mathcal{C}_4, \frac{2ch\lambda}{d(p+1)(p+2)}+\taustar-\mathcal{C}_3\}}.
\end{align}
As in the previous section we need the following norms 
\begin{align}
\nor{\LRp{\sigb^k,\uk}}^2_{\Omega_h}&:=\nor{\sigb^k}_{\Omega_h}^2 + \nor{\uk}_{\Omega_h}^2, \quad \nor{\LRp{\sigb^k,\uk}}^2_\Gh&:=\nor{\sigb^k}_{\Gh}^2 + \nor{\uk}_{\Gh}^2.\nonumber
\end{align}

\begin{theorem} 
\theolab{ThmUpwindDDM}
Suppose that the mesh size $h$ and the solution order $p$ are chosen such that  $\mc{B}>0$ and $\mc{D}<1$, 
the algorithm  \eqref{ErrorDDMCD1}-\eqref{ErrorDDMCD2} converges in the following sense
\begin{eqnarray}\nonumber
\nor{\LRp{\sigb^{k},\uk}}^2_\Gh \le \mc{D}^{k}\nor{\LRp{\sigb^{0},\u^{0}}}^2_\Gh, \quad
\nor{\LRp{\sigb^{k},\uk}}^2_{\Omega_h} \le (\mc{E}\mc{D}+\mc{F})\mc{D}^{k-1}\nor{\LRp{\sigb^{0},\u^{0}}}^2_\Gh, 
\end{eqnarray}
where $\mc{D}, \mc{E}$ and $\mc{F}$ are defined in \eqref{DefineD}. 
\end{theorem}

\begin{proof}
Choosing $\sigbkp$ and $\ukp$ as test functions in
\eqref{ErrorDDMCD1}-\eqref{ErrorDDMCD2}, integrating the second term
in \eqref{ErrorDDMCD1} by parts, using \eqnref{Identity} for second term in \eqref{ErrorDDMCD2}, 
and then summing up the resulting two equations we get
\begin{eqnarray}\nonumber
	&&\kappa^{-1}\LRp{\sigb^{k+1},\sigb^{k+1}}_K+\LRp{\frac{-\Div\betab}{2}\ukp,\ukp}_K+\nu\LRp{\ukp,{\ukp}}_K-\frac{1}{2}\LRa{\betab\cdot{\bf n} \ukp,\ukp}_\pK\\\label{ErrorDDMCD4}
& &
+\LRa{\hat{\u}^{k},\sigb^{k+1}\cdot{\bf n}}_\pK+\LRa{\betab\cdot{\bf n}\ukp+\tau(\ukp-\hat{\u}^{k}),\ukp}_\pK=0.
\end{eqnarray}
Due to the condition \eqref{CondCoercivity}
\begin{eqnarray}\nonumber
	&&\kappa^{-1}\LRp{\sigb^{k+1},\sigb^{k+1}}_K+\lambda\LRp{\ukp,\ukp}_K+\LRa{\tau\ukp,\ukp}_\pK \\\label{ErrorDDMCD5}
& &
\leq\LRa{\tau\hat{\u}^{k},\ukp}_\pK-\frac{1}{2}\LRa{\betab\cdot{\bf n} \ukp,\ukp}_\pK-\LRa{\hat{\u}^{k},\sigb^{k+1}\cdot{\bf n}}_\pK.
\end{eqnarray}

By Cauchy-Schwarz and Young's inequalities and the fact that $|\betab\cdot{\bf n}|\leq \|\betab\cdot{\bf n}\|_{L^\infty(\pOmega_h)}$ and let $\bar\tau := \|\tau\|_{L^\infty(\pOmega_h)}$, $\taustar := \inf\limits_{\pK \in \pOmega_h}\tau$,
\begin{eqnarray}\nonumber
    &&\kappa^{-1}\|\sigb^{k+1}\|_{L^2(K)}^2+\lambda\|\ukp\|^2_{L^2(K)}+\taustar\|\ukp\|^2_{L^2(\pK)}\\\nonumber
&\leq&\frac{\bar\tau\varepsilon +1}{2\varepsilon}\|\hat\u^{k}\|_{L^2(\pK)}^2+\frac{\bar\tau+\|\betab\cdot{\bf n}\|_{L^\infty(\pK)}}{2}\|\ukp\|_{L^2(\pK)}^2+\frac{\varepsilon}{2}\|\sigb^{k+1}\|_{L^2(\pK)}^2.
\end{eqnarray}
Therefore
\begin{eqnarray}\nonumber
\label{ErrorDDMCD7}
&&\sum_K\LRs{\kappa^{-1}\|\sigb^{k+1}\|_{L^2(K)}^2+\lambda\|\ukp\|^2_{L^2(K)}+\taustar\|\ukp\|^2_{L^2(\pK)}}\leq\sum_\pK\frac{\bar\tau\varepsilon +1}{2\varepsilon}\|\hat\u^{k}\|_{L^2(\pK)}^2\\
       &&
+\frac{\bar\tau+\|\betab\cdot{\bf n}\|_{L^\infty(\pK)}}{2}\|\ukp\|_{L^2(\pK)}^2+\frac{\varepsilon}{2}\|\sigb^{k+1}\|_{L^2(\pK)}^2.
\end{eqnarray}
By Cauchy-Schwarz inequality
\begin{eqnarray*}\nonumber
&&\|\hat{\u}^{k}\|_{L^2(\pK)}^2
\leq\frac{3\|\jump{\sigb^{k}\cdot\n}\|_{L^2(\pK)}^2+3\|\jump{\betab\cdot\n \uk}\|_{L^2(\pK)}^2+3\|\jump{\tau \uk}\|_{L^2(\pK)}^2}{{4}},
\end{eqnarray*}
which implies
\begin{eqnarray}
\label{ErrorDDMCD8}
&&\sum_\pK\|\hat{\u}^{k}\|_{L^2(\pK)}^2\leq\sum_\pK{3\|\sigb^{k}\|_{L^2(\pK)}^2+{3(\|\betab\cdot{\bf n}\|_{L^\infty(\pK)}^{2}+\bar\tau^{2})}\|{ \uk}\|_{L^2(\pK)}^2}.
\end{eqnarray}
Combining \eqref{ErrorDDMCD7} and \eqref{ErrorDDMCD8}, we get
\begin{eqnarray}\nonumber
	&&\sum_K\LRs{\kappa^{-1}\|\sigb^{k+1}\|_{L^2(K)}^2+\lambda\|\ukp\|^2_{L^2(K)}+\taustar\|\ukp\|^2_{L^2(\pK)}}\\\nonumber
 &\leq&\sum_\pK\LRs{\frac{3(\bar\tau\varepsilon+1)}{2\varepsilon}{\|\sigb^{k}\|_{L^2(\pK)}^2+\frac{3(\|\betab\cdot{\bf n}\|_{L^\infty(\pK)}^{2}+\bar\tau^{2})(\bar\tau\varepsilon+1)}{2\varepsilon}\|{ \uk}\|_{L^2(\pK)}^2}\right.\\\label{ErrorDDMCD9}
 & &\left.+\frac{\bar\tau+\|\betab\cdot{\bf n}\|_{L^\infty(\pK)}}{2}\|\u^{k+1}\|_{L^2(\pK)}^2+\frac{\varepsilon}{2}\|\sigb^{k+1}\|_{L^2(\pK)}^2}.
\end{eqnarray}
By the inverse trace inequality \eqref{FEMTraceIn} we infer from \eqref{ErrorDDMCD9} that
\begin{eqnarray}\nonumber
    &&\sum_\pK\LRs{(\frac{2ch\kappa^{-1}}{d(p+1)(p+2)}-\mathcal{C}_4)\|\sigb^{k+1}\|_{L^2(\pK)}^2+(\frac{2ch\lambda}{d(p+1)(p+2)}
+\taustar-\mathcal{C}_3)\|\ukp\|^2_{L^2(\pK)}}\\ \nonumber
    &&\leq\sum_\pK\LRs{\mathcal{C}_1\|{ \uk}\|_{L^2(\pK)}^2+\mathcal{C}_2\|\sigb^{k}\|_{L^2(\pK)}^2}, \nonumber
\end{eqnarray}
which implies
 \begin{eqnarray}\nonumber
     \nor{\LRp{\sigb^{k+1},\ukp}}^2_\Gh \le \mc{D}\nor{\LRp{\sigb^{k},\uk}}^2_\Gh,
\end{eqnarray}
where the constant $\mc{D}$ is computed as in \eqref{DefineD}.
Therefore 
\begin{eqnarray}\label{TraceExponentialDecay_CD}
    \nor{\LRp{\sigb^{k+1},\ukp}}^2_\Gh \le \mc{D}^{k+1}\nor{\LRp{\sigb^{0},\u^{0}}}^2_\Gh.
\end{eqnarray}
Inequalities \eqref{ErrorDDMCD9} and \eqref{TraceExponentialDecay_CD} imply
\begin{eqnarray}\nonumber
    \nor{\LRp{\sigb^{k+1},\ukp}}^2_{\Omega_h} \le (\mc{E}\mc{D}+\mc{F})\mc{D}^{k}\nor{\LRp{\sigb^{0},\u^{0}}}^2_\Gh,
\end{eqnarray}
and this concludes the proof.
\end{proof}
\begin{remark}
\label{time-conv-diff}
For time-dependent convection-diffusion equation, we choose to
discretize the spatial differential operators using HDG. For the temporal
derivative, we use implicit time stepping methods, again with either backward Euler or Crank-Nicolson method
for simplicity. The iHDG approach in this case is almost identical to
the one for steady state equation except that we now have an
additional $L^2$-term $\LRp{\ukp,\v}_{\K}/\Delta t$ in the local
equation \eqref{ErrorDDMCD2}. This improves the convergence of
iHDG. Indeed, the convergence analysis is almost identical except we
now have $\lambda + 1/\Delta t$ in place of
$\lambda$.
\end{remark}

\subsection{Comment on iHDG methods for elliptic PDEs}
\seclab{convergence-condition-convdiff}
In this section we consider elliptic PDEs with $\betab = 0$, $\kappa = 1$ in \eqnref{Convection-diffusion-eqn}. First using the conditions for convergence derived in section \secref{iHDG-first-order-CDR} let us analyze the iHDG scheme with upwind flux \eqref{general_flux_convdiff}. Now for elliptic PDEs the stabilization $\tau$ for upwind flux reduces to $\tau = 1$ and this violates the condition $\mathcal{B}>0$ in Theorem \theoref{ThmUpwindDDM}. Therefore for any mesh, the iHDG scheme
with upwind flux will diverge and it is also observed in our numerical experiments. 

To fix the issue, we carry out a similar analysis as in section
\secref{iHDG-first-order-CDR}, but this time without a specific
$\tau$. Taking $\varepsilon=\frac{1}{\bar\tau}$, the condition
$\mathcal{B}>0$ dictates the following mesh dependent $\tau$ for the
convergence of the iHDG scheme
\[
    \bar\tau > \mc{O}\LRp{\frac{d(p+1)(p+2)}{4h}}.
\]
This result shows that the convergence of iHDG scheme with upwind flux
for elliptic PDEs requires mesh dependent stabilization. It is also
worth noting that this coincides with the form of stabilization used
in hybridizable interior penalty methods \cite{GanderHajian:2015:ASM}
for elliptic PDEs, even though the scheme described in
\cite{GanderHajian:2015:ASM} and iHDG are different.

Guided by the above analysis, we take $\tau = \bar\tau = \taustar = \frac{\gamma(p+1)(p+2)}{h}$,
where $\gamma>\frac{d}{4}$ is a constant we still need to enforce
$\mathcal{D}<1$ for the convergence of the iHDG scheme. Generally this requires
 $4$ conditions depending upon minimum and maximum of constants
$\mathcal{A,B}$ as in Theorem \theoref{ThmUpwindDDM}. However, the simple choice
of $\tau = \frac{\gamma(p+1)(p+2)}{h}$ makes
$\mathcal{A}=\mathcal{C}_1$ and the number of  conditions is reduced to $2$ depending on
whichever term in $\mathcal{B}$ is minimum. They are given by

\[
    h > \mc{O}\LRp{\frac{\sqrt{23\gamma d}(p+1)(p+2)}{2\sqrt{\lambda}}}
\quad \text{ or }\quad
    h > \mc{O}\LRp{\frac{\sqrt{24d}(p+1)(p+2)\gamma}{\sqrt{4\gamma-d}}}.
\]

We can see that the iHDG scheme as an iterative solver is
conditionally convergent for diffusion dominated PDEs, that is, the
mesh may not be too fine to violate the above condition on the
meshsize. This suggests that, for fine mesh problems, we can use the
iHDG algorithm as a fine scale preconditioner within a  Krylov subspace
method such as GMRES. In a forthcoming paper, we will construct a two level
preconditioner by combining iHDG with a multilevel coarse scale
preconditioner.

\section{Numerical results}
\seclab{numerics} In this section various numerical results supporting
the theoretical results are provided for the 2D and 3D transport
equations, the linearized shallow water equation, and the
convection-diffusion equation in different regimes.

\subsection{Steady state transport equation}
The goal is to verify Theorem \theoref{DDMConvergence} and Theorem
\theoref{iHDG-NPC} for the transport equation \eqnref{transport} in 2D
and 3D settings using the upwind HDG and the NPC fluxes.

\subsubsection{2D steady state transport equation with discontinuous solution}
We consider the case similar to the one in
\cite{Bui-Thanh15,MR2511736} where $f=0$
and $\betab=(1+sin(\pi y/2), 2)$ in \eqnref{transport}. The domain is
$[0,2]\times[0,2]$ and the inflow boundary conditions are given
by
\[
g = 
\left\{
\begin{array}{ll}
1 & x = 0, 0 \le y \le 2 \\
\sin^6\LRp{\pi x} & 0< x \le 1, y = 0 \\
0 & 1 \le x \le 2, y = 0
\end{array}
\right.
.
\]
To terminate the iHDG algorithm, we  use the following stopping  criteria
\begin{equation}
\eqnlab{tolerancecriteriarelerror}
\norm{u^k-u^{k-1}}_{L^2}<10^{-10},
\end{equation}
i.e., iHDG stops when there is  insignificant change between two successive iterations.

The evolution of the iterative solution  for the mesh with $1024$
elements and solution order $4$ using both upwind and NPC fluxes is shown in
figure \figref{2ddiscontiterfig2}.
In both cases, we observe that the iterative solution evolves from
inflow to outflow as the number of iterations increases. Thanks to the
built-in upwinding mechanism of the iHDG algorithm this implicit
marching is automatic, that is, we do not order the elements to march the flow direction. As can be
seen, iHDG with NPC flux converges faster (in fact in finite number of
iterations) as predicted by Theorem \theoref{iHDG-NPC}.
\begin{figure}[h!t!b!]
  \subfigure[$\uk$ at $k = 16$]{
    \includegraphics[trim=1.0cm 0.5cm 1.5cm 1cm,clip=true,width=0.3\columnwidth]{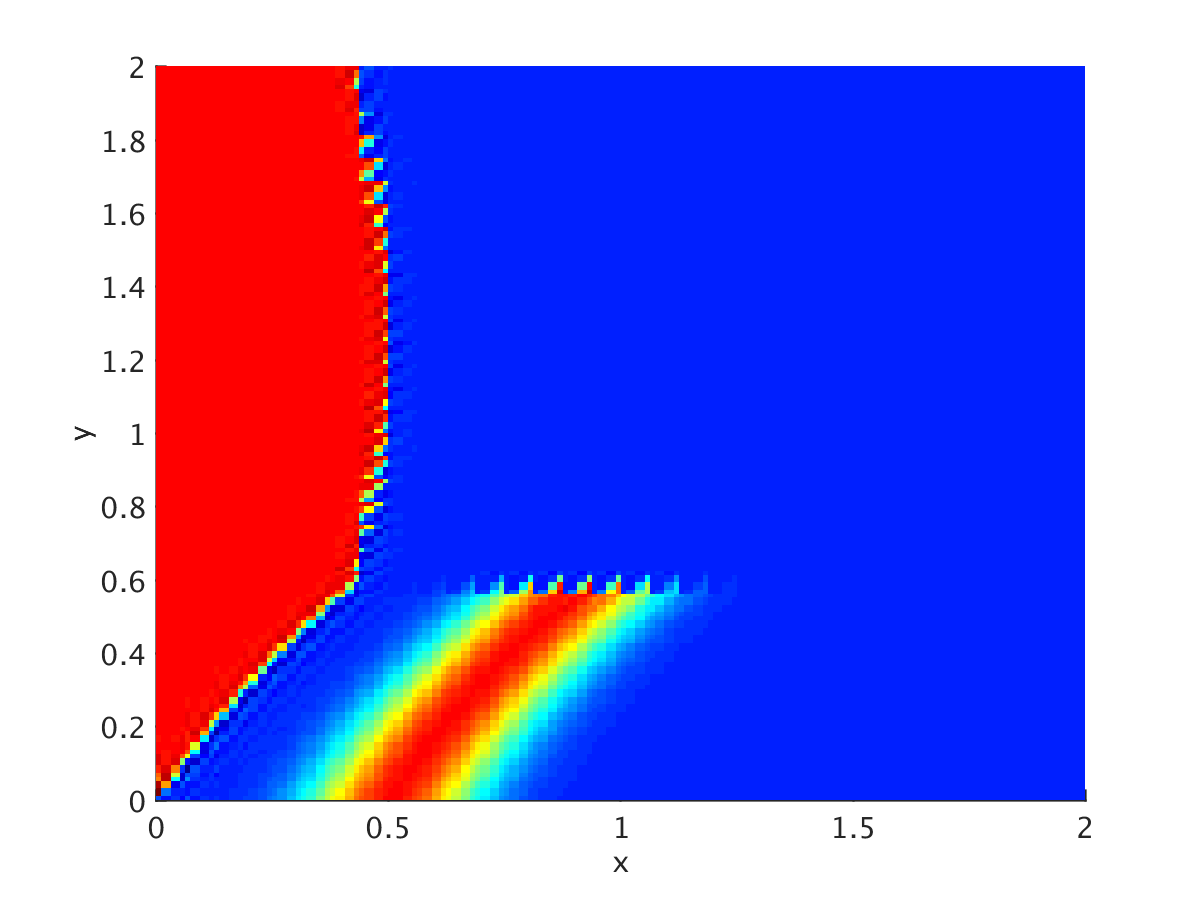}
  }
  \subfigure[$\uk$ at $k = 48$]{
    \includegraphics[trim=1.0cm 0.5cm 1.5cm 1cm,clip=true,width=0.3\columnwidth]{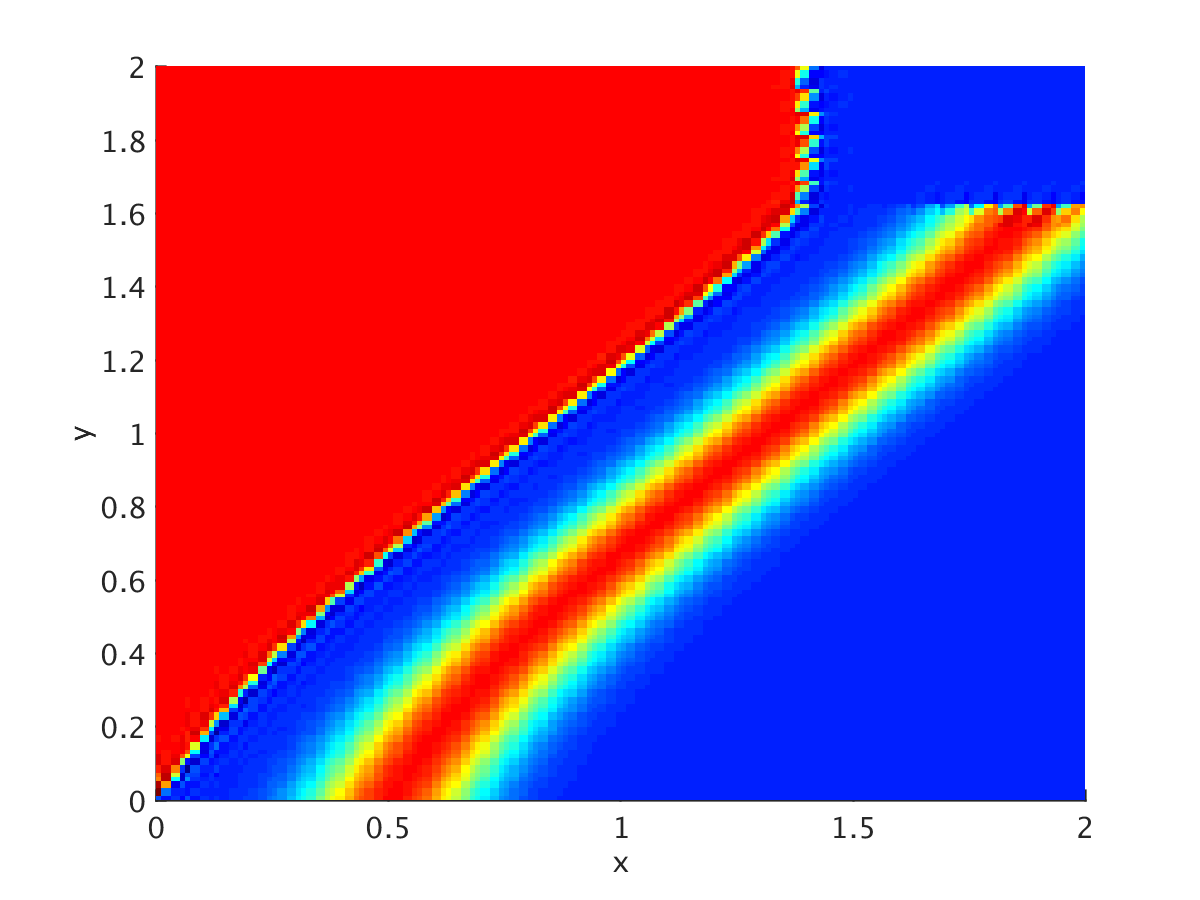}
  }
  \subfigure[$\uk$ at $k = 64$]{
    \includegraphics[trim=1.0cm 0.5cm 1.5cm 1cm,clip=true,width=0.3\columnwidth]{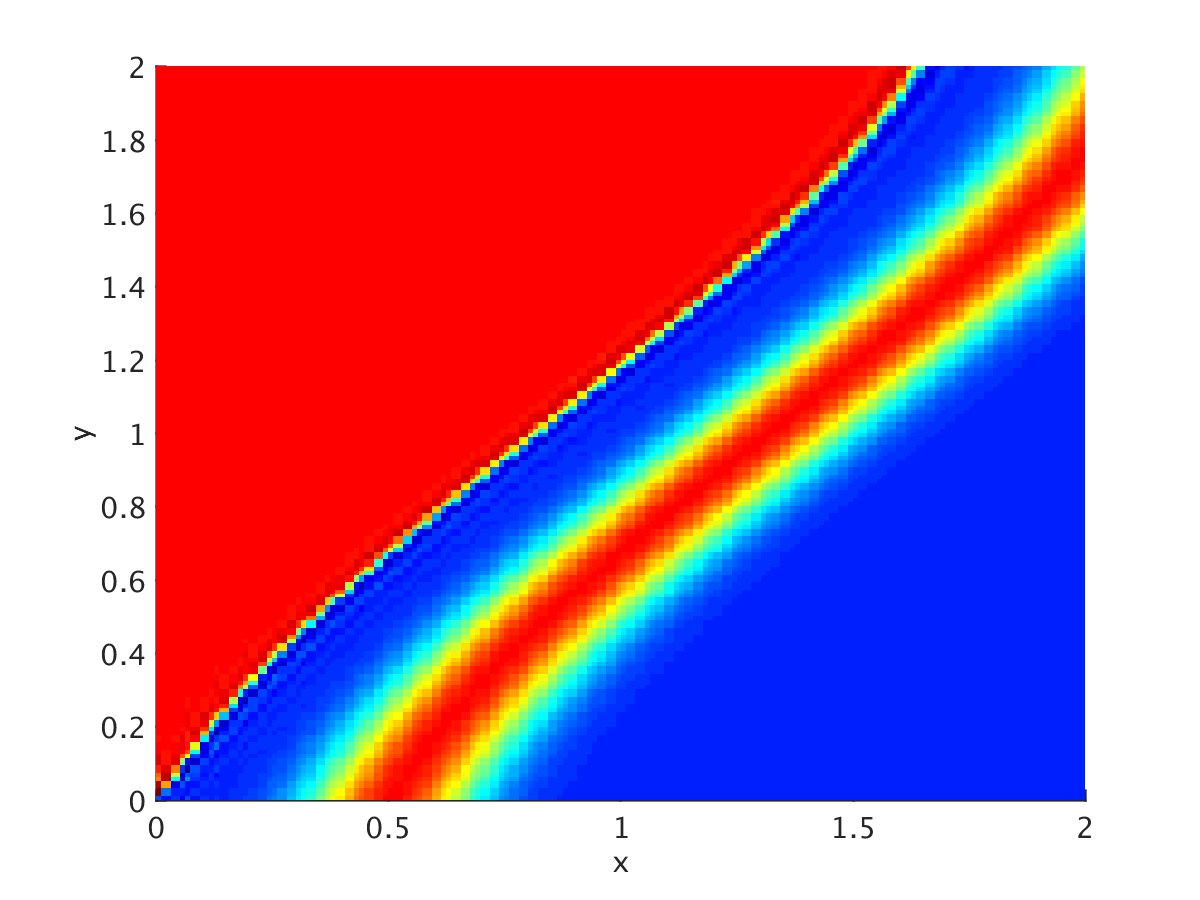}
  }
  \subfigure[$\uk$ at $k = 16$]{
    \includegraphics[trim=1.0cm 0.5cm 1.5cm 1cm,clip=true,width=0.31\columnwidth]{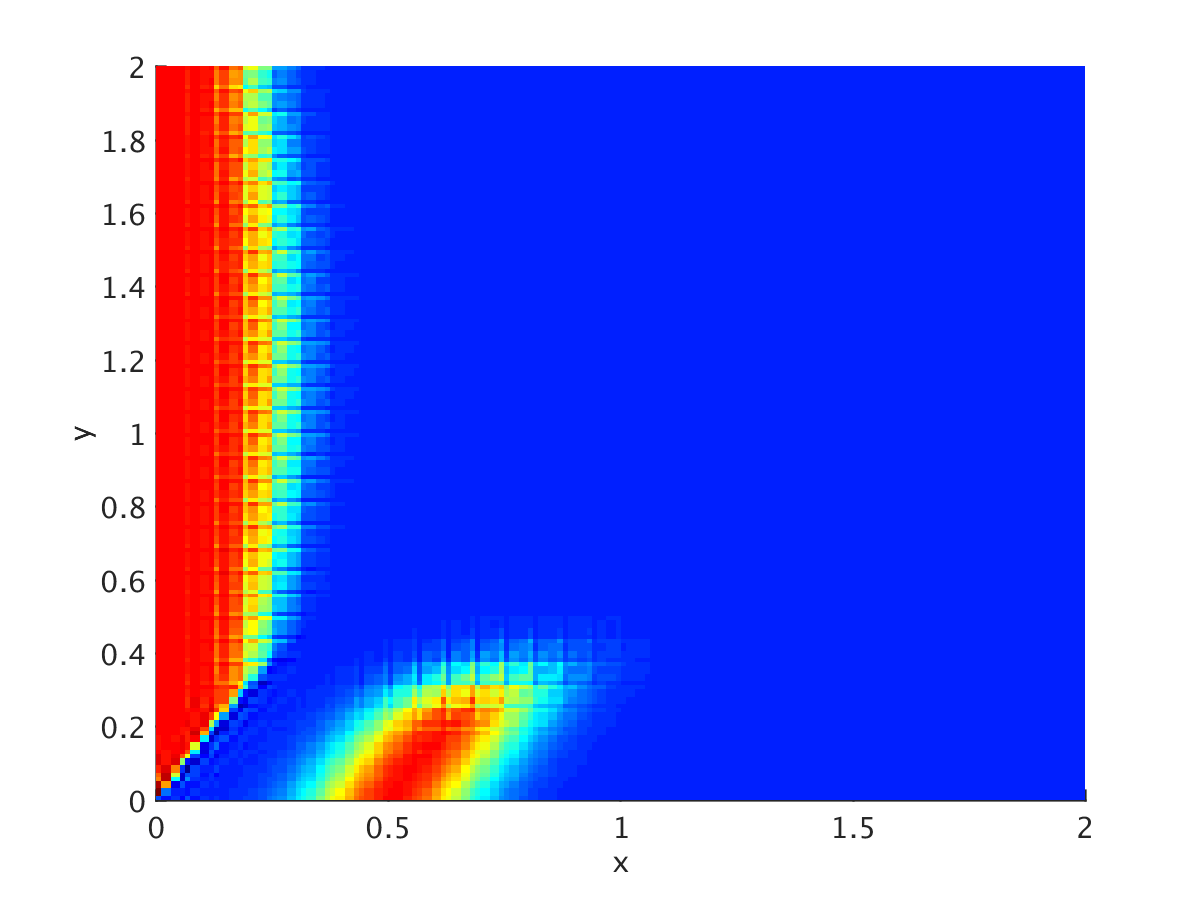}
  }
 \subfigure[$\uk$ at $k = 48$]{
    \includegraphics[trim=1.0cm 0.5cm 1.5cm 1cm,clip=true,width=0.31\columnwidth]{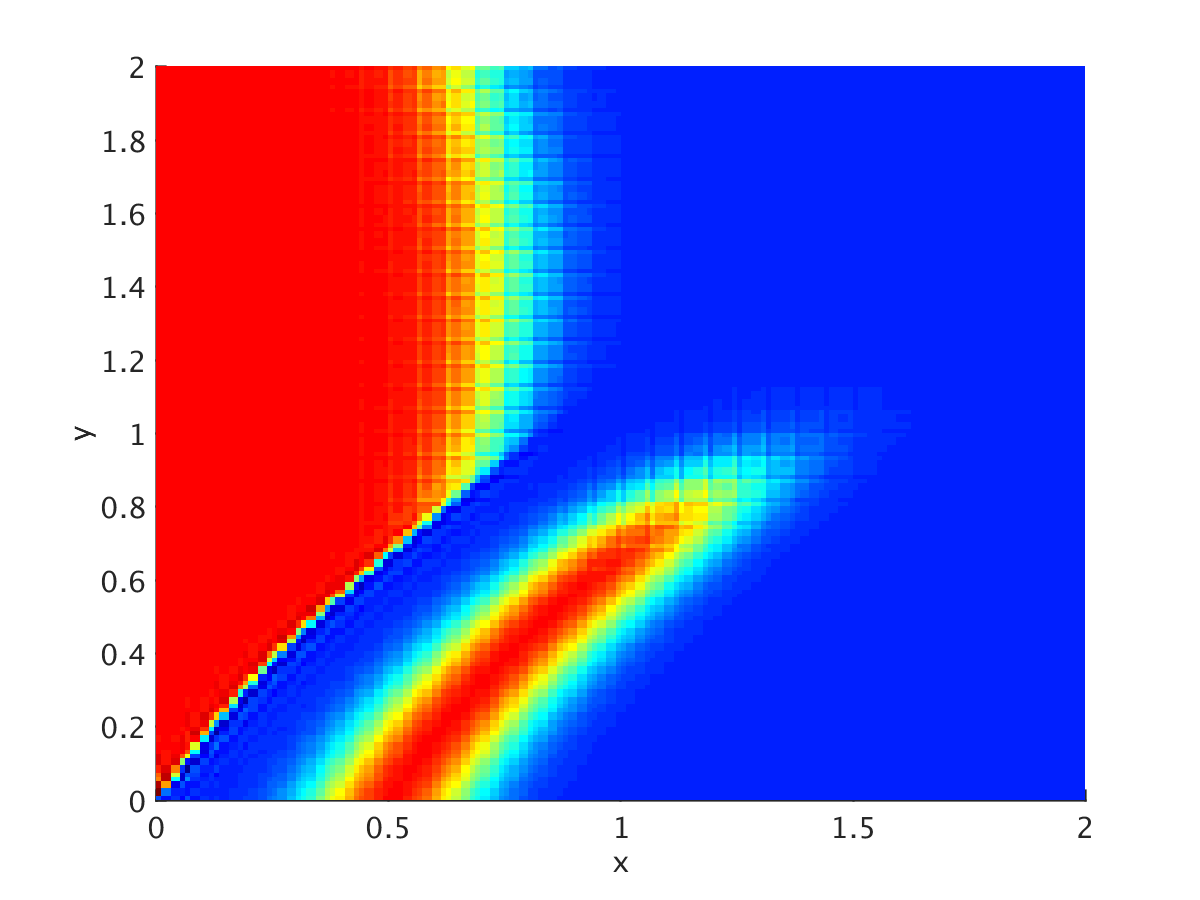}
  }
 \subfigure[$\uk$ at $k = 196$]{
    \includegraphics[trim=1.0cm 0.5cm 1.5cm 1cm,clip=true,width=0.31\columnwidth]{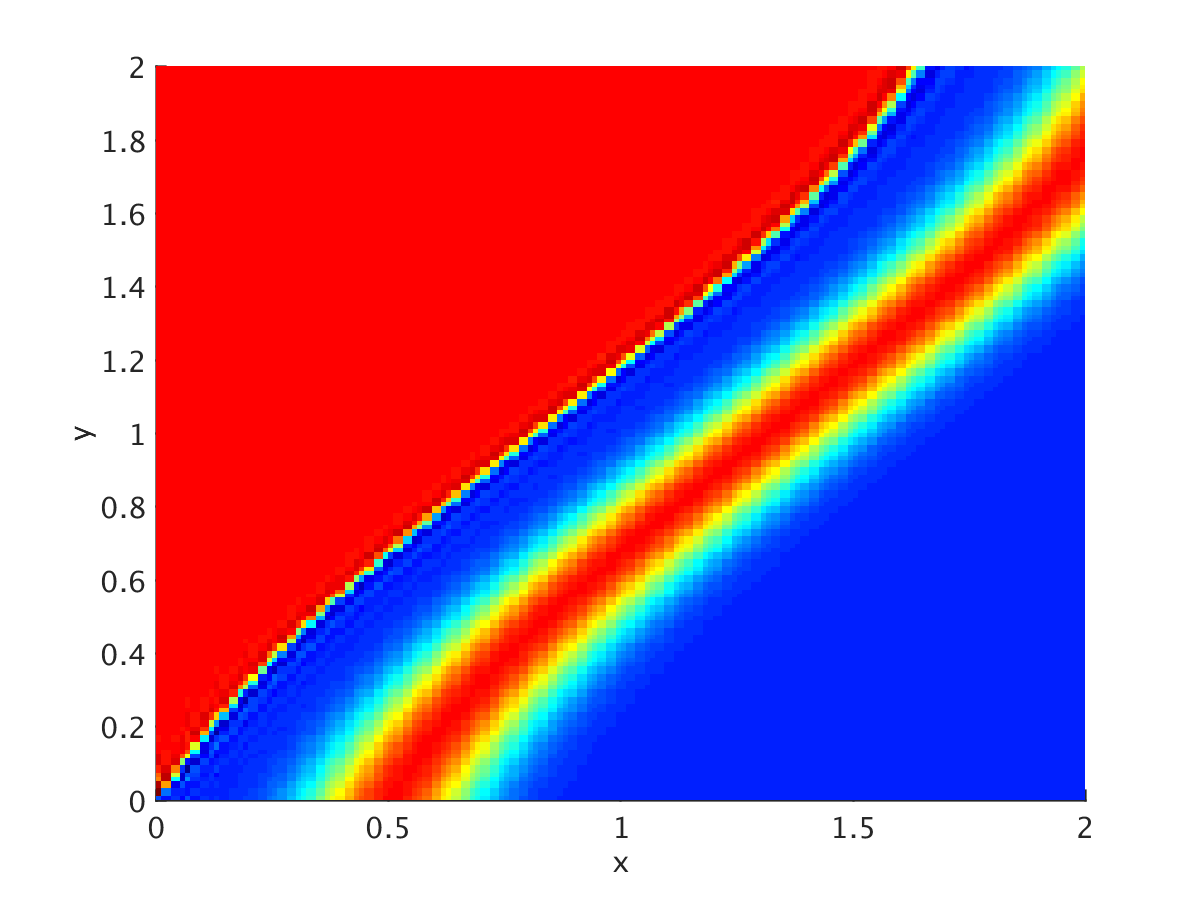}
  }
  \caption{Evolution of the iterative solution for the 2D transport equation using the NPC flux (top row) and the upwind HDG flux (bottom row).}
  \figlab{2ddiscontiterfig2}
\end{figure}

\begin{table}[h!b!t!]
\begin{center} 
\caption{The number of iterations taken by the
  iHDG algorithm using NPC and upwind HDG fluxes for the transport
  equation in 2D and 3D settings.}
\label{tab:2d_discont_3d_smooth}
\begin{tabular}{ | r || r || r || c | c | c | c | }
\hline
\multirow{2}{*}{$N_{el}$(2D)} & \multirow{2}{*}{$N_{el}$(3D)} & \multirow{2}{*}{$p$} & \multicolumn{2}{ |c| }{2D solution} & \multicolumn{2}{ |c| }{3D solution} \\
& & & Upwind & NPC & Upwind & NPC \\
\hline
16 & 8 & 3 & 65 & 9 & 39 & 7 \\
\hline
64 & 64 & 3 & 91 & 17 & 49 & 12 \\
\hline
256 & 512 & 3 & 133 & 33 & 79 & 23 \\
\hline
1024 & 4096 & 3 & 209 & 65 & 136 & 47 \\
\hline
\hline
16 & 8 & 4 & 65 & 9 & 35 & 6 \\
\hline
64 & 64 & 4 & 87 & 17 & 51 & 12 \\
\hline
256 & 512 & 4 & 129 & 33 & 83 & 24 \\
\hline
1024 & 4096 & 4 & 196 & 64 & 143 & 48 \\
\hline
\end{tabular}
\end{center}
\end{table}

The $4$th and $5$th columns of Table \ref{tab:2d_discont_3d_smooth}
show the number of iterations required to converge for both the fluxes
with different meshes and solution orders $3$ and $4$. We observe that the
number of iterations is (almost) independent of solution order\footnote{The results for $p=\LRc{1,2}$
are not shown, as the number of iterations is very similar that of the $p=\LRc{3,4}$-cases.} for
both the fluxes, which is in agreement with the theoretical results in
Theorems \theoref{DDMConvergence} and \theoref{iHDG-NPC}. {\em This is
important for high-order methods, i.e., the solution order (and hence
accuracy) can be increased while keeping the number of iHDG iterations unchanged}.

\subsubsection{3D steady state transport equation with smooth solution}
\label{3D_steady_advection}
In this example we choose $\betab=(z, x, y)$ in \eqnref{transport}. Also
we take the following exact solution
$$u^e=\frac{1}{\pi}\sin(\pi x)\cos(\pi y)\sin(\pi z).$$ The forcing is
selected in such a way that it corresponds to the exact solution.
Here the domain is $[0,1]\times[0,1]\times[0,1]$ with faces $x=0$,
$y=0$ and $z=0$ as the inflow boundaries. A structured hexahedral mesh is
used for the simulations.  Since we know the exact solution we use the
following stopping criteria:
\begin{equation}
\eqnlab{tolerancecriteria}
|\norm{u^k-u^e}_{\Ltwo}-\norm{u^{k-1}-u^e}_{\Ltwo}|<10^{-10}.
\end{equation}

\begin{figure}[h!t]
\vspace{-5mm}
\subfigure[Upwind flux]{
\includegraphics[trim=2.75cm 8cm 4cm 9cm,clip=true,width=0.5\textwidth]{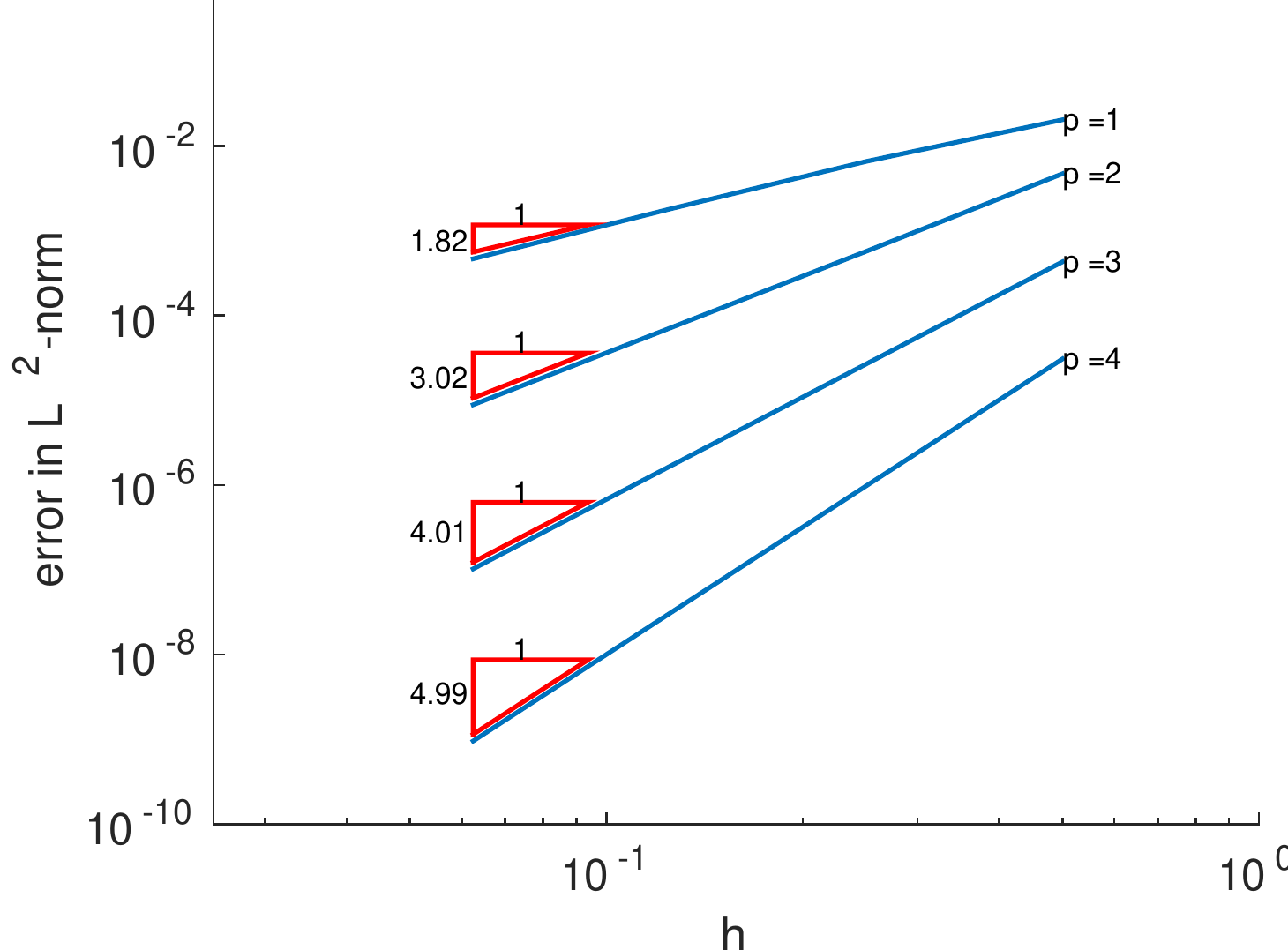}
\label{upwind_3d_advec_steady}
}
\subfigure[NPC flux]{
\includegraphics[trim=2.75cm 8cm 4cm 9cm,clip=true,width=0.5\textwidth]{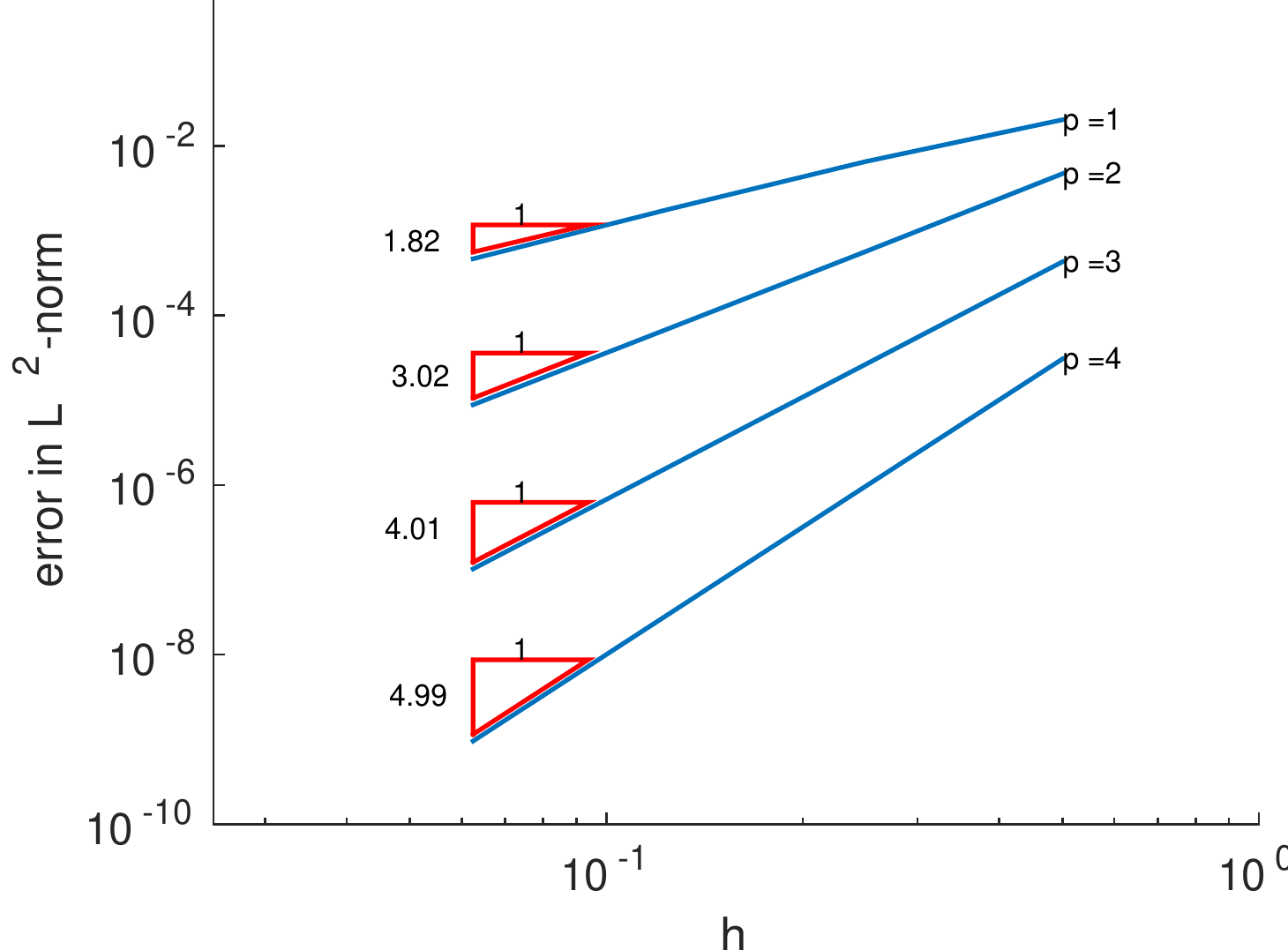}
\label{Nguyen_3d_advec_steady}
}
\caption{ h-convergence of the HDG method using iHDG with upwind and NPC fluxes. } 
\label{l2err_3dsteady}
\end{figure}

\begin{figure}[h!t]
\vspace{-12mm}
\includegraphics[trim=1cm 3cm 2cm 4cm,clip=true,width=0.5\textwidth]{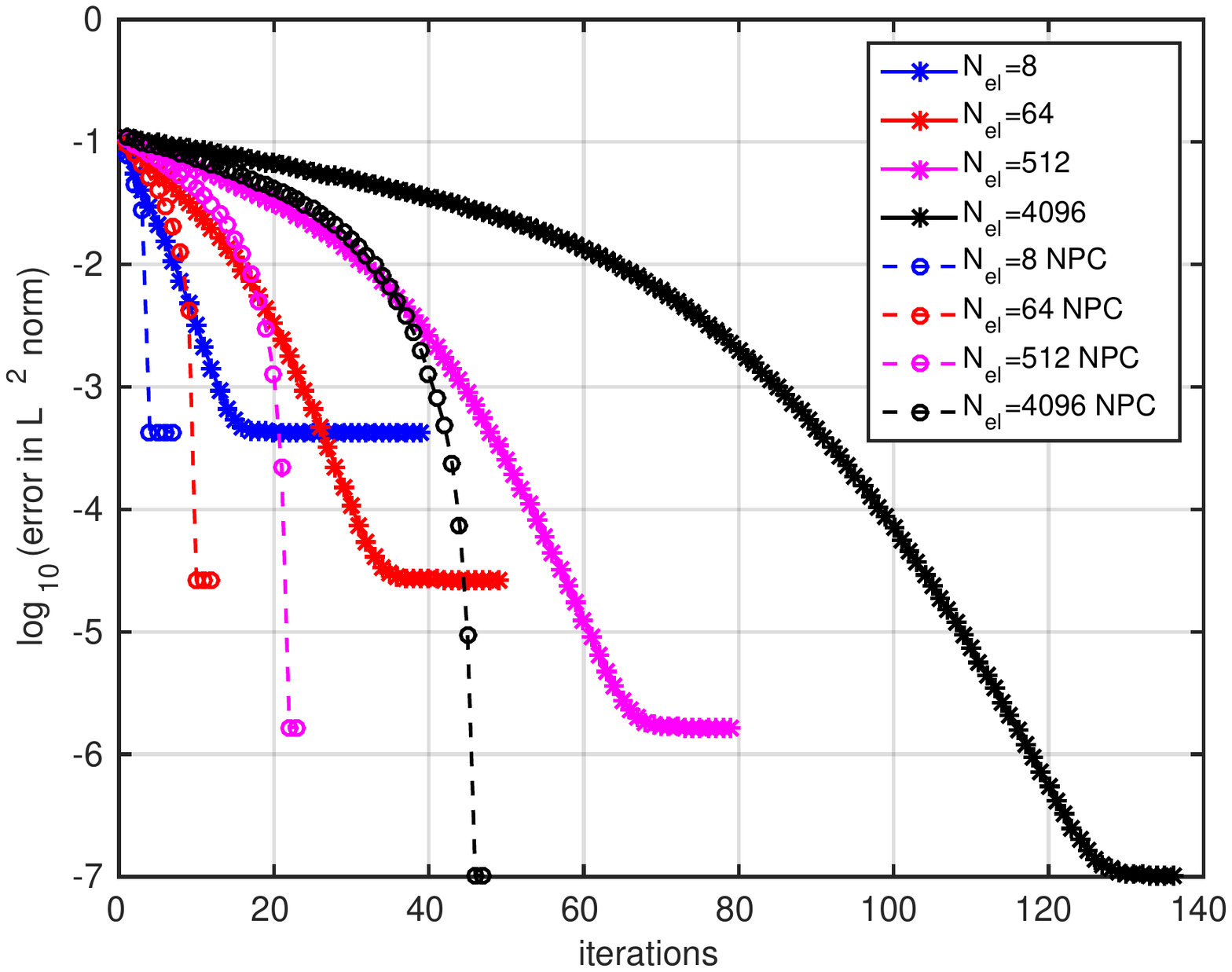}
\includegraphics[trim=1cm 3cm 2cm 4cm,clip=true,width=0.5\textwidth]{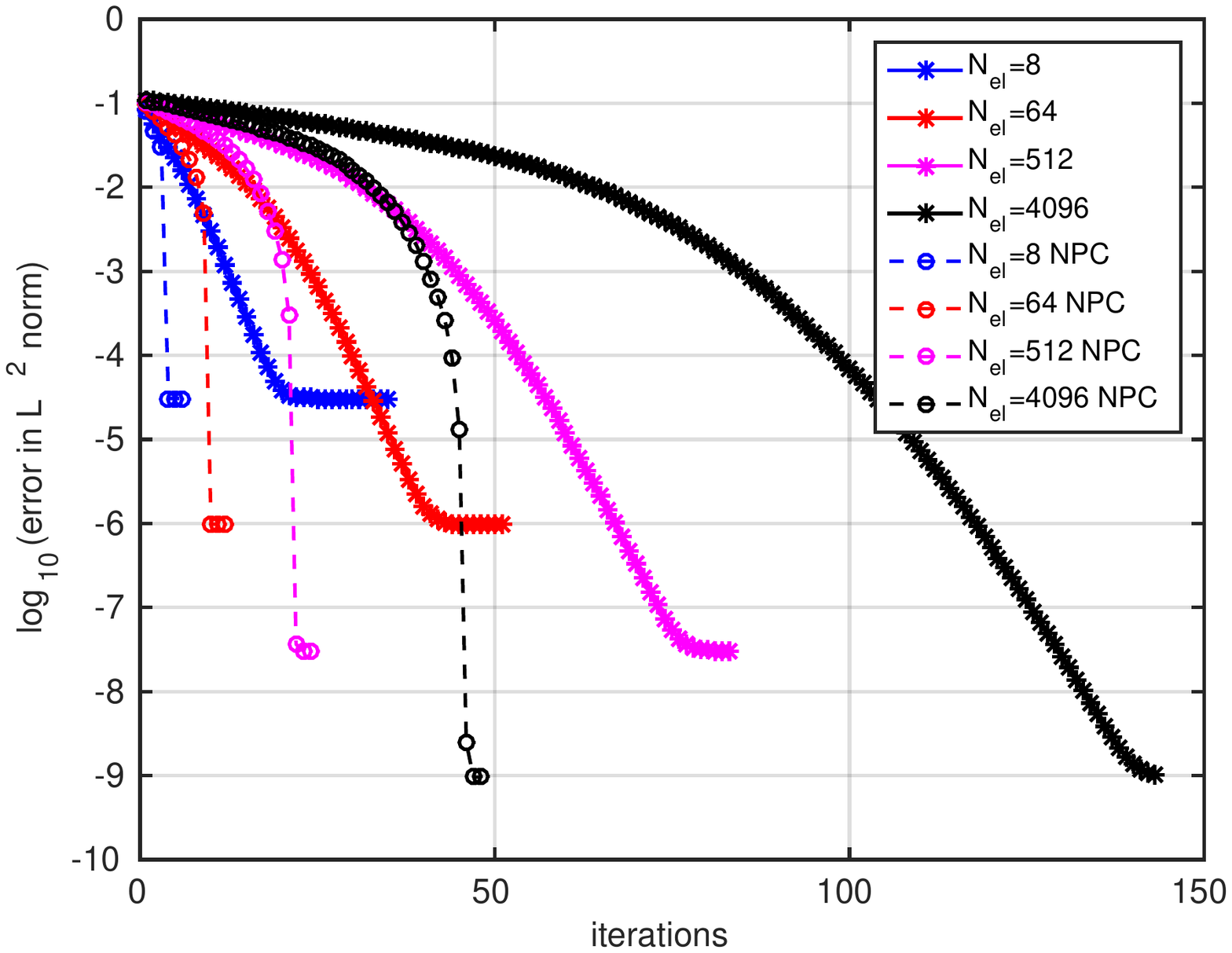}
\vspace{-15mm}
\caption{Error history in terms of the number iterations for
solution order $p=3$ (left)  and $p = 4$ (right) as the mesh is refined (the number of elements $\Nel$ increases).} 
\label{p3p4_3dsteady}
\end{figure}

\begin{figure}[h!t!b!]
  \subfigure[$\u$ at $iteration = 1$]{
    \includegraphics[trim=4cm 0cm 3cm 0cm,clip=true,width=0.31\columnwidth]{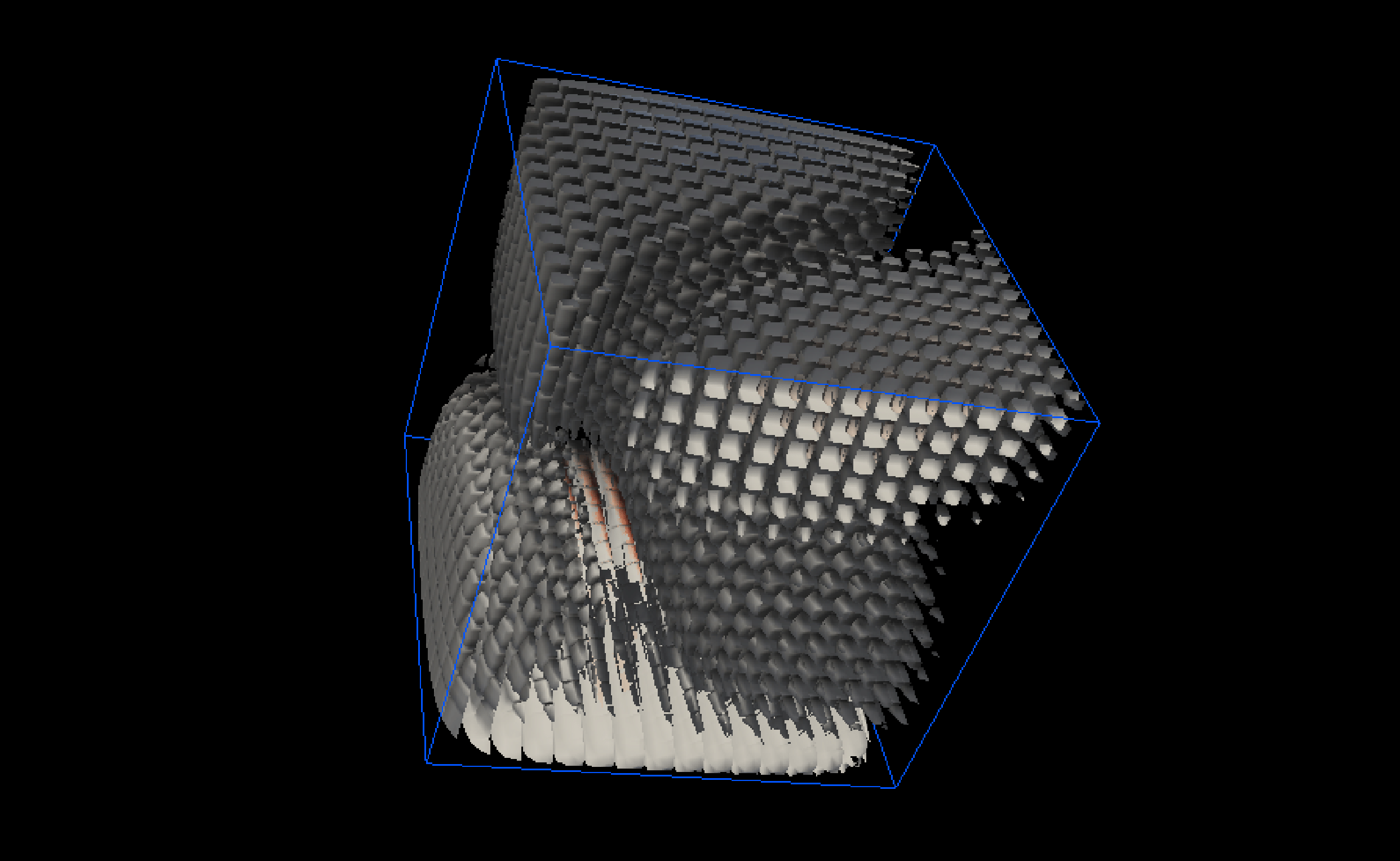}
  }
 \subfigure[$\u$ at $iteration = 32$]{
    \includegraphics[trim=4cm 0cm 3cm 0cm,clip=true,width=0.31\columnwidth]{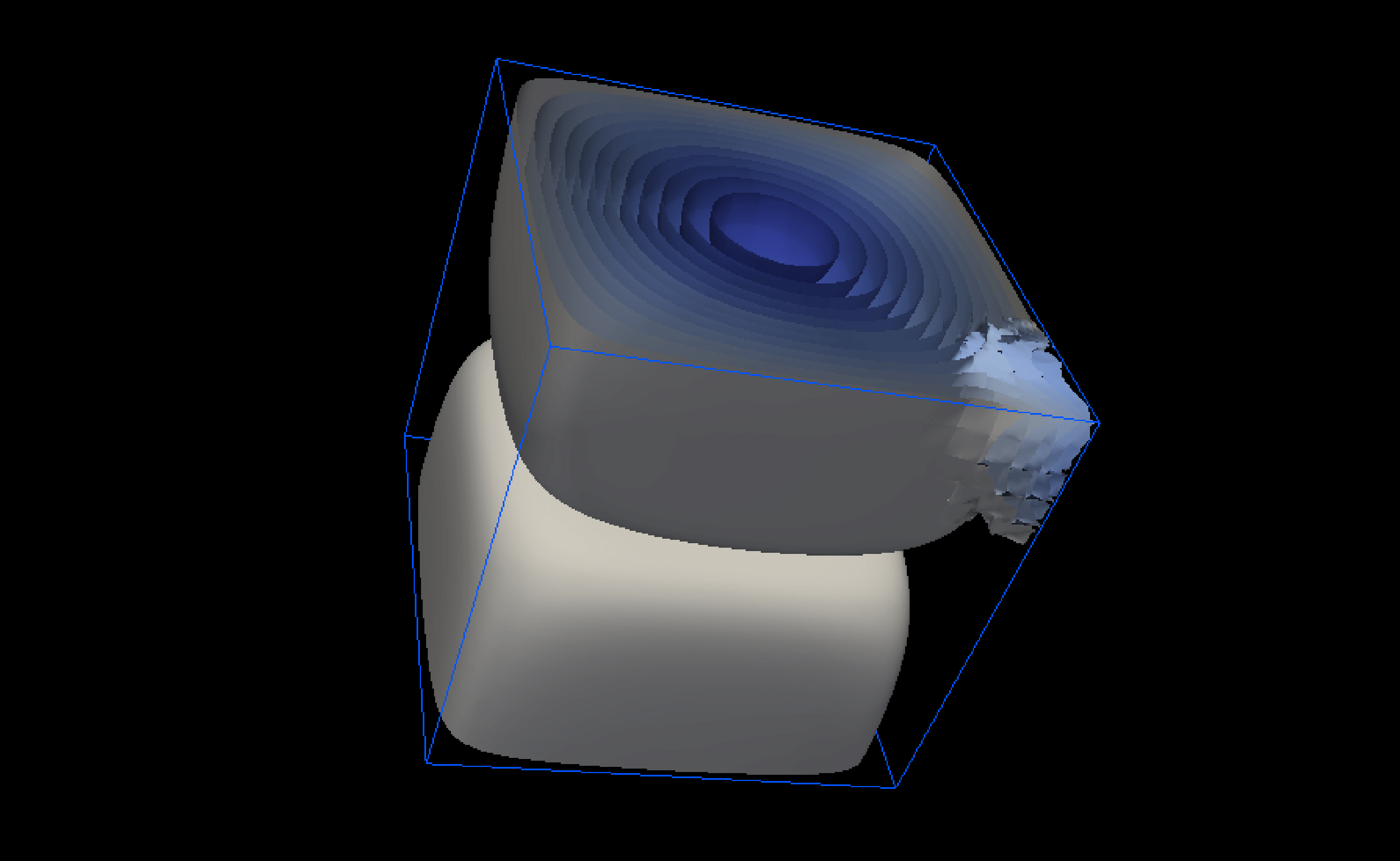}
  }
 \subfigure[$\u$ at $iteration = 48$]{
    \includegraphics[trim=4cm 0cm 3cm 0cm,clip=true,width=0.31\columnwidth]{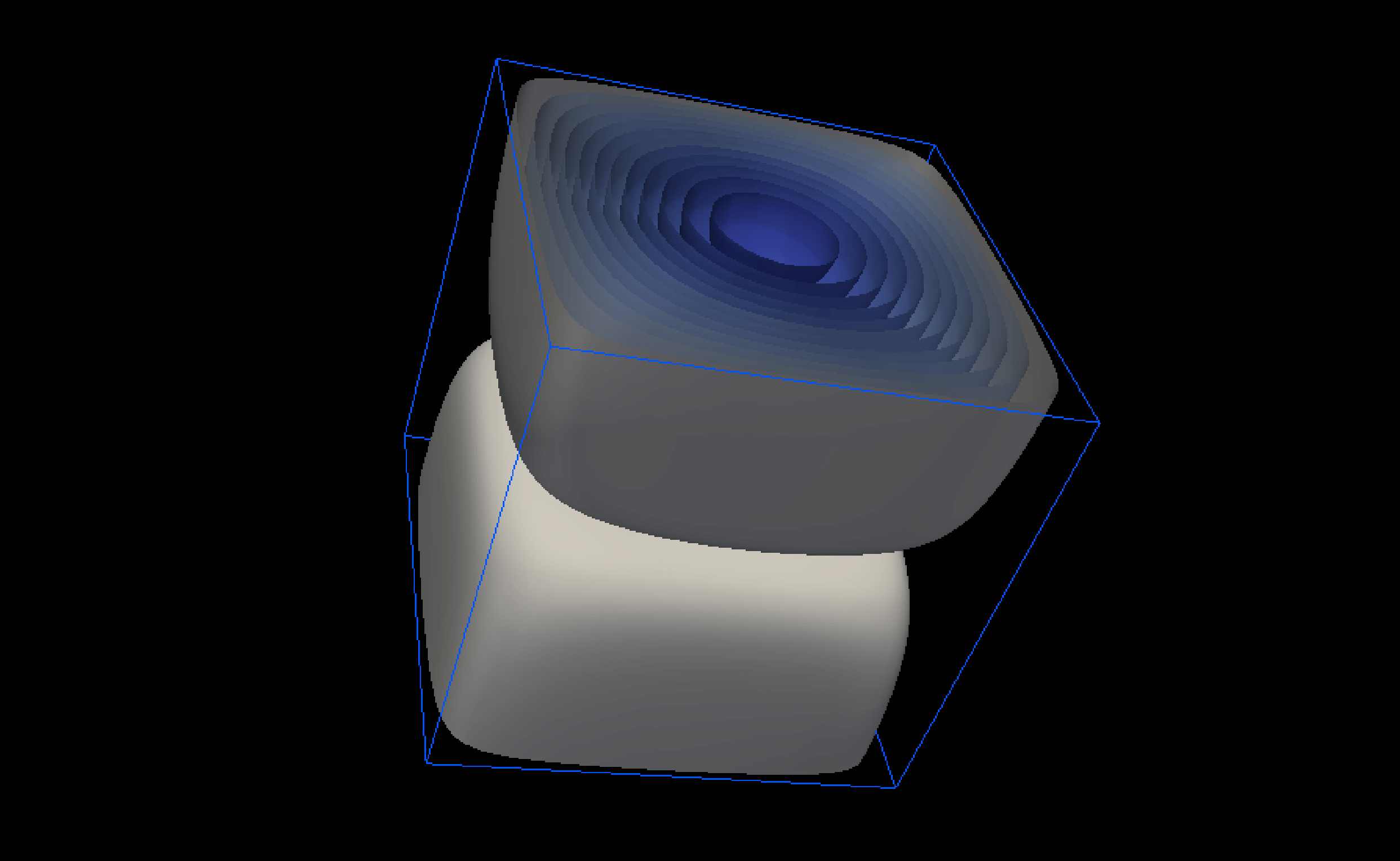}
  }
  \subfigure[$\u$ at $iteration = 1$]{
    \includegraphics[trim=4cm 0cm 3cm 0cm,clip=true,width=0.31\columnwidth]{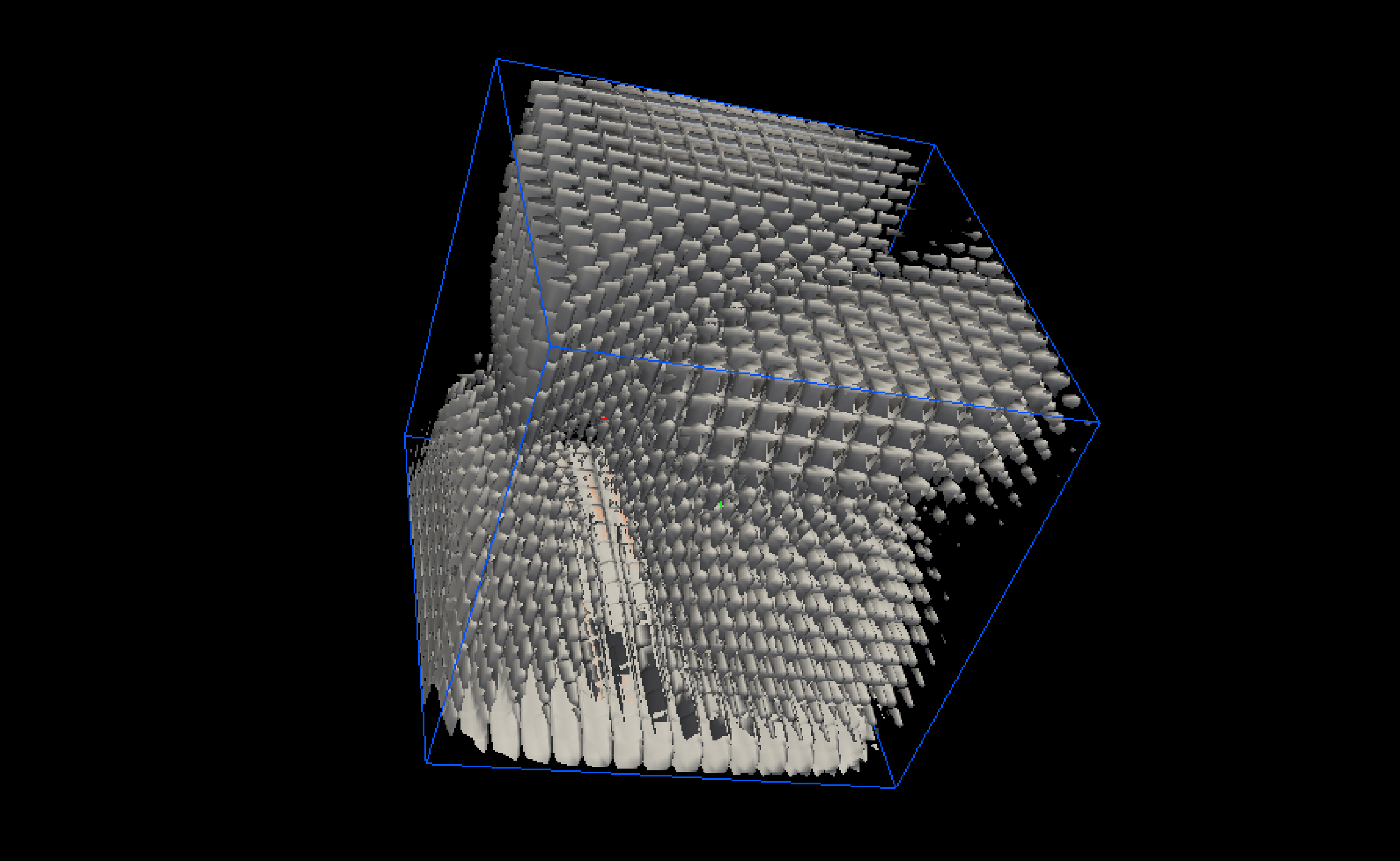}
  }
   \subfigure[$\u$ at $iteration = 32$]{
    \includegraphics[trim=4cm 0cm 3cm 0cm,clip=true,width=0.31\columnwidth]{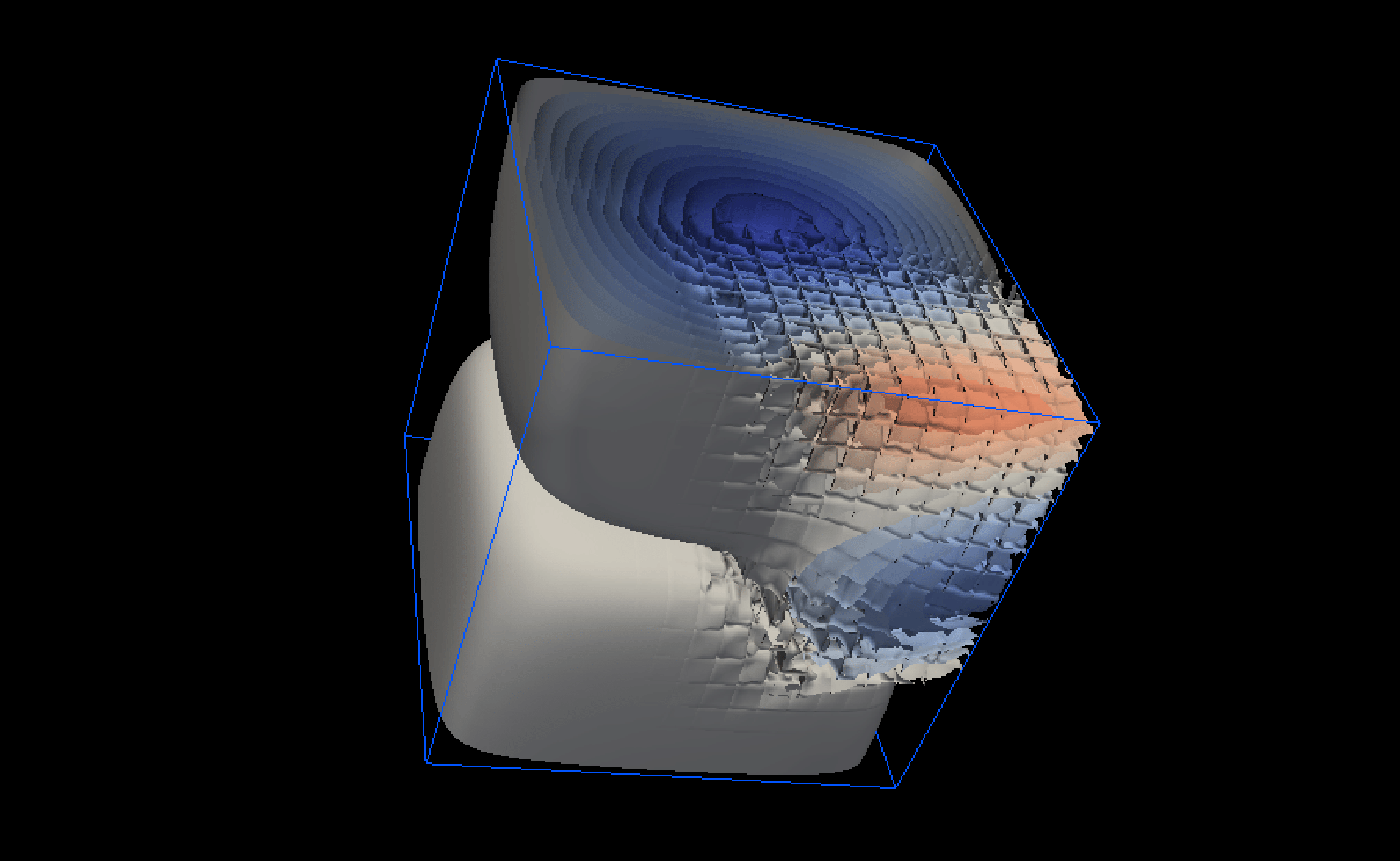}
  }
  \subfigure[$\u$ at $iteration = 143$]{
    \includegraphics[trim=4cm 0cm 3cm 0cm,clip=true,width=0.31\columnwidth]{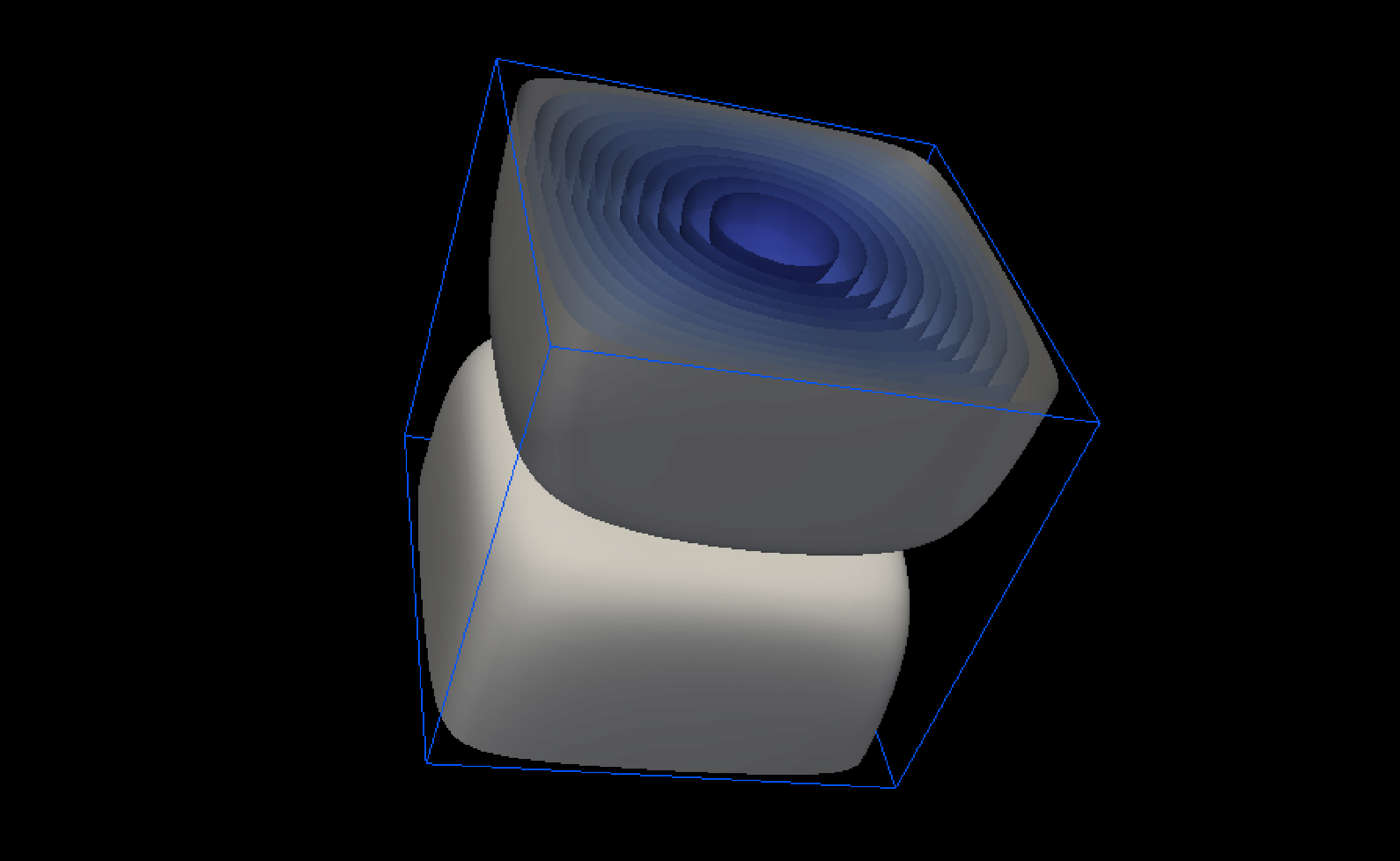}
  }
 \caption{Evolution of the iHDG solution in terms of the number of iterations for the NPC flux (top row) and  the upwind HDG flux (bottom row).}
 \vspace{-7mm}
  \figlab{3diterfig2}
\end{figure}

Figure \ref{l2err_3dsteady} shows the $h$-convergence of the HDG
discretization with the iHDG iterative solver. The convergence is optimal
with rate $(p+1)$ for both fluxes. Figure \ref{p3p4_3dsteady} compares the convergence history of the
iHDG solver in the log-linear scale.  As proved in Theorem
\theoref{DDMConvergence} the iHDG with upwind flux is exponentially
convergent with respect to the number of iterations $k$, while the convergence
is attained in finite number of iterations for the NPC flux as predicted
in Theorem \theoref{iHDG-NPC}. Note that the stagnation region
observed near the end of each curve is due to the fact that for a
particular mesh size $h$ and solution order $p$ we can achieve only
as much accuracy as prescribed by the HDG discretization error and
cannot go beyond that. Numerical results for different solution
orders\footnote{Here the results for $p=\LRc{1,2}$ are omitted for brevity.} also verify the fact that the convergence of the iHDG algorithm is
independent of the solution order $p$. The evolution of the iHDG
solution in terms of the number of iterations is shown in Figure
\figref{3diterfig2}. Again, for the scalar transport equation, iHDG
automatically marches the solution from the inflow to the outflow. We
also record in the $6$th and $7$th columns of Table
\ref{tab:2d_discont_3d_smooth} the number of iterations that the iHDG
algorithm took for both the fluxes. As predicted by our theoretical findings,
the number of
iterations is independent of the solution order.

\subsection{Linearized shallow water equations}
In this section we consider equation \eqnref{linearizedShallow} with a linear
standing wave, for which, we set $\Phi=1$, $f=0$, $\gamma=0$ (zero bottom friction),
$\taub=0$ (zero wind stress). The domain is $[0, 1]\times[0, 1]$ and
wall boundary condition is applied on the domain boundary. 
The following exact
solution \cite{GiraldoWarburton08} is taken
\begin{subequations}
\eqnlab{shallowexact}
\begin{align}
	\phi^e&=\cos(\pi x)\cos(\pi y)\cos(\sqrt{2} \pi t), \\
	u^e&=\frac{1}{\sqrt{2}}\sin(\pi x)\cos(\pi y)\sin(\sqrt{2} \pi t), \\
	v^e&=\frac{1}{\sqrt{2}}\cos(\pi x)\sin(\pi y)\sin(\sqrt{2} \pi t). 
\end{align}
\end{subequations}

\begin{figure}[h!b!t!]
\begin{minipage}{0.5\linewidth}
\centering
\includegraphics[trim=2.5cm 7.75cm 5.5cm 9.25cm,clip=true,width=0.7\textwidth]{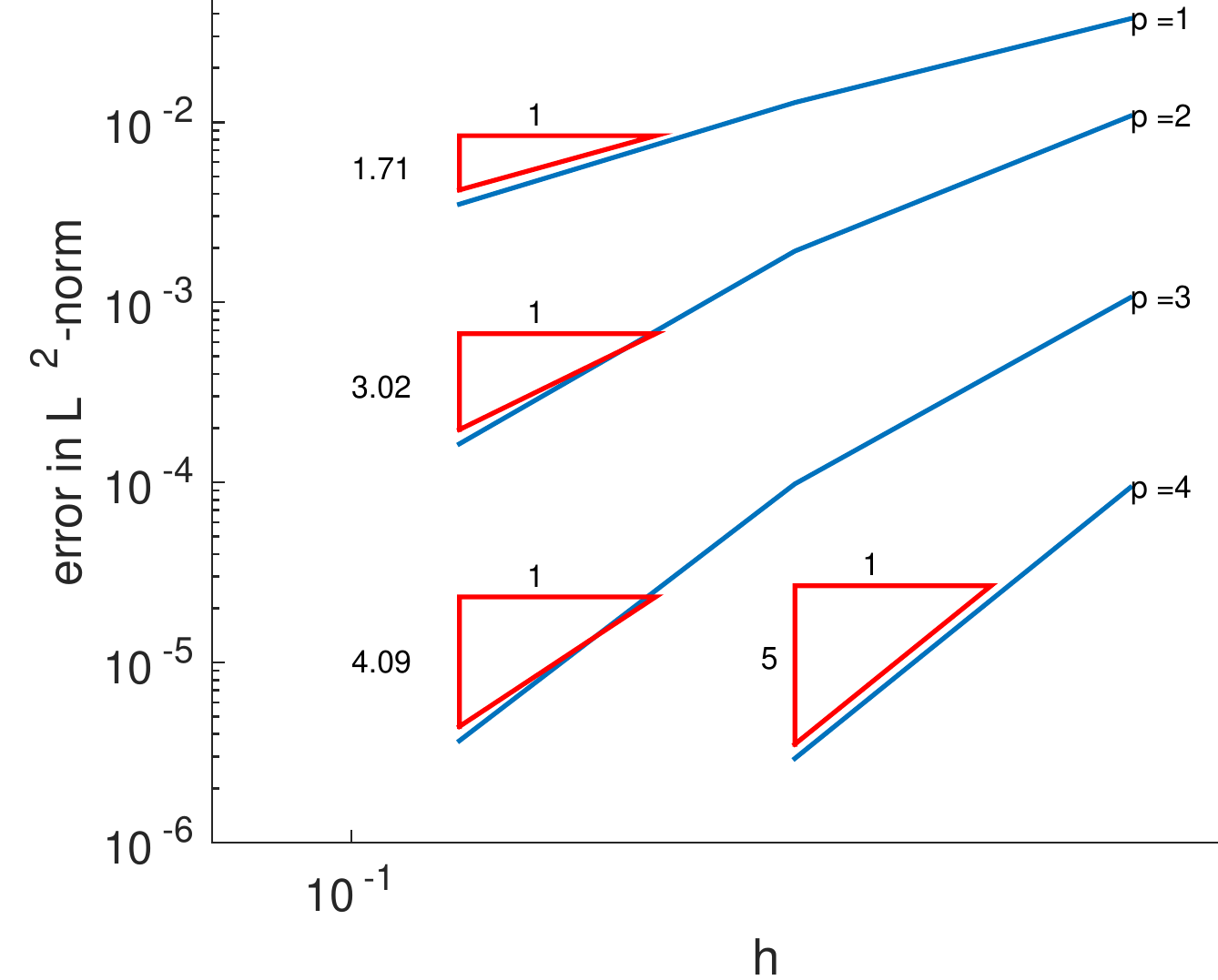}
\end{minipage}
\begin{minipage}{0.5\linewidth}
\begin{center}
\begin{tabular}{ | r || c | c | c | c | }
\hline
\multirow{2}{*}{$N_{el}$} & \multicolumn{4}{ |c| }{Solution order } \\
& $1$ & $2$ & $3$ & $4$ \\
\hline
\hline
16 & 12 & 13 & 9 & 10 \\
\hline
64 & 11 & 12 & 7 & 9 \\ 
\hline
256 & 9 & 11 & 7 & 8 \\
\hline
1024 & 7 & 10 & 7 & 7 \\
\hline
\end{tabular}
\end{center}
\end{minipage}
\caption{h-convergence of iHDG for $10^{5}$ time steps with $\Delta t=10^{-6}$ (left) and the number of iHDG iterations per time step with $\Delta t=\frac{h}{(p+1)(p+2)}$ (right) for the linearized shallow water equation.}\label{tab:Comparison_shallow_water}
\vspace{-7mm}
\end{figure}
The iHDG algorithm with upwind HDG flux described in Section
\secref{iHDG-hyperbolic} along with the Crank-Nicolson method for time discretization is employed in this problem.  The convergence
of the solution is presented in Figure \ref{tab:Comparison_shallow_water}. Here
we have taken $\Delta t=10^{-6}$ as the stepsize with $10^5$ steps. As
can be seen, the optimal convergence rate of $(p+1)$ is attained. The
number of iterations required per time step in this case is constant
and is always equal to $2$ for all meshes and solution orders
considered. The reason is due to i) a warm-start strategy, that is, the
initial guess for each time step is taken as the solution of the
previous time step, and ii) small time stepsize.

Following remark \ref{Shallow-water-remark} we choose $\Delta
t=\frac{h}{(p+1)(p+2)}$ and report the number of iterations for
different meshes and solution orders in Figure
\ref{tab:Comparison_shallow_water}. Clearly, finer mesh and higher
solution order require smaller time stepsize, and hence less number
of iterations, for the iHDG algorithm to converge.

\subsection{Convection-Diffusion Equation}
\seclab{Convection-diffusion-steady}
In this section equation \eqnref{Convection-diffusion-eqn} is considered with
the exact solution
taken as
$$u^e=\frac{1}{\pi}\sin(\pi x)\cos(\pi y)\sin(\pi z).$$ The forcing is
chosen such that it corresponds to the exact solution. The domain is
same as the one in section \ref{3D_steady_advection}. Dirichlet
boundary condition based on the exact solution is applied on the
boundary faces and the stopping criteria is same as
\eqnref{tolerancecriteria}.
\subsubsection{Convection dominated regime}
Let us consider $10^{-3}\le\kappa\le10^{-6}$, $\nu=1$, and $\betab=(1+z, 1+x, 1+y)$. Since the maximum
velocity in this example is $\mc{O}(1)$, this represents convection
dominated regime. 
Figure \ref{l2err_epm3_epm6_upwind} shows the optimal h-convergence of the iHDG method
with the upwind HDG flux\footnote{The convergence with NPC flux is similar and hence not shown.} 
for $\kappa=10^{-3}$ and $\kappa=10^{-6}$, respectively. The error history 
for solution orders $p=\LRc{3,4}$ is given
in Figure \ref{CDR_epm3_epm6_p3p4}. As expected, for
$\kappa=10^{-6}$, the iHDG method with either upwind or NPC flux behaves
similar to the pure convection case. From table \ref{tab:Comparison_different_kappa} 
the convergence of the upwind iHDG approach remains the same, that is, the number of
iterations is insensitive to the diffusion coefficient $\kappa$. The iHDG approach
with the NPC flux improves as $\kappa$ decreases. This is because the stabilization ($\tau$)
of the NPC flux contains $\kappa$ whereas the stabilization of the upwind flux does not.
\begin{figure}[h!b!t!]
\vspace{-19mm}
\subfigure[$\kappa=10^{-3}$]{
\includegraphics[trim=3.4cm 8cm 4cm 5cm,clip=true,width=0.5\textwidth]{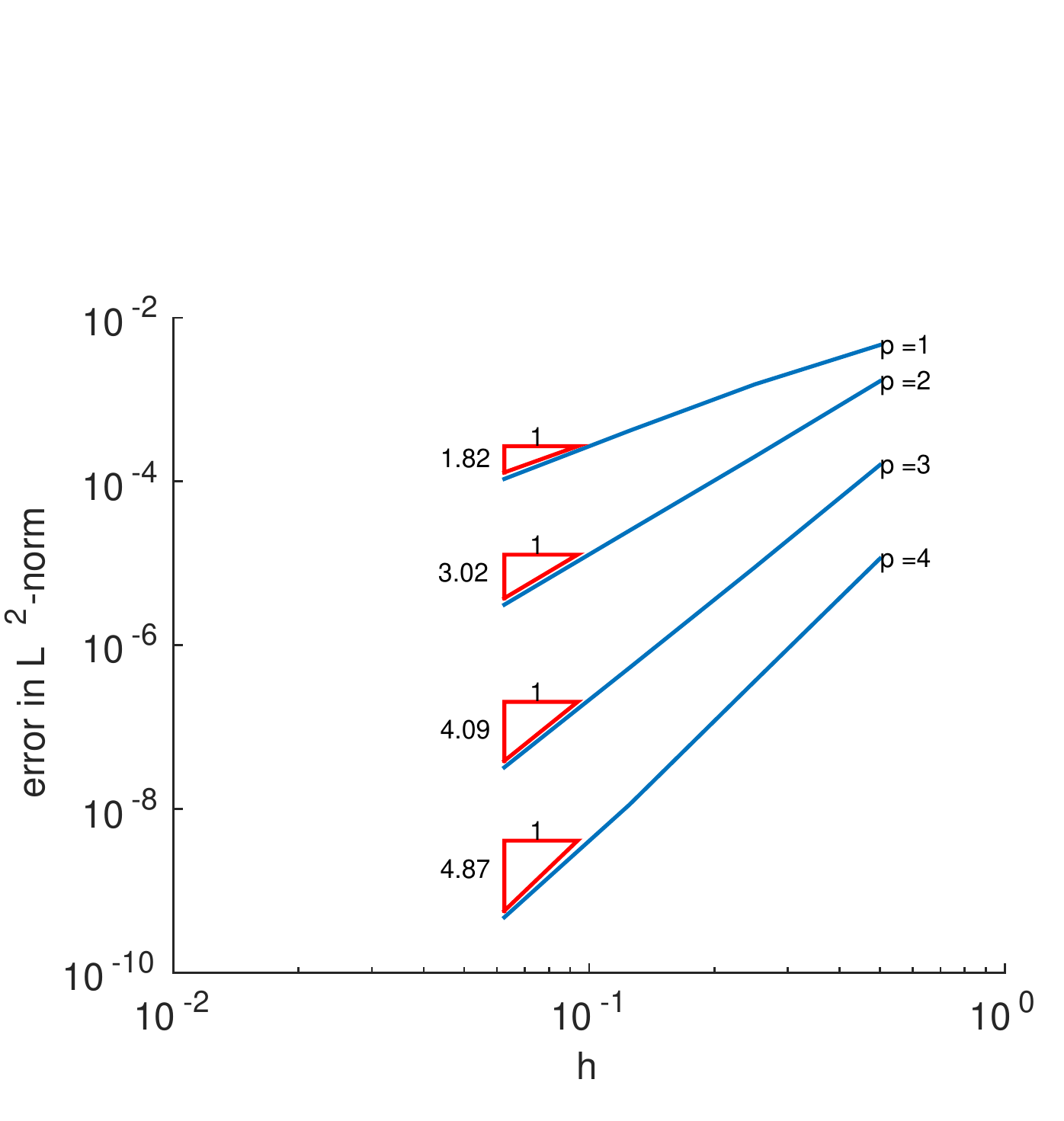}
\label{l2err_epm3-upwind}
}
\subfigure[$\kappa=10^{-6}$]{
\includegraphics[trim=3.2cm 8cm 4cm 5cm,clip=true,width=0.5\textwidth]{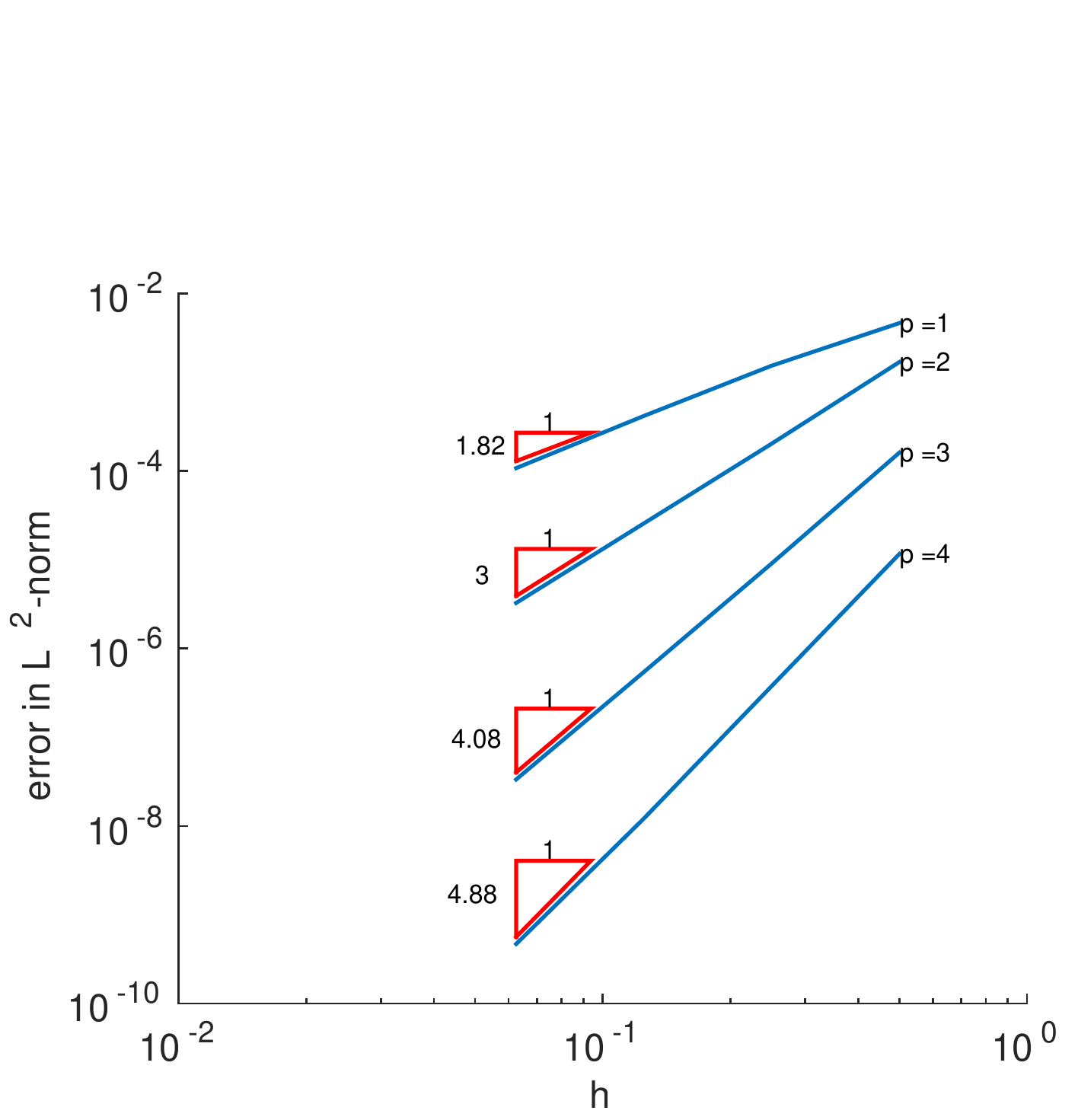}
\label{l2err_epm6-upwind}
}
\caption{h-convergence of iHDG method with upwind HDG flux for $\kappa=\LRc{10^{-3}, 10^{-6}}$.} 
\label{l2err_epm3_epm6_upwind}
\end{figure}

\begin{table}[h!b!t!]
\begin{center}
\caption{The number of iHDG iterations for various $\kappa$ with upwind and NPC fluxes.} 
\label{tab:Comparison_different_kappa}
\begin{tabular}{ | r || r || c | c | c | c | c | c | }
\hline
\multirow{2}{*}{$h$} & \multirow{2}{*}{$p$} & \multicolumn{3}{|c|}{Upwind flux} & \multicolumn{3}{|c|}{NPC flux} \\
& & $\kappa=10^{-2}$ & $\kappa=10^{-3}$ & $\kappa=10^{-6}$ & $\kappa=10^{-2}$ & $\kappa=10^{-3}$ & $\kappa=10^{-6}$ \\
\hline
0.5 & 1 & 24 & 23 & 23 & 10 & 7 & 5 \\
\hline
0.25 & 1 & 30 & 34 & 35 & 21 & 14 & 12 \\
\hline
0.125 & 1 & 50 & 55 & 56 & 94 & 26 & 22 \\
\hline
0.0625 & 1 & 90 & 94 & 97 & * & 52 & 46 \\
\hline
\hline
0.5 & 2 & 26 & 24 & 25 & 13 & 8 & 5 \\
\hline
0.25 & 2 & 41 & 42 & 42 & 22 & 15 & 12 \\
\hline
0.125 & 2 & 66 & 67 & 67 & * & 26 & 23 \\
\hline
0.0625 & 2 & * & 109 & 110 & * & 59 & 46 \\
\hline
\hline
0.5 & 3 & 27 & 31 & 31 & 21 & 8 & 5 \\
\hline
0.25 & 3 & 33 & 33 & 38 & * & 16 & 12 \\
\hline
0.125 & 3 & * & 58 & 60 & * & 30 & 24 \\
\hline
0.0625 & 3 & * & 102 & 106 & * & 59 & 48 \\
\hline
\hline
0.5 & 4 & 26 & 27 & 27 & 73 & 8 & 5 \\
\hline
0.25 & 4 & 50 & 41 & 43 & * & 16 & 12 \\
\hline
0.125 & 4 & * & 71 & 72 & * & 32 & 24 \\
\hline
0.0625 & 4 & * & 123 & 125 & * & 71 & 48 \\
\hline
\end{tabular}
\end{center}
\end{table}

\begin{figure}[h!t!b!]
\subfigure[Error history for p=3]{
\includegraphics[trim=3cm 7cm 3cm 8cm,clip=true,width=0.5\textwidth]{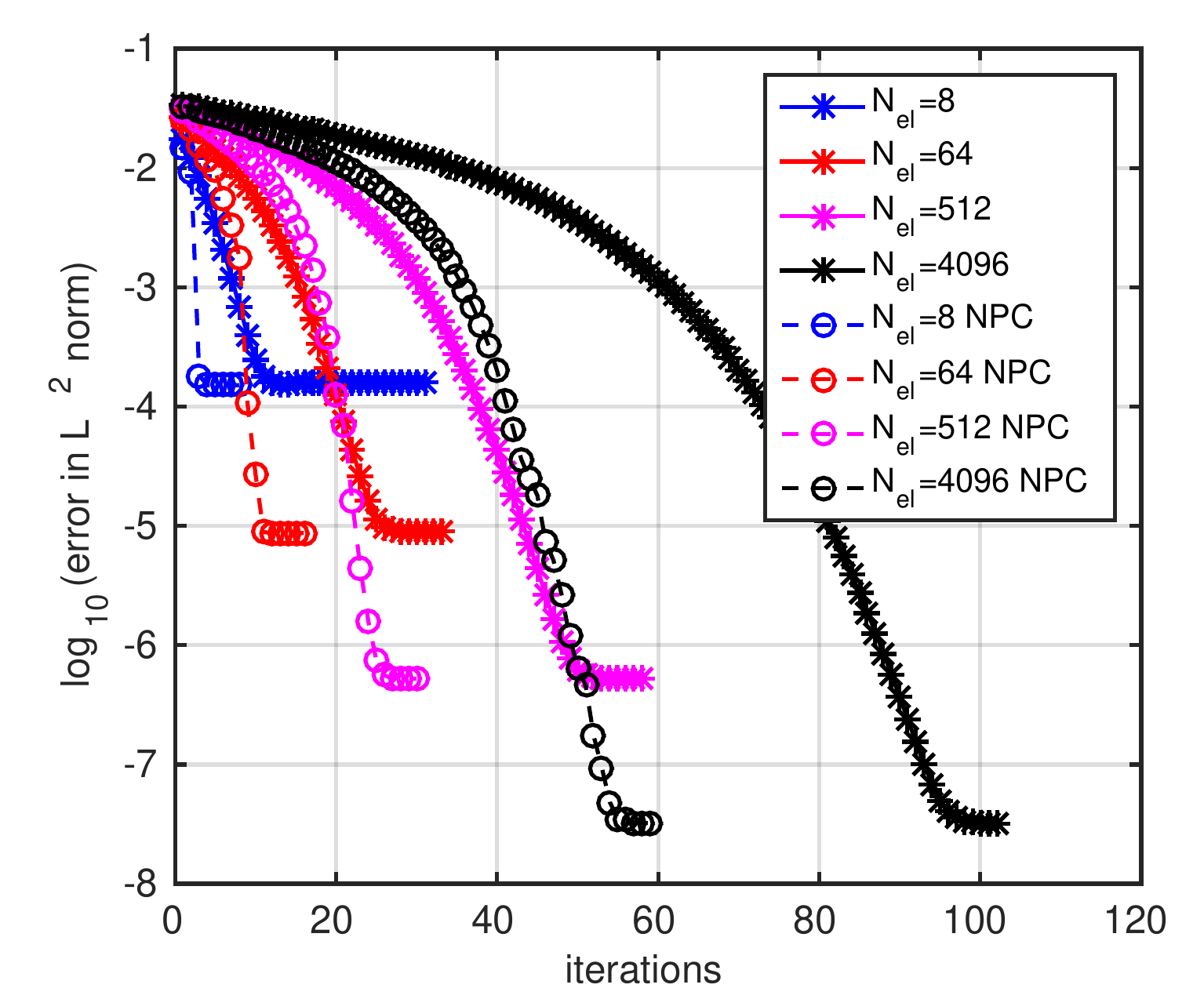}
\label{CDR_epm3_p_3}
}
\subfigure[Error history for p=4]{
\includegraphics[trim=3cm 7cm 3cm 8cm,clip=true,width=0.5\textwidth]{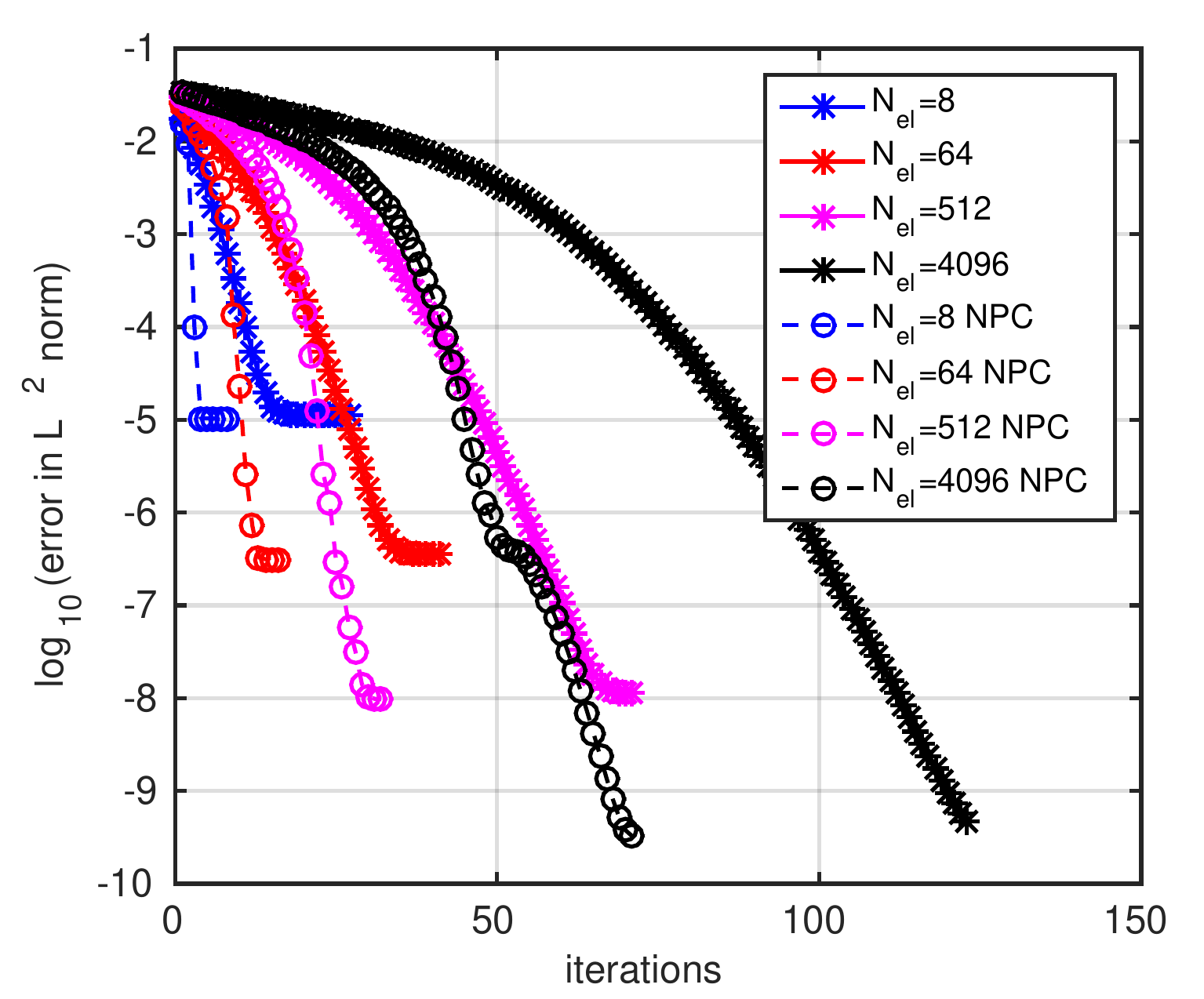}
\label{CDR_epm3_p_4}
}
\subfigure[Error history for p=3]{
\includegraphics[trim=3cm 7cm 3cm 8cm,clip=true,width=0.5\textwidth]{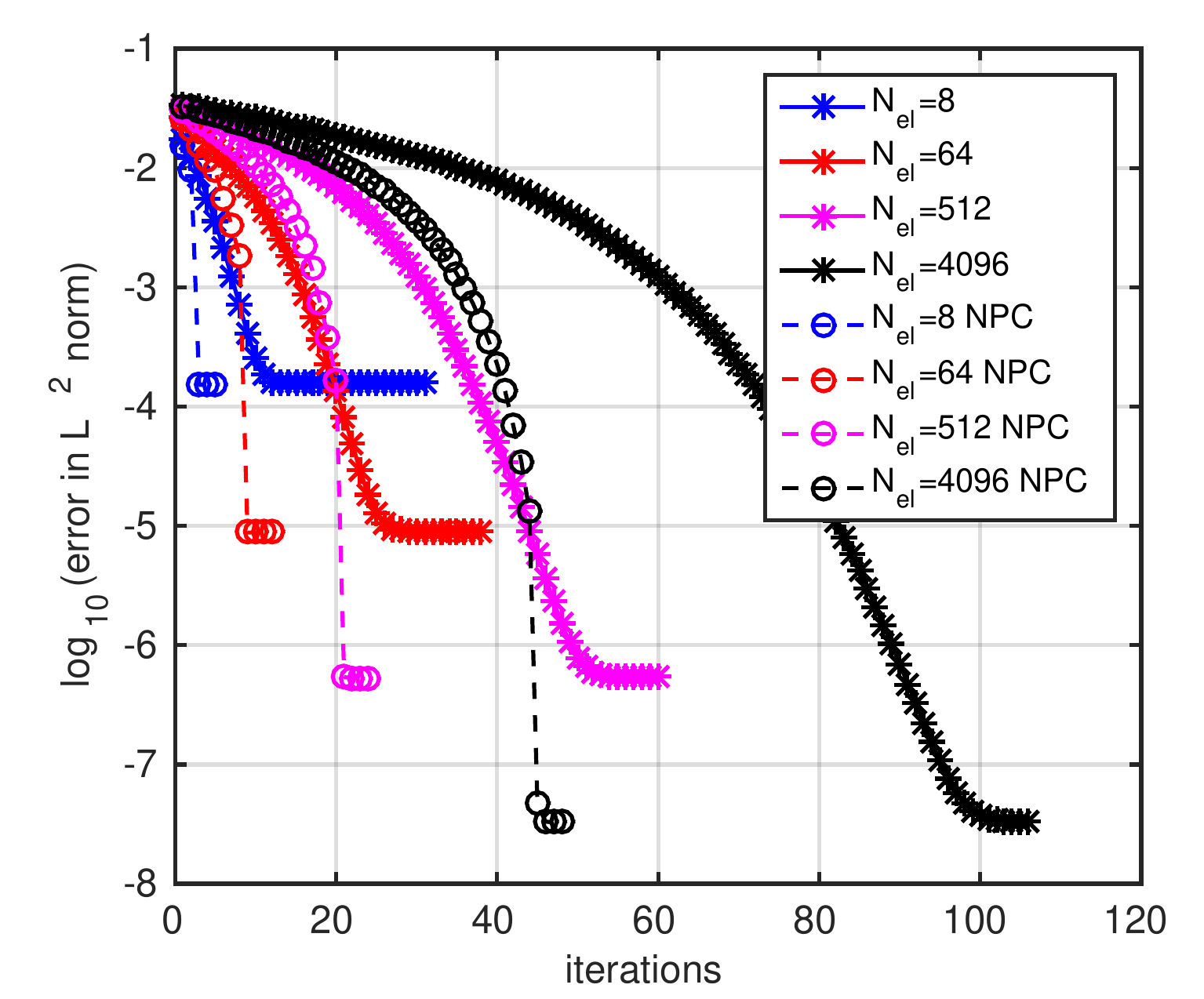}
\label{CDR_epm6_p_3}
}
\subfigure[Error history for p=4]{
\includegraphics[trim=3cm 7cm 3cm 8cm,clip=true,width=0.5\textwidth]{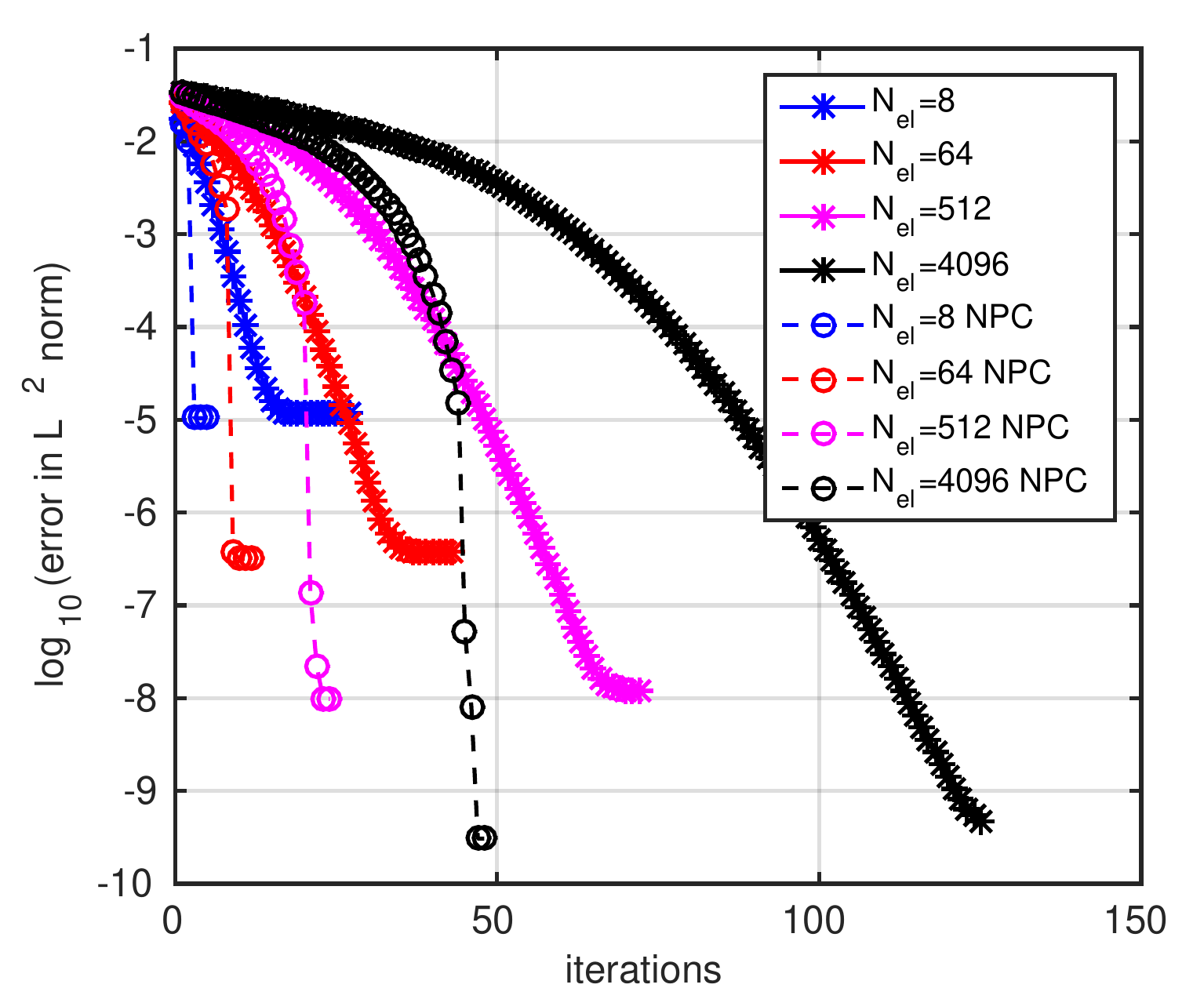}
\label{CDR_epm6_p_4}
}
\caption{Convergence history for the iHDG method with upwind and NPC fluxes for $\kappa=10^{-3}$ (top row) and $\kappa=10^{-6}$ (bottom row).}
\label{CDR_epm3_epm6_p3p4}
\vspace{-10mm}
\end{figure}

\subsubsection{Mixed (hyperbolic-elliptic) regime}
In this regime, we take $\kappa=10^{-2}$, $\nu=1$, and $\betab=(1+z,
1+x, 1+y)$. Table \ref{tab:Comparison_different_kappa} shows that both
upwind and NPC fluxes fail to converge for a number of cases (``$*$'' indicates divergence)
though the upwind iHDG is more robust. This is due to the violation of the
necessary condition $\mathcal{B}>0$ in Theorem
\theoref{ThmUpwindDDM} for finer meshes. From section
\secref{iHDG-first-order-CDR}, by choosing
$\varepsilon=\mc{O}\LRp{\frac{1}{\bar\tau}}$ for both 
upwind and NPC fluxes we can estimate the minimum mesh sizes 
and they are shown in table \ref{tab:theory_k_0.01}. Comparing the second and the third rows of
table \ref{tab:theory_k_0.01} with the first, the third and the sixth
columns of table \ref{tab:Comparison_different_kappa} we see that the
numerical results differ from the theoretical estimates by a constant.
In case of the upwind flux the constant is 2 and for the NPC flux it is 4. That is, if we multiply second and third rows of table \ref{tab:theory_k_0.01}
with 2 and 4 respectively and compare it with the numerical results in 
table \ref{tab:Comparison_different_kappa} we can see an agreement.

\begin{table}[h!b!t!]
\begin{center}
\caption{Theoretical estimates on the minimum mesh size for convergence of upwind and NPC fluxes for $\kappa=0.01$ from section \secref{iHDG-first-order-CDR}.} 
\label{tab:theory_k_0.01}
\begin{tabular}{ | r || c | c | c | c | }
\hline
{Flux} & {$p=1$} & $p=2$ & $p=3$ & $p=4$ \\
\hline
Upwind & $h>0.019$ & $h>0.0375$ & $h>0.0625$ & $h>0.094$ \\
\hline
NPC & $h>0.0225$ & $h>0.045$ & $h>0.075$ & $h>0.1125$ \\ 
\hline
\end{tabular}
\end{center}
\end{table}

\subsubsection{Diffusion regime (elliptic equation)}
\seclab{elliptic-results} As an example for the diffusion limit, we
take $\betab=0$ and $\kappa=1$.  In order to verify the conditional
convergence in section \secref{convergence-condition-convdiff}, we
choose three different values of $\nu$ in the set $\LRc{1,10,100}$.
Recall in section \secref{convergence-condition-convdiff} that
the upwind flux does not converge for $\kappa=1$ and
$\betab=0$ (pure diffusion regime). It is true for NPC flux also, due to
lack of mesh dependent stabilization. From the condtions derived in section 
\secref{convergence-condition-convdiff} we choose $\tau=\frac{(p+1)(p+2)}{h}$.
The stopping criteria is taken as in \eqnref{tolerancecriteria}.

In table \ref{tab:elliptic_nu_comparison} we compare the number of
iterations the iHDG algorithm takes to converge for $\nu=\LRc{1,10}$ cases. Note that
``$*$'' indicates that the scheme either reaches $2000$ iterations or
diverges.
The convergent condition in section
\secref{convergence-condition-convdiff} is equivalent to
$h>\mc{O}\LRp{\frac{1}{\sqrt{\lambda}}}$ and since $\sqrt{\lambda} =
\sqrt{\nu}$ we see similar convergence/divergence behavior for both
$\nu=1$ and $10$ (because the lower bound for $h$ is of the same order
for these cases). For $\nu=100$ the lower bound for $h$ is one order
of magnitude smaller, and this allows us to obtain convergence for two additional cases:
i) $N_{el}=512$ and $p=4$ with $1160$ iterations; and ii) $N_{el}=4096$ and  $p=2$
with $1450$ iterations.

\begin{table}[h!b!t!]
\begin{center}
\caption{The number of iHDG iterations for $\nu=1$ and $\nu=10$ with $\tau=\frac{(p+1)(p+2)}{h}$.} 
\label{tab:elliptic_nu_comparison}
\begin{tabular}{ | r || c | c | c | c | c | c | c | c | }
\hline
\multirow{2}{*}{$N_{el}$} & \multicolumn{4}{|c|}{$\nu=1$} & \multicolumn{4}{|c|}{$\nu=10$} \\
& $p=1$ & $p=2$ & $p=3$ & $p=4$ & $p=1$ & $p=2$ & $p=3$ & $p=4$ \\
\hline
8 & 4 & 24 & 60 & 98 & 3 & 17 & 54 & 90 \\
\hline
64 & 37 & 119 & 285 & 569 & 32 & 108 & 249 & 497 \\
\hline
512 & 158 & 527 & 1296 & * & 130 & 429 & 1124 & * \\
\hline
4096 & 1519 & * & * & * & 1178 & * & * & * \\
\hline
\end{tabular}
\end{center}
\end{table}

\begin{figure}[h!t!b!]
\subfigure[$\nu=1$]{
\includegraphics[trim=0.5cm 10cm 7cm 6cm,clip=true,width=0.5\textwidth]{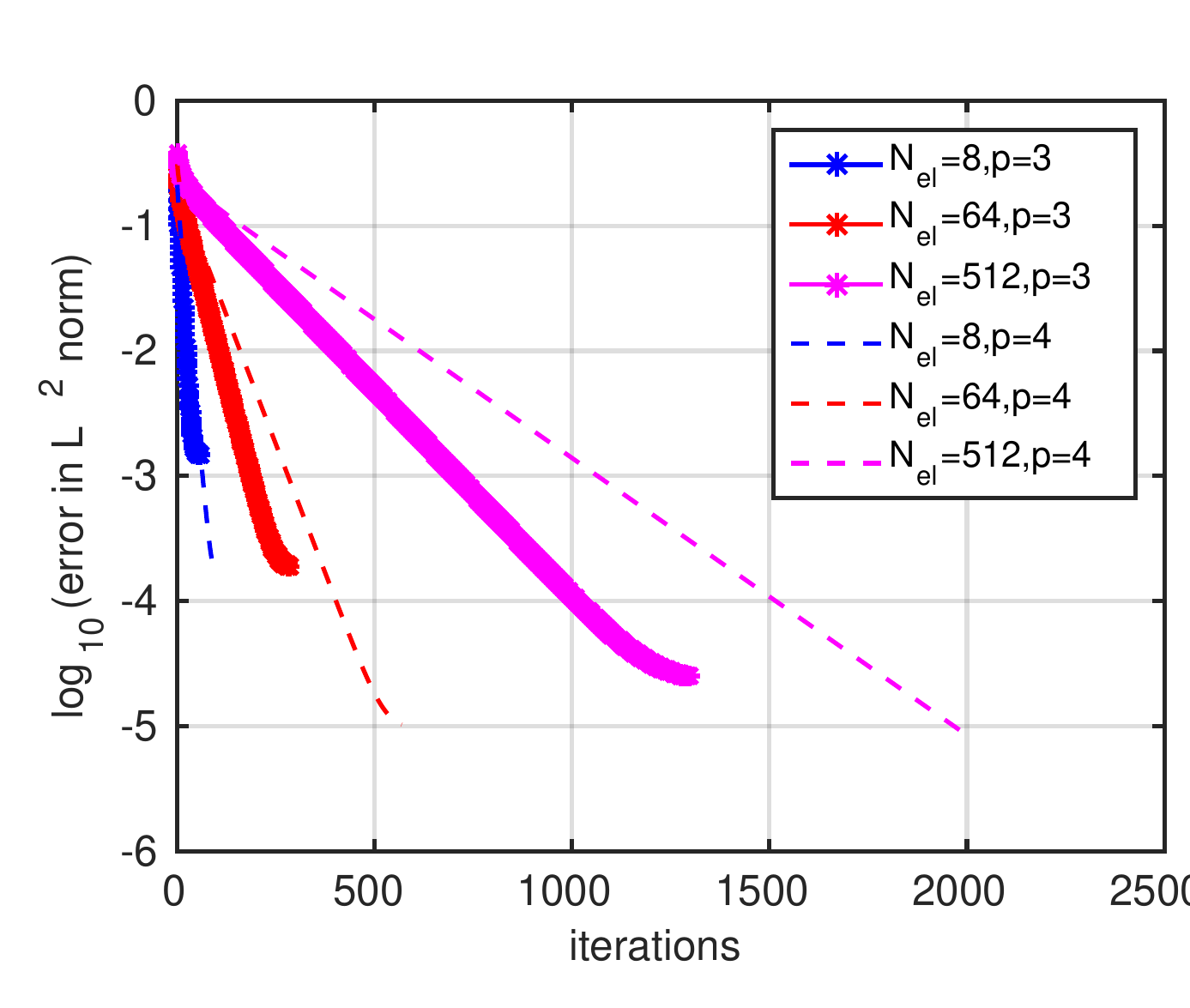}
\label{nu_1_p3p4}
}
\subfigure[$\nu=10$]{
\includegraphics[trim=0.5cm 10cm 7cm 6cm,clip=true,width=0.5\textwidth]{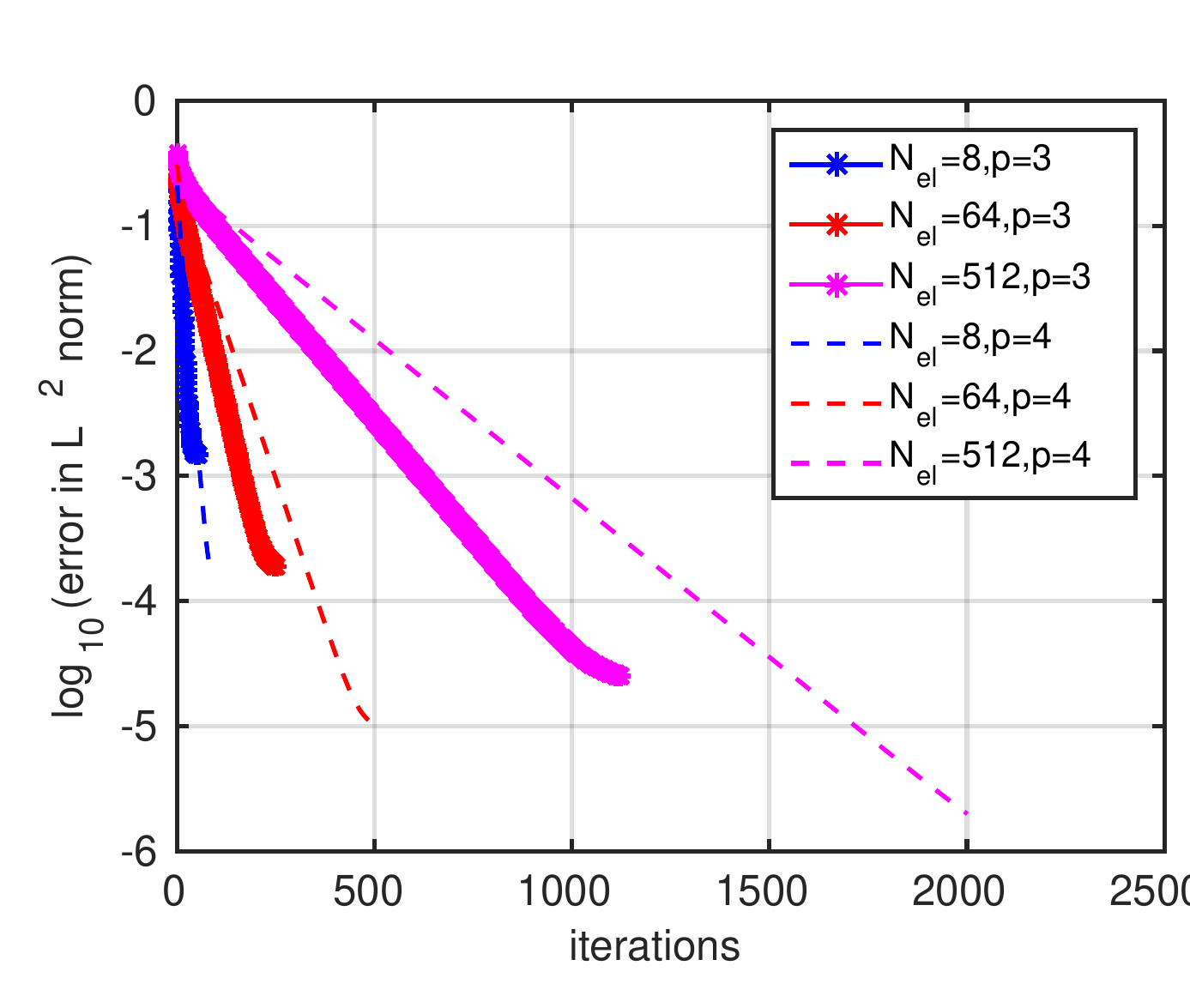}
\label{nu_10_p3p4}
}
\caption{Convergence of the iHDG algorithm for different mesh size $h$
  and solution order $p=\LRc{3,4}$ for 3D elliptic equation with
  $\nu=1$ and $\nu=10$.}
\label{nu_1_10}
\end{figure}

\begin{figure}[h!b!t!]
  \subfigure[$\sigma$ at $iteration = 1$]{
    \includegraphics[trim=4cm 0cm 3cm 0cm,clip=true,width=0.31\columnwidth]{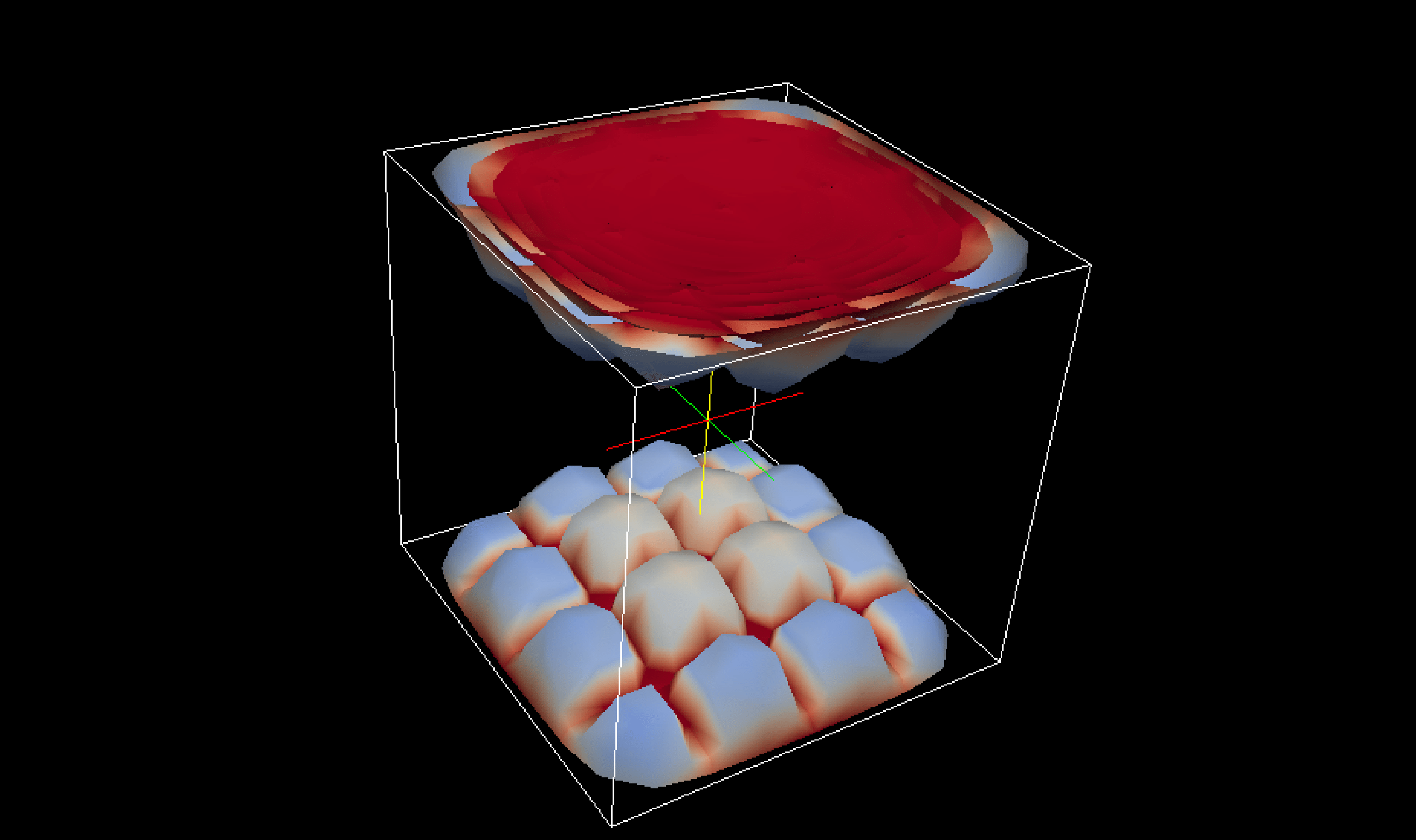}
  }
  \subfigure[$\sigma$ at $iteration = 10$]{
    \includegraphics[trim=4cm 0cm 3cm 0cm,clip=true,width=0.31\columnwidth]{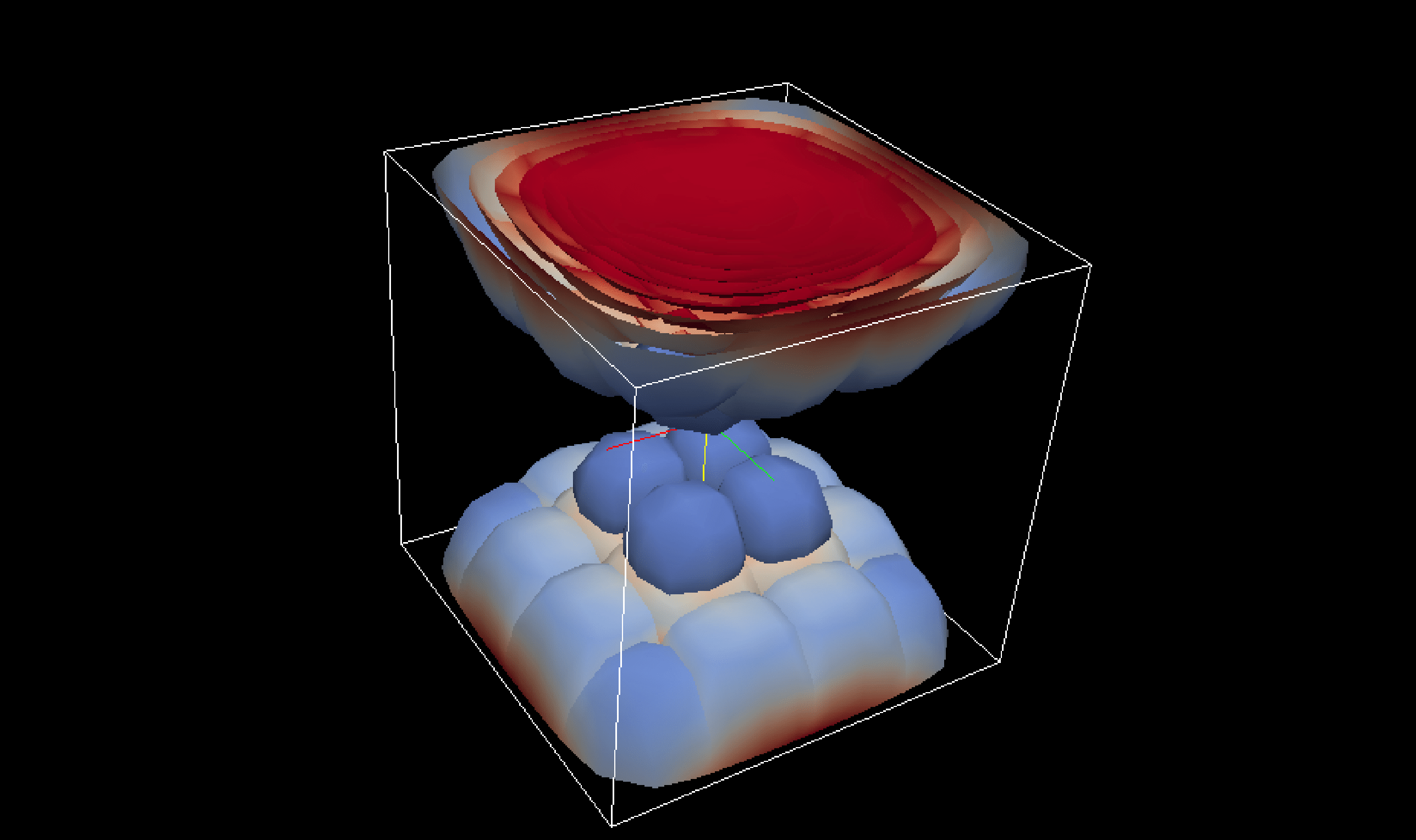}
  }
   \subfigure[$\sigma$ at $iteration = 20$]{
     \includegraphics[trim=4cm 0cm 3cm 0cm,clip=true,width=0.31\columnwidth]{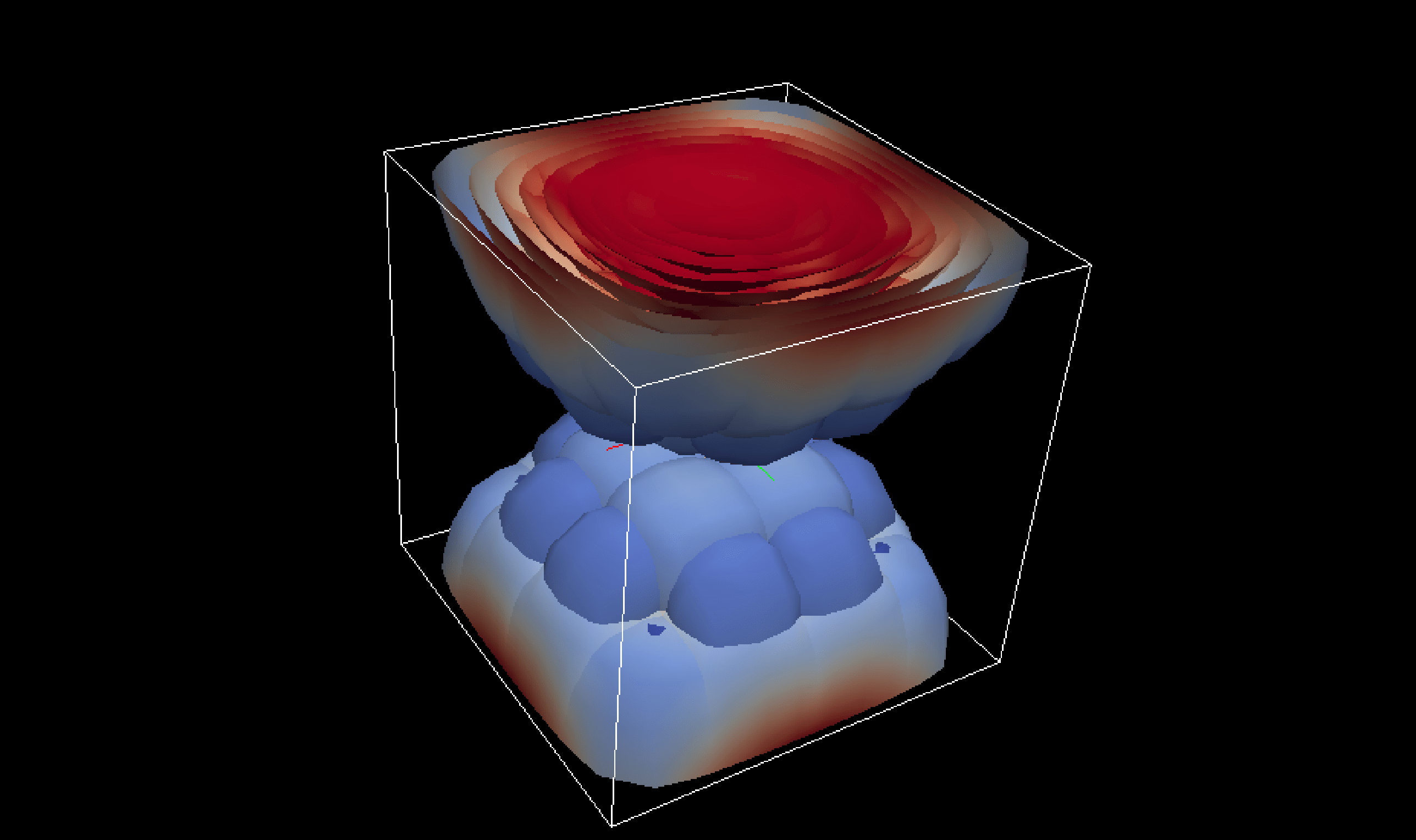}
   }
  \subfigure[$\sigma$ at $iteration = 30$]{
    \includegraphics[trim=4cm 0cm 3cm 0cm,clip=true,width=0.31\columnwidth]{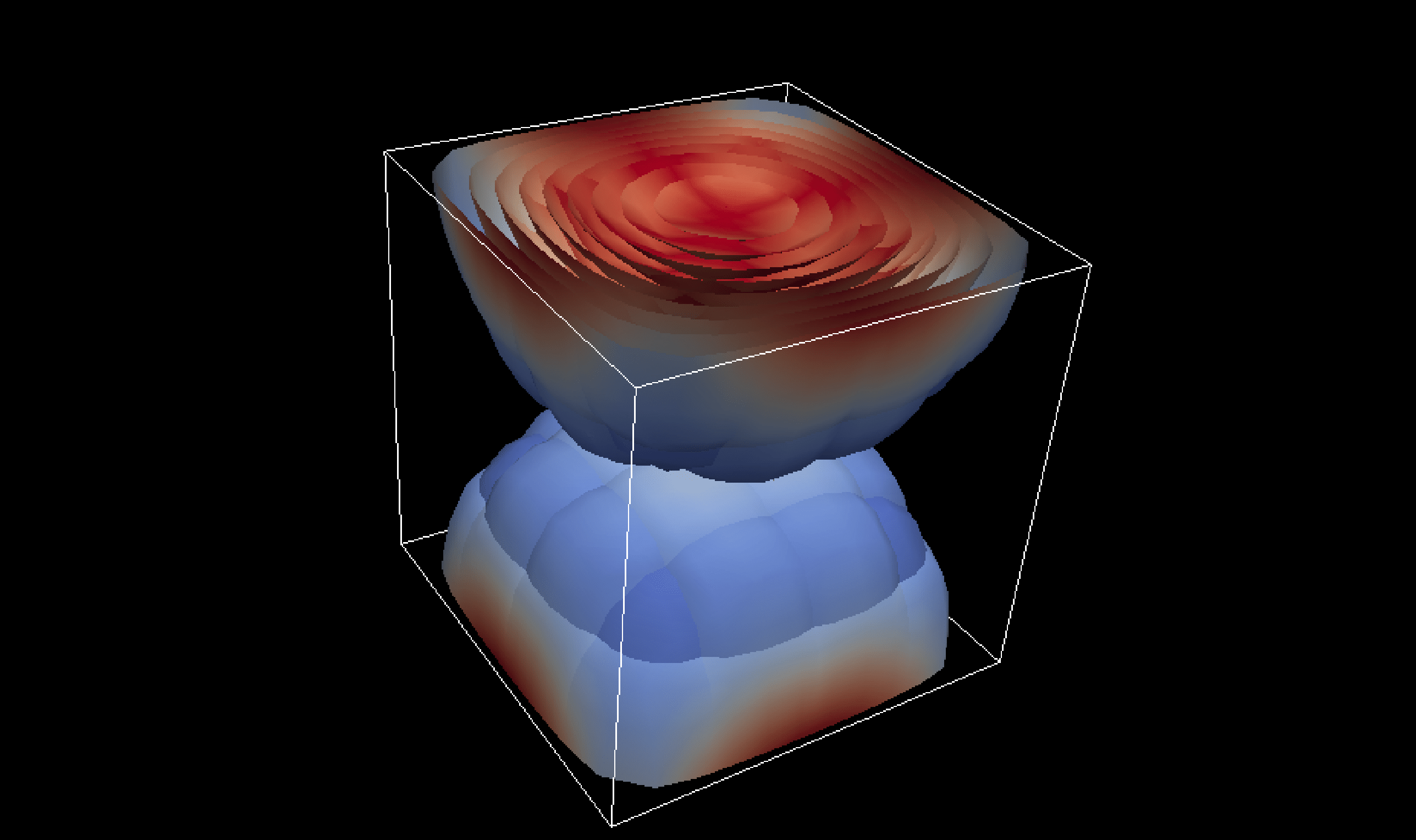}
  }
   \subfigure[$\sigma$ at $iteration = 90$]{
     \includegraphics[trim=4cm 0cm 3cm 0cm,clip=true,width=0.31\columnwidth]{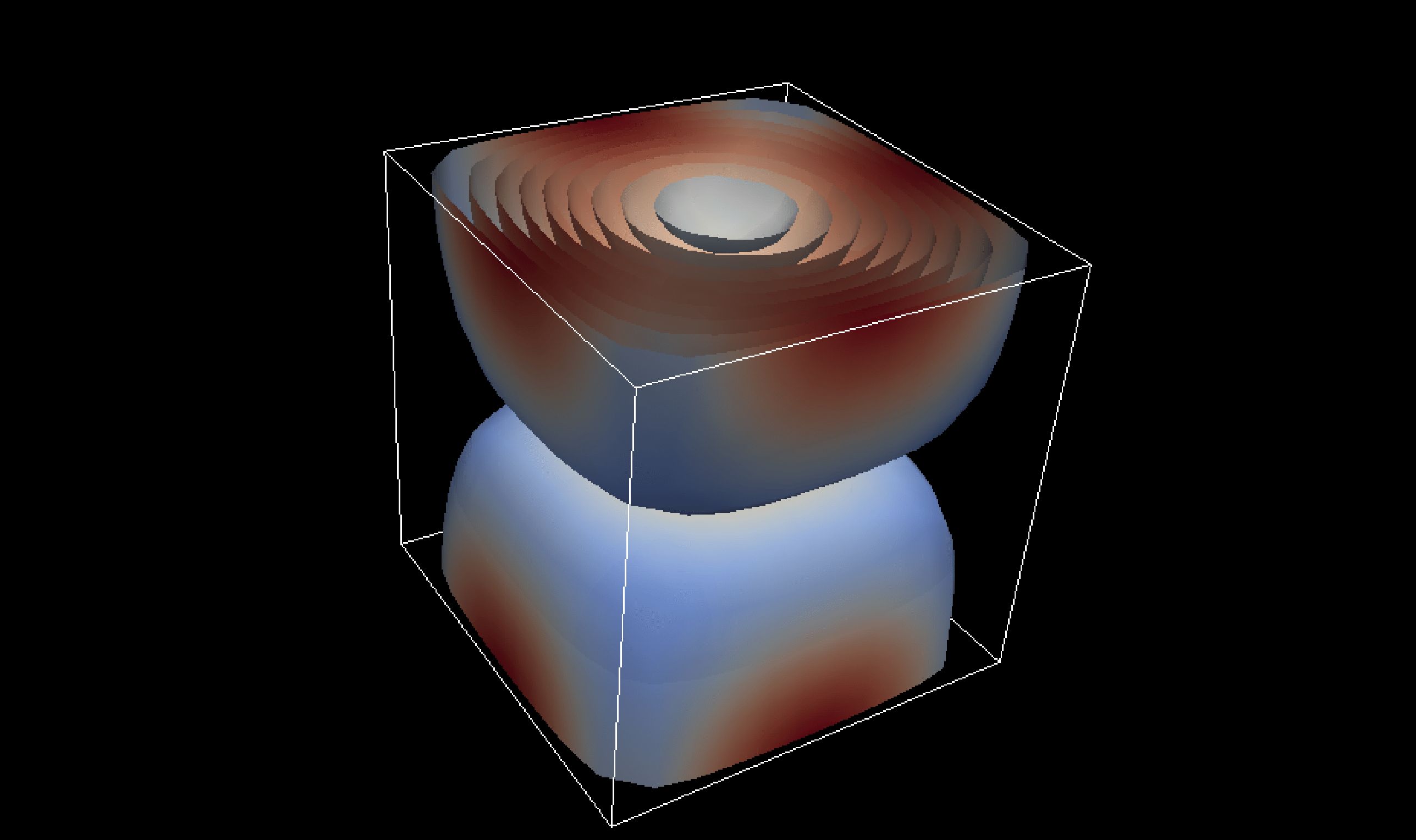}
   }
  \subfigure[$\sigma$ at $iteration = 569$]{
    \includegraphics[trim=4cm 0cm 3cm 0cm,clip=true,width=0.31\columnwidth]{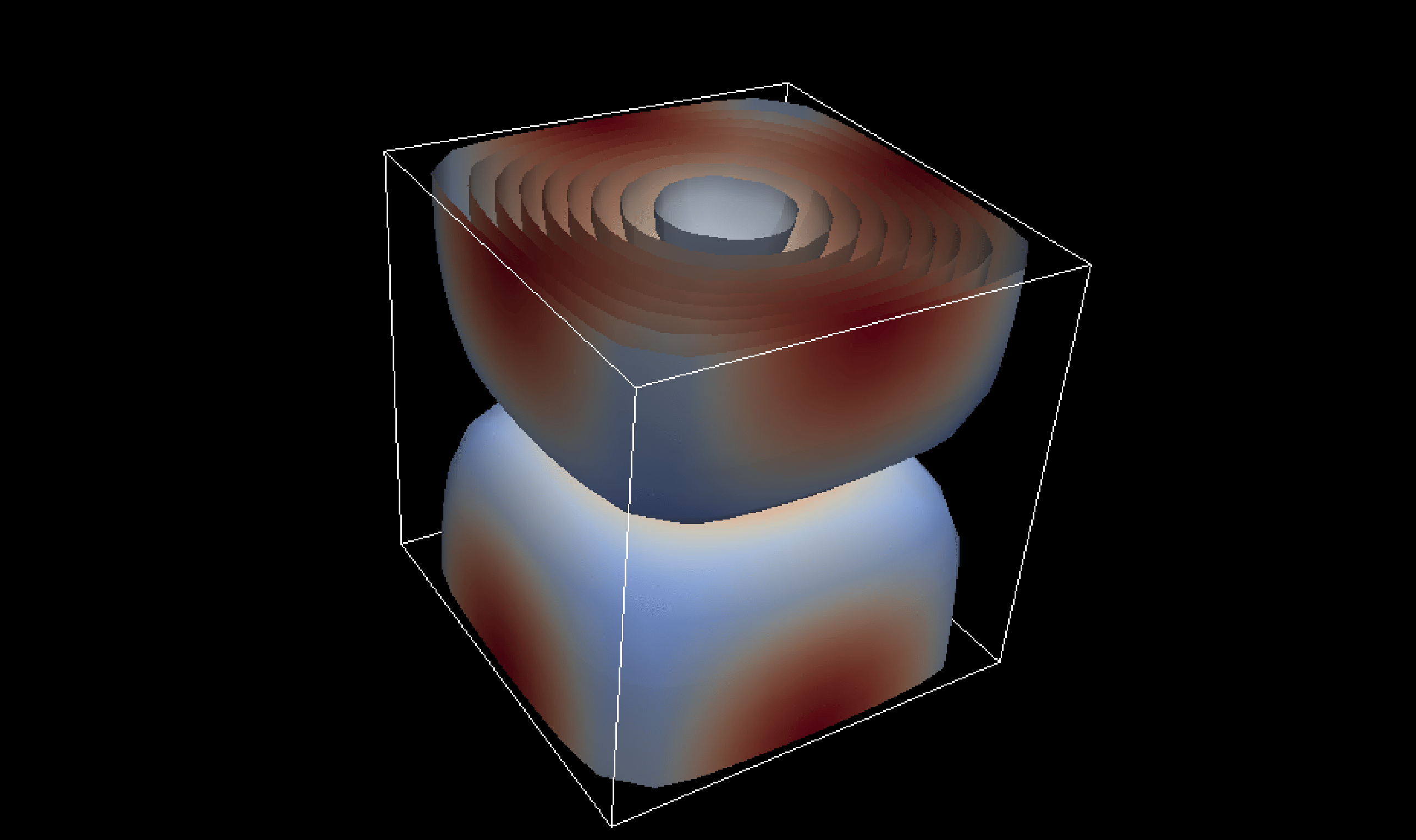}
  }
  \caption{Evolution of $\sigma$ with respect to the number of iterations for $\nu=1$.}
  \figlab{3dellipticQ}
  \vspace{-7mm}
\end{figure}

Figure \ref{nu_1_10} shows the convergence history for different
meshes and solution orders for $\nu=1$ and $\nu = 10$. We notice that the
convergence trend is different from the pure convection and
convection-diffusion cases, that is, it is exponential in the number
of iterations starting from the beginning, but the rate is less. We
also show in Figure \figref{3dellipticQ} the evolution of the magnitude of $\sigb$ ($\sigma=|\sigb|$) with
respect to the number of iterations for $N_{el}=64$ and solution
order $p=4$. Unlike the convection (or convection-dominated) case, the
convergence of the iHDG solution in this case does not have a
preferable direction as the elliptic nature of the PDEs is encoded in the iHDG algorithm via the numerical flux.

\subsection{Time dependent convection-diffusion equation}
In this section  we consider the following equation
\begin{subequations}
\eqnlab{Time-Convection-diffusion-eqn}
\begin{align}
\kappa^{-1}\sigb^e + \Grad \u^e &= 0 \quad \text{ in } \Omega, \\ 
\frac{\partial \u^e}{\partial t}+\Div\sigb^e +\betab \cdot \Grad \u^e &= 0 \quad \text{ in }
\Omega.
\end{align}
\end{subequations}
We are interested in the transport of the contaminant concentration
\cite{NguyenPeraireCockburn09a,BashirWillcoxGhattasEtAl08} with
diffusivity $\kappa=0.01$ in a three dimensional domain
$\Omega=[0,5]\times[-1.25,1.25]\times[-1.25,1.25]$ with convective
velocity field $\betab$ given as\footnote{Here, $\betab$ is an
  extension of 2D analytical solution of Kovasznay flow
  \cite{Kovasznay1948}.}
\[
	\betab=\LRp{1-e^{\gamma x}\cos(2\pi y),\frac{\gamma}{2\pi}e^{\gamma x}\sin(2\pi y),0},
\]
where $\gamma=\frac{Re}{2}-\sqrt{\frac{Re^2}{4}+4\pi^2}$ and
$Re=100$. Starting from $t = 0$, and for every second afterwards, the
same distribution of contaminant concentration of the form
\[
u_0=e^{\frac{(x-1)^2+y^2+z^2}{0.5^2}}+e^{\frac{(x-1)^2+(y-0.5)^2+z^2}{0.5^2}}+e^{\frac{(x-1)^2+(y+0.5)^2+z^2}{0.5^2}}
\]
is injected into the flow field.  A time stepsize of $\Delta t=0.025$
is selected and the simulation is run for $400$ time steps, i.e. till
$T=10$, with the Crank-Nicolson method.  Here, we use the mesh with
$N_{el}=512$ elements and solution order $p=4$.  On the left boundary,
i.e. $x=0,-1.25\leq y \leq 1.25,-1.25\leq z \leq 1.25$, the Dirichlet
boundary condition $\u=0$ is applied, while on the remaining boundary
faces, we employ the homogeneous Neumann boundary condition $\Grad \u
\cdot \n=0$. The Peclet number for this problem is $200$. This
exhibits a wide range of mixed hyperbolic and parabolic regimes and
hence is a good test bed for the proposed iHDG scheme. In this
example, the iHDG algorithm with NPC flux does not converge even for
 coarse meshes and low solution orders (this can be seen in
section \secref{Convection-diffusion-steady}). We therefore show
numerical results only for the upwind HDG flux in Figure
\figref{3dcontaminant}.  The scheme requires approximately $46$
iterations for each time step to reach the stopping criteria
$\norm{u^k-u^{k-1}}_{L^2}<10^{-6}$.  Since $\Grad\cdot\betab=0$, $\nu=0$ and
$\lambda=1/\Delta t$ (see remark \ref{time-conv-diff})  we obtain from  section
\secref{iHDG-first-order-CDR} the estimate for the time stepsize: $\Delta
t=\mc{O}\LRp{\frac{h}{(p+1)(p+2)}}$. Using this estimate we
compare the number of iterations the iHDG algorithm takes to converge for different meshes and solution orders in table \ref{tab:Comparison_contaminant_transport}. We do not obtain convergence for $N_{el}=4096$ and solution orders equal to $3$ and $4$: the main reason is that this problem is in the mixed hyperbolic and parabolic regime and the above setting does not satisfy the condition for convergence (see Section
\secref{Convection-diffusion-steady}).

\begin{figure}[h!b!t!]
  \subfigure[$\u$ at $t=0$]{
    \includegraphics[trim=6cm 6cm 7cm 6cm,clip=true,width=0.48\columnwidth]{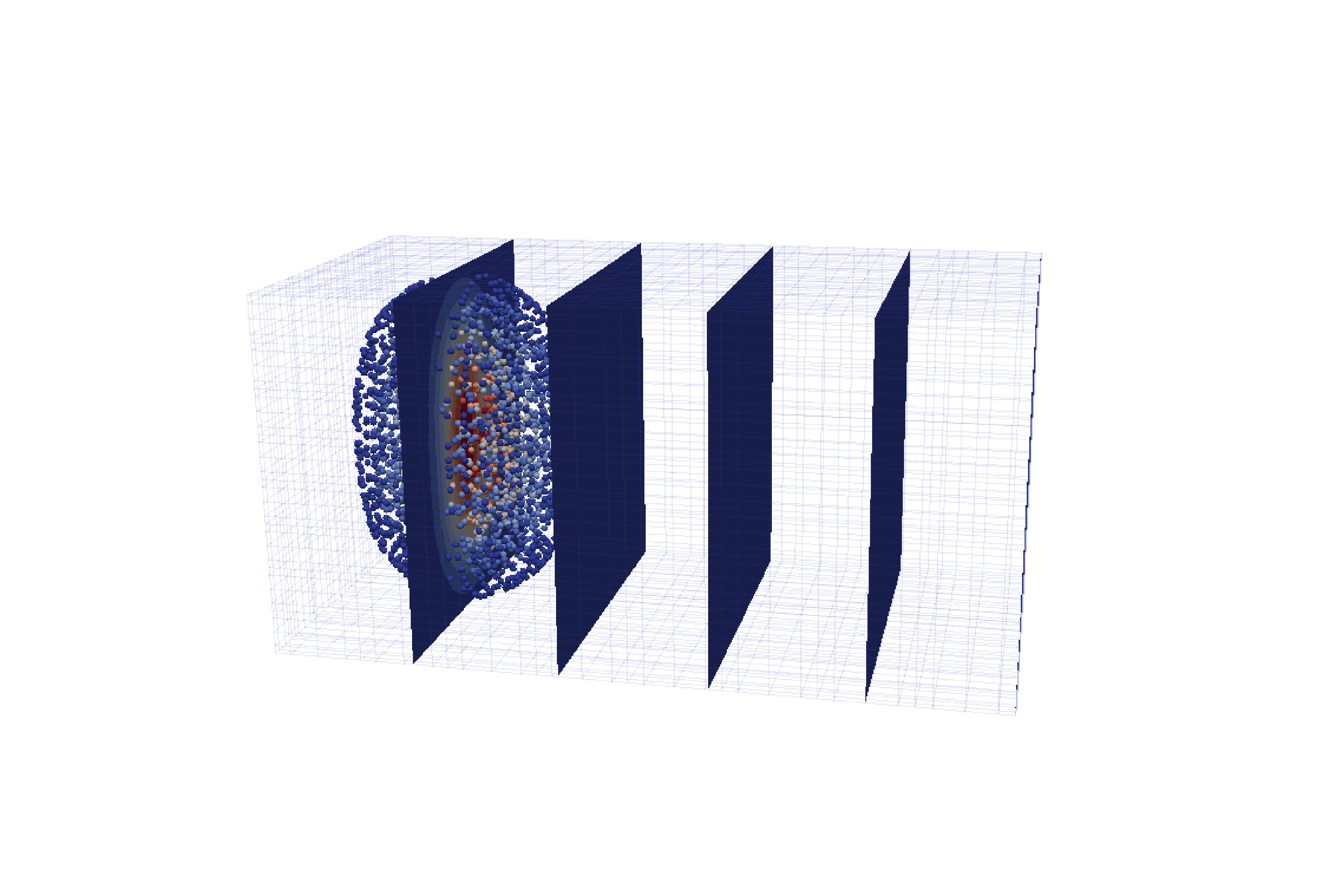}
  }
  \subfigure[$\u$ at $t=0.75$]{
    \includegraphics[trim=6cm 6cm 7cm 6cm,clip=true,width=0.48\columnwidth]{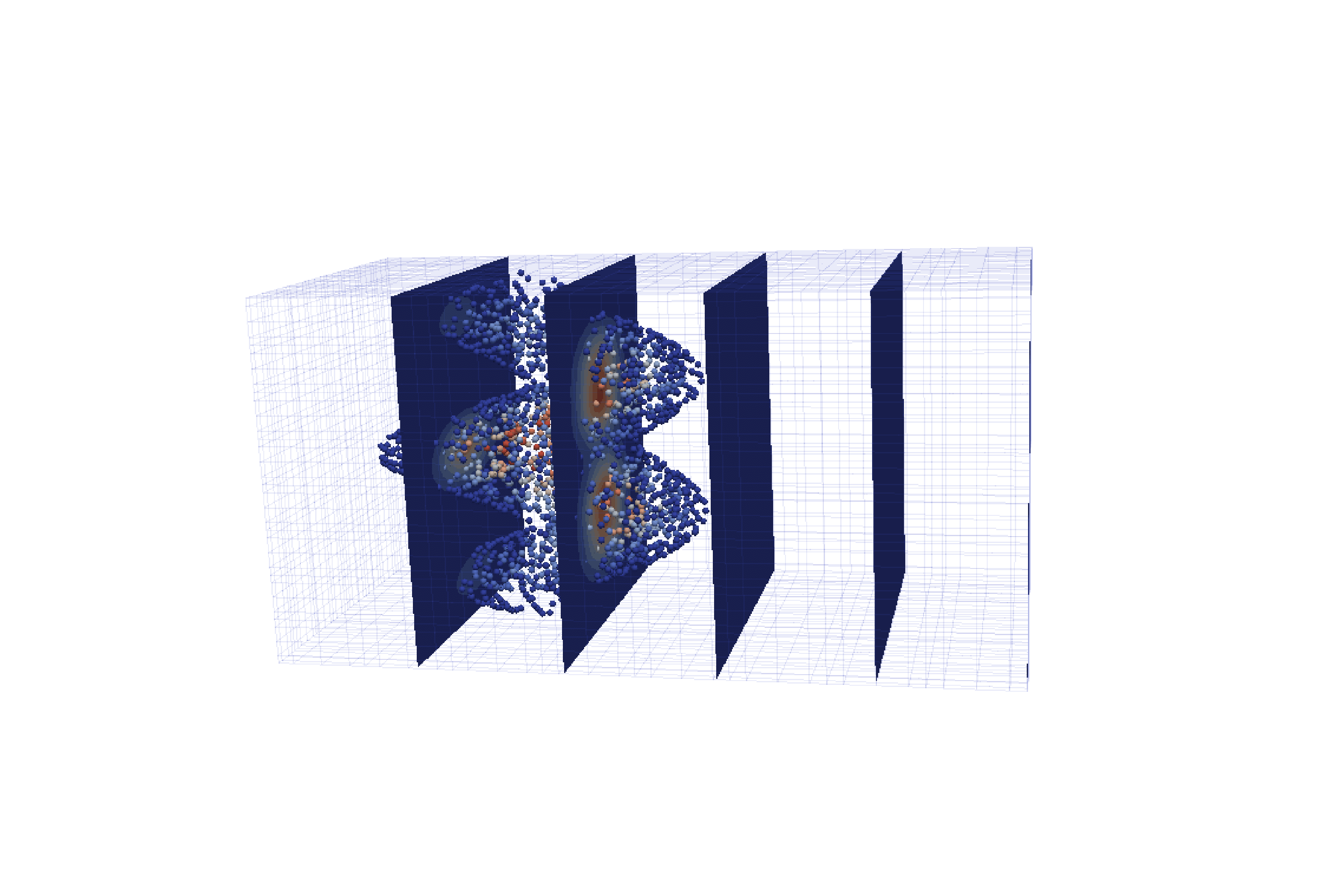}
  }
  \subfigure[$\u$ at $t=2.0$]{
    \includegraphics[trim=6cm 6cm 7cm 6cm,clip=true,width=0.48\columnwidth]{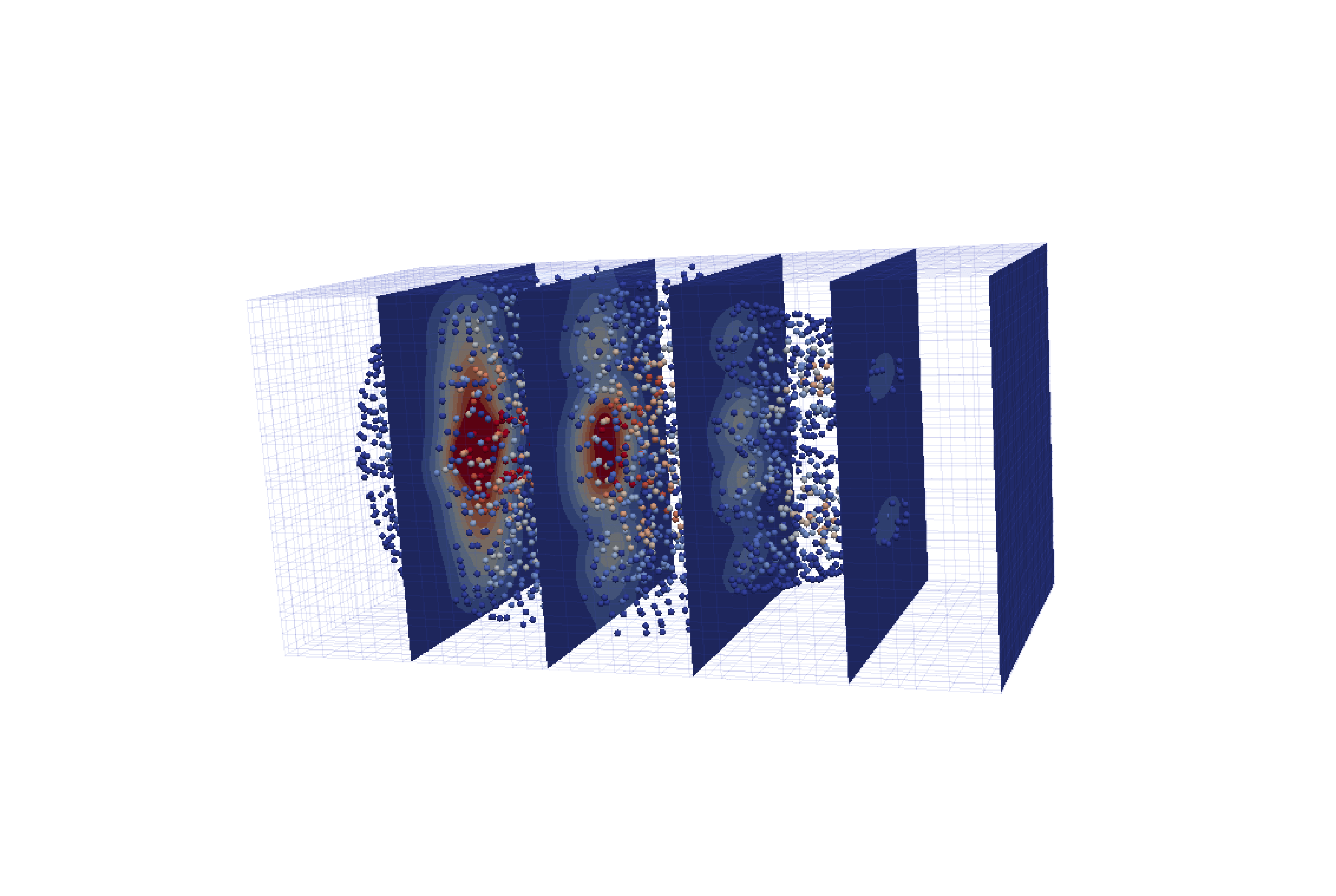}
  }
  \subfigure[$\u$ at $t=3.0$]{
    \includegraphics[trim=6cm 6cm 7cm 6cm,clip=true,width=0.48\columnwidth]{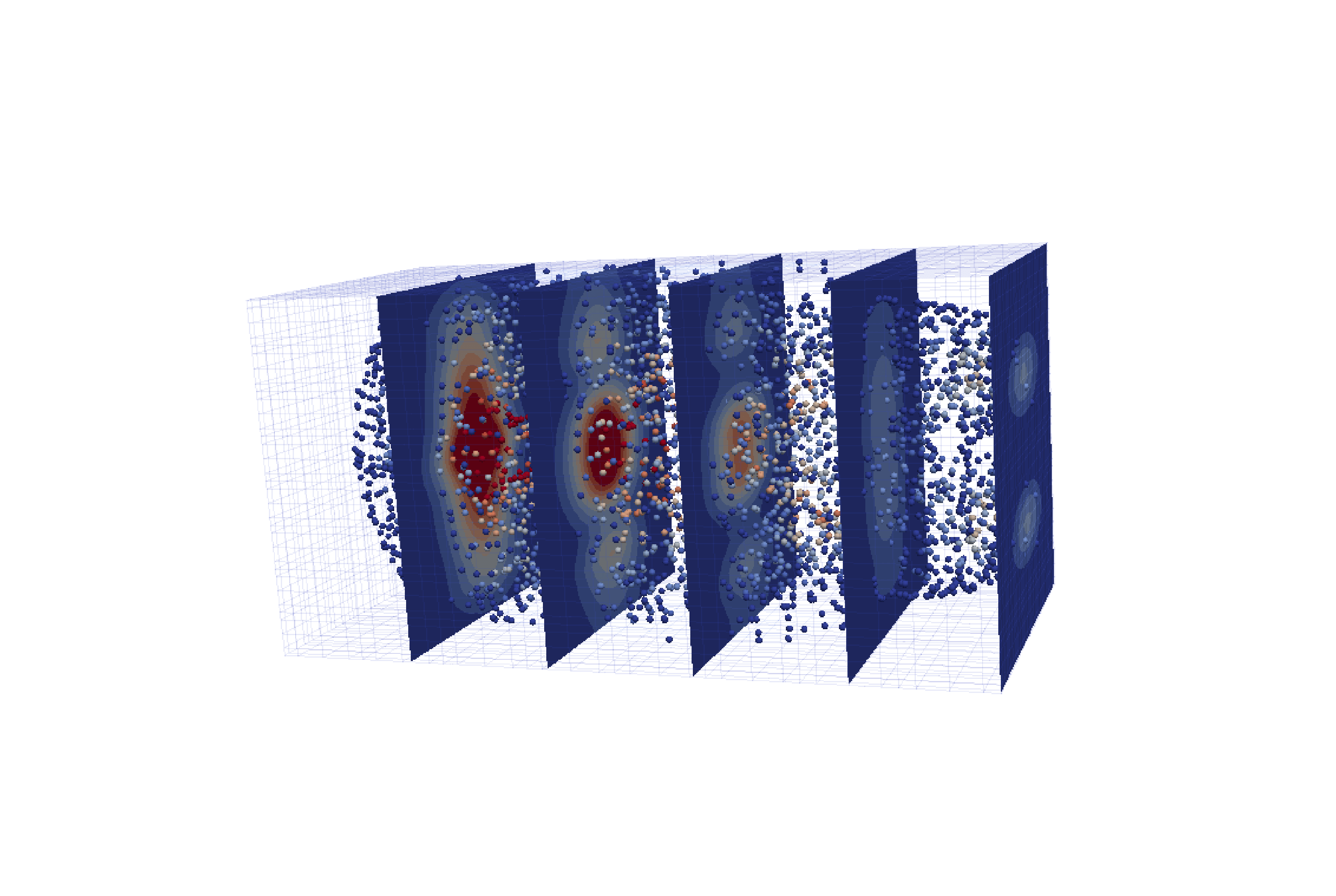}
  }
  \caption{Solution $\u$ as a function of time using the iHDG
    algorithm with the upwind flux for the contaminant transport
    problem \eqnref{Time-Convection-diffusion-eqn}.}
  \figlab{3dcontaminant}
\end{figure}

\begin{table}[h!b!t!]
\begin{center}
\caption{The number of iHDG iterations per time step for the
  contaminant transport problem 
  with various solution orders and mesh sizes.}
\label{tab:Comparison_contaminant_transport}
\begin{tabular}{ | r || c | c | c | c | }
\hline
{$N_{el}$} & {$p=1$} & $p=2$ & $p=3$ & $p=4$ \\
\hline
8 & 12 & 24 & 23 & 22 \\
\hline
64 & 26 & 26 & 22 & 20 \\ 
\hline
512 & 21 & 23 & 17 & 37 \\
\hline
4096 & 17 & 18 & * & * \\
\hline
\end{tabular}
\end{center}
\end{table}

\section{Conclusions and Future Work}
\seclab{conclusions}

We have presented an iterative solver, namely iHDG, for HDG
discretizations of linear PDEs. The method exploits the structure of
HDG discretization and idea from domain decomposition methods. One of
the key features of the iHDG algorithm is that it requires only local
solves, i.e. element-by-element and face-by-face, completely
independent of each other, during each iteration. It can also be
considered as a block Gauss-Seidel method for the augmented HDG system
with volume and weighted trace unknowns. Thanks to the built-in
stabilization via the weighted trace and the structure of the HDG discretization, unlike traditional Gauss-Seidel schemes, the convergence
of iHDG is independent of the ordering of the unknowns.  Using an energy
approach we rigorously derive the conditions under which the iHDG
algorithm is convergent for the transport equation, the linearized
shallow water equation, and the convection-diffusion equation.  In
particular, for the scalar transport equation, the algorithm is
convergent for all meshes and solution orders, and the convergence
rate is independent of solution order. This feature makes the iHDG
solver especially suitable for high-order DG methods, that is,
high-order (and hence more accurate) solutions do not require more
number of iterations. The scheme in fact performs an implicit marching
and the solution converges in patches of elements automatically from the inflow to
the outflow boundaries.  For the linearized shallow water equation, we
prove that the convergence is conditional on the mesh size and the
solution order. Similar conditional convergence is also shown for the
convection-diffusion equation in the first order form. We have studied
the performance of the scheme in convection dominated, mixed
(hyperbolic-elliptic) and diffusion regimes and numerically verified
our theoretical results.
Ongoing work is to improve the convergence of the iHDG scheme for diffusion-dominated
problems. We are also developing a two grid preconditioner for elliptic problems with iHDG as a fine scale 
preconditioner. A different avenue for the future work is to explore the opportunity to employ/study
iHDG as a scalable smoother for multigrid solvers.

\section*{Acknowledgements}
The authors would like to thank Dr. Jeonghun J. Lee for fruitful discussions. The research was initiated when M.-B. Tran visited The University of Texas at Austin, and he would like to thank the institution for the hospitality. We are indebted to Dr. Hari Sundar for sharing his high-order
finite element library {\bf homg}, on which we have implemented the
iHDG algorithms and produced numerical results. We also thank Dr. Vishwas Rao for careful proofreading of the paper.
Finally we are indebted to anonymous referees for their critical and useful comments that improved this paper substantially.

\bibliography{references,ceo}

\end{document}